\documentclass[11pt]{article}
\usepackage{amsmath}
\usepackage{amssymb}
\usepackage{amsbsy}
\usepackage{amsthm}
\usepackage{epsfig}
\usepackage{wrapfig}
\usepackage{color,cancel}
\usepackage{fullpage}
\usepackage{float}
\usepackage{ulem}
\usepackage{breqn}
\usepackage{graphicx}
\usepackage{caption}
\usepackage{subcaption}
\usepackage{bbm}
\usepackage{bigints}
\usepackage{mathtools}
\usepackage{verbatim}

\usepackage{url} 
\usepackage{amsmath,amsfonts,amsthm,amssymb,epsfig}
\usepackage{graphicx}

\usepackage{amsmath,amsfonts,amssymb,color,graphicx}

\newtheorem{theorem}{Theorem}[section]

\newtheorem{corollary}{Corollary}[section]
\newtheorem{lemma}{Lemma}[section]

\newtheorem{remark}{Remark}[section]

\numberwithin{equation}{section}
\newcommand{\norm}[1]{\left\Vert#1\right\Vert}
\newcommand{\eps}{\varepsilon}

\usepackage{calrsfs}
\newcommand{\bH}{{\bf H}}
\newcommand{\bu}{{\bf u}}
\newcommand{\bv}{{\bf v}}
\newcommand{\bw}{{\bf w}}
\newcommand{\bx}{{\bf x}}

\newcommand{\bff}{{\bf f}}
\newcommand{\bfg}{{\bf g}}
\newcommand{\bphi}{{\boldsymbol \phi}}
\newcommand{\bfX}{{\bf X}}\newcommand{\bfH}{{\bf H}}
\newcommand{\bfV}{{\bf V}}

\newcommand{\bfeta}{{\boldsymbol \eta}}

\usepackage{color}

\title{Stokes with variable viscosity}
\usepackage[section]{placeins}
\makeatletter
\AtBeginDocument{%
	\expandafter\renewcommand\expandafter\subsection\expandafter{%
		\expandafter\@fb@secFB\subsection
	}%
}

\date{}

\title{A Family of Second Order Time Stepping Methods for the Darcy-Brinkman Equations}

\author{
Aytekin  \c{C}{\i}b{\i}k
\thanks{Department of Mathematics, Gazi University, 06500 Ankara, Turkey; abayram@gazi.edu.tr}
\and
Medine Demir
\thanks{Department of Mathematics, Middle East Technical University, 06800 Ankara, Turkey; dmedine@metu.edu.tr}
\and
Songul Kaya
\thanks{
Department of Mathematics, Middle East Technical University, 06800 Ankara, Turkey; smerdan@metu.edu.tr}
}

\begin{document}
\maketitle

\begin{abstract}This study presents an efficient, accurate, effective and unconditionally stable  time stepping scheme for the Darcy-Brinkman equations in double-diffusive convection. The stabilization within the proposed method uses the idea of stabilizing the curvature for velocity, temperature and concentration equations. Accuracy in time is proven and the convergence results for the fully discrete solutions of problem variables are given. Several numerical examples including a convergence study are provided that support the derived theoretical results and demonstrate the efficiency and the accuracy of the method.
\end{abstract}


\section{Introduction}

Double-diffusive convection is a mechanism, in which the fluid motion occurs due to buoyancy arising from the combination of temperature and concentration gradients. It is related with an increasing number of fields such as, metallurgy, oceanography, contaminant transport, petroleum drilling etc.,\cite{BS89,NT15,SG14, XPT}. The accurate and efficient numerical solutions of these flows are known to be the core of many applications. The main objective of this paper is to propose, analyze and test a family of second order time stepping methods for the Darcy-Brinkman system, by extending an earlier study of       \cite{JMR} for the Navier-Stokes equations (NSE) based on the pioneering work of \cite{C14,C16}. The paper's underlying ideas are to incorporate linearizations and stabilization terms such that the discrete curvature solutions in velocity, temperature, concentration and pressure are proportional to this combination.

Under the assumption of Boussinesq approximation, the governing equations of double-diffusive convection  is given by the Darcy-Brinkman system (see \cite{goy}),
\begin{eqnarray}\label{bous}
 \begin {array}{rcll}
\bu_t -\nu \Delta \bu+ (\bu\cdot\nabla)\bu +Da^{-1}\bu+ \nabla p &=& (\beta_T T + \beta_S S)\bfg  &
 \mathrm{in }\ (0,t]\times \Omega, \\
  \nabla \cdot \bu&=& 0& \mathrm{in }\ (0,t]\times  \Omega,\\
\bu&=& \mathbf{0}& \mathrm{in }\ (0,t]\times\partial \Omega,\\
 T_t-\gamma \Delta T+\bu\cdot \nabla T&=&0 & \mathrm{in }\ (0,t]\times\partial \Omega,\\
 S_t-{D_c}\Delta S+\bu\cdot \nabla S&=&0& \mathrm{in }\ (0,t]\times\partial \Omega,\\
T,S&=&0 &\mathrm{on} \ \Gamma_D,\\
\dfrac{\partial T}{\partial n}&=&0, \ \dfrac{\partial S}{\partial n}=0& \mathrm{on} \ \Gamma_N,\\
\bu(0,\bx) & = & \bu_0,\ T(0,\bx)  =  T_0 , \ S(0,\bx) = S_0 & \mathrm{in }\ \Omega.
 \end{array}
\end{eqnarray}
Here $\bu$ is the fluid velocity, $\bu_0$, the initial velocity,  $p$
the pressure, $T$ the temperature, $T_0$, the initial temperature, $S$ the concentration, $S_0$, the initial concentration. We also have the kinematic viscosity $\nu >0$, the Darcy number $Da$, the thermal diffusivity $\gamma > 0$, the mass diffusivity ${D_c} > 0$,the gravitational acceleration vector $\bfg$ and the thermal and solutal expansion coefficients are $\beta_T$, $\beta_S$, respectively.
The dimensionless parameters are the buoyancy ratio $N=\dfrac{\beta_S \Delta S}{\beta_T \Delta T}$, the Schmidt number $Sc=\dfrac{\nu}{D_c}$, Prandtl number $Pr=\dfrac{\nu}{\gamma}$, the Darcy number $Da=\dfrac{K}{H^2}$, the Lewis number $Le=\dfrac{Sc}{Pr}$ and the thermal Rayleigh number $Ra=\dfrac{\bfg \beta_T \Delta T H^3}{\nu \gamma }$. $\Gamma_D$ be a regular open subset and $\Gamma_D=\partial \Omega \setminus \Gamma_N $. Here cavity height $H$, a permeability $K$, $\Delta T$ and $\Delta S$ are the temperature and concentration differences, respectively.

The common solution approach for the numerical solution of the time dependent multiphysics problems requires their discretization in space and time as well as linearization. There have been a considerable number of different studies devoted to such discretizations, see e.g.\cite{GS00}. Moreover, since  linear extrapolation schemes require the solution of only one linear system per time step, they are preferable when it is compared with the fully implicit schemes. The Crank Nicholson with linear extrapolation (CNLE) of \cite{B76} has been considered in many recent studies such as \cite{I13} and two step backward differentiation formula BDF2 with linear extrapolation (BDF2LE) for different flow problems \cite{AKR,ZXZ}. The numerical analysis of the BDF2LE time stepping scheme for natural convection equations in the semi discrete and fully discrete case has been carried out by \cite{R12}, and three step backward differentiation for the  double-diffusive convection with Soret effect is considered in \cite{R15}.

In this study, we focus on an accurate regularization for a family of second order time stepping methods for the Darcy-Brinkman system through underlying ideas of \cite{JMR}. As it is shown below, a new family of the time stepping methods includes well-known CNLE and BDF2LE schemes obtained by appropriate choices of the parameters. The method (\ref{forvel})-(\ref{forcont}) can be thought between CNLE and BDF2LE. A successful stabilization method is achieved by using the idea of `curvature stabilization'. As noted in \cite{JMR}, a family of the method based on curvature stabilization leads to a sufficient stabilization along with the optimal accuracy in time. Hence, it is a natural and important next step to extend this methodology to flows governed by the system (\ref{bous}).

The paper is organized as follows. In Section 2, the numerical scheme of the second order time stepping method for (\ref{bous}) is described. In Section 3, some mathematical preliminaries are presented which are useful in the analysis. We give comprehensive stability analysis and a priori error analysis of the method in Section 4. Section 5 presents numerical illustrations that verify the analytical results. The paper ends with conclusion remarks.

\section{A family of second order time stepping schemes }

In this section, a family of second order IMEX time stepping methods for (\ref{bous}) is presented in detail. For this purpose, let partition the time interval $[0,t]$ into $N$ sub intervals with time step size $\Delta t=t/N$ and $t_{n+1}=(n+1)\Delta t$ with $n=0,1,2,...,N-1$. For simplicity, we consider the constant step sizes $\Delta t=t_{n}-t_{n-1}$ and  the quantities at time level $t_n$ are denoted by a subscript $n$. A general time step form of (\ref{bous}) reads as
\begin{gather}
\frac{(\theta+\frac{1}{2})\bu_{n+1}-2\theta \bu_n+(\theta-\frac {1}{2})\bu_{n-1}}{\Delta t} -\theta (\nu+\eps)\Delta \bu_{n+1}-(\nu-\theta(\nu+2\eps))\Delta \bu_n-\theta \eps\Delta \bu_{n-1} \nonumber
\\
+((\theta+1)\bu_n-\theta \bu_{n-1})\cdot \nabla \Big(\theta \frac{\nu+\eps}{\nu}\bu_{n+1}+\Big (1-\theta \frac{\nu+2\eps}{\nu}\bu_{n}\Big)+\theta\frac{\eps}{\nu}\bu_{n-1}\Big)  \nonumber
\\
+Da^{-1}\Big(\theta \frac{\nu+\eps}{\nu}\bu_{n+1}+\Big (1-\theta \frac{\nu+2\eps}{\nu}\bu_{n}\Big)+\theta\frac{\eps}{\nu}\bu_{n-1}\Big) +\theta \frac{\nu+\epsilon}{\nu}\nabla p_{n+1}\nonumber
\\
+(1-\theta \frac{\nu+2\eps}{\nu})\nabla p_n+\theta \frac{\epsilon}{\nu}\nabla p_{n-1}=\Big(\beta_T((\theta+1)T_n-\theta T_{n-1})+ \beta_S((\theta+1)S_n-\theta S_{n-1}))\Big)\bfg, \label{forvel}
\\
\nabla \cdot \bu_{n+1}=0 , \label{conteq}
\\
\frac{(\theta+\frac{1}{2})T_{n+1}-2\theta T_n+(\theta-\frac {1}{2})T_{n-1}}{\Delta t} -\theta (\gamma+\eps_1)\Delta T_{n+1}-(\gamma-\theta(\gamma+2\eps_1))\Delta T_n-\theta \eps\Delta T_{n-1} \nonumber
\\
+((\theta+1)\bu_n-\theta \bu_{n-1})\cdot \nabla \Big(\theta \frac{\gamma+\eps_1}{\gamma}T_{n+1}+\Big (1-\theta \frac{\gamma+2\eps_1}{\gamma}T_{n}\Big)+\theta\frac{\eps_1}{\gamma}T_{n-1}\Big)
=0, \label{fortemp}
\\
\frac{(\theta+\frac{1}{2})S_{n+1}-2\theta S_n+(\theta-\frac {1}{2})S_{n-1}}{\Delta t} -\theta (D_c+\eps_2)\Delta S_{n+1}-(D_c-\theta(D_c+2\eps_2))\Delta S_n-\theta \eps_2\Delta S_{n-1} \nonumber
\\
+((\theta+1)\bu_n-\theta \bu_{n-1})\cdot \nabla \Big(\theta \frac{D_c+\eps_2}{D_c}S_{n+1}+\Big (1-\theta \frac{D_c+2\eps_2}{D_c}S_{n}\Big)+\theta\frac{\eps_2}{D_c}S_{n-1}\Big)
=0, \label{forcont}
\end{gather}
with the parameters $\theta\in {[\frac{1}{2},1]}$ and $\epsilon$, $\epsilon_{1},\eps_2 \geq 0$. Numerical realizations suggest that sufficient stabilizations are obtained with the choices  $\eps={O}(\nu)$, $\eps_1={O}({\gamma})$ and $\eps_2={O}({D_c})$. There are several variants of given time step scheme. By appropriate choices of $\theta$, $\eps$, $\eps_1$ and $\eps_2$ well known time stepping schemes are obtained. For instance, the choices $\theta=1,\eps=\eps_1=\eps_2=0$ and $\theta=1/2, \eps=\eps_1=\eps_2=0$ lead to just usual BDF2LE and CNLE, respectively.

There are number of investigations of  the stabilization techniques of the time discretizations applied to NSE. Some related works include
\cite{L09}, where an artificial viscosity stabilization of the linear system and correction for the associated loss of accuracy is studied in \cite{DP09}, where  two first-order semi-implicit schemes for eddy viscosity model is investigated for the NSE. Examples of other related stabilizations can be found in the literature, e.g, \cite{H03,WL97}.

Unlike the common stabilization technique discussed in these studies, in the method (\ref{forvel})-(\ref{forcont}) the curvature $(\bu_{n+1}-2\bu_{n}+\bu_{n-1})$ is stabilized to obtain efficient and accurate robust method for (\ref{bous}). The idea was first used in \cite{C14,C16} as a time-stepping method for ODE's. A considerable number of investigations based on the same spirit can be found in \cite{HLLT,LLT,LC14,W9,YJ}. The researchers  show that the curvature stabilization method has more advantages of being more accurate and keeping important flow quantities.

To define the method precisely, we will approximate the solution of (\ref{forvel})-(\ref{fortemp}) by using the finite element method. Let $\bfX=(\bfH_0^1(\Omega))^d$, $Q= L_0^2(\Omega)$ be the velocity and pressure spaces and  $W:=\{S\in H^1(\Omega): S=0 \ \text{on} \ \Gamma_D\}$ and $\Psi:=\{\Phi\in H^1(\Omega): \Phi=0 \ \text{on} \ \Gamma_D\}$ be the temperature and concentration spaces, respectively.

Let $\bfX^h\subset \bfX, W^h\subset W, \Psi^h\subset \Psi$ and $Q^h\subset Q$  be finite element spaces where the velocity and pressure spaces fulfill the inf-sup condition (\ref{infsup}). The usual $L^{2}(\Omega)$ norm and the inner product is denoted by $\norm{.}$ and $(\cdot,\cdot) $, respectively. Define skew symmetric trilinear forms
 \begin{eqnarray}
 b^{*}(\bu,\bv,\bw) &=&\frac{1}{2} (\bu \cdot \nabla\bv, \bw)-\frac{1}{2}(\bu\cdot \nabla\bw, \bv), \label{s1}
 \\
 c^{*}(\bu,T,\chi) &=&\frac{1}{2} (\bu \cdot \nabla T, \chi)-\frac{1}{2}(\bu\cdot \nabla\
 \chi, T), \label{s2}
 \\
  d^{*}(\bu,S,\Phi) &=&\frac{1}{2} (\bu \cdot \nabla S, \Phi)-\frac{1}{2}(\bu\cdot \nabla \Phi, S) \label{s3}
 \end{eqnarray}
and the operators
\begin{eqnarray}
{D}_{n,\theta}(w) :=\frac{(\theta+\frac{1}{2})w_{n+1}-2\theta w_n+(\theta-\frac {1}{2})w_{n-1}}{\Delta t},  \label{op1}
\\
{F}_{n+\theta}^{\delta,\mu}(w):= \theta \frac{(\mu+\delta)}{\mu} w_{n+1}+\Big(1-\theta\frac{\mu+2\delta}{\mu}\Big) w_n+\theta \frac{\delta}{\mu} w_{n-1}, \label{op2}
\\
{H}_{n+\theta}(w):=(\theta+1)w_n-\theta w_{n-1}, \label{op3}
\end{eqnarray}
where $({\delta,\mu})=(\eps,\nu)$ in the case $w=\bu$ (for the velocity), $({\delta,\mu})=(\eps_1,\gamma)$ in the case $w=T$ (for the temperature) and $({\delta,\mu})=(\eps_2,D_c)$ (for the concentration).\\Note that using the definition of the skew symmetric trilinear forms $(\ref{s1})$ and $(\ref{s2})$, one can directly obtain
	\begin{eqnarray}
	b^{*}(\bu,\bv,\bv)=0,\quad c^{*}(\bu,T,T)=0,\quad d^{*}(\bu,S,S)=0 \label{bcdfn}
	\end{eqnarray}
for all $\bu, \bv \in \bfX$ and $ T\in W,  S\in\Psi $. By using the operators (\ref{op1})-(\ref{op2}) and trilinear forms (\ref{s1})-(\ref{s3}), we state a family of second order time stepping method (\ref{forvel})-(\ref{forcont}) in finite dimensional spaces.

{\bf Algorithm}. Let the initial conditions $\bu_0$, $T_0$ and $S_0$ be given. Define $\bu_0^h$, $\bu_{-1}^h$, $T_0^h$, $T_{-1}^h$, $S_{0}^h$ and $S_{-1}^h$ as the nodal interpolants of $\bu_0(\bx), T_0$ and $S_0$, respectively.
Then, given time step $\Delta t$ and $\bu_n,\bu_{n-1},T_n, T_{n-1}, S_n$ and $S_{n-1},$ compute $\bu_{n+1}\in \bfX^h, T_{n+1}\in W^h, S_{n+1}\in \Psi^h, $ and $p_{n+1}\in Q^h$ satisfying
\begin{gather}
({D}_{n+\theta}(\bu^h),\bv^h)+\nu ({F}_{n+\theta}^{\eps,\nu}(\nabla \bu^h),\nabla \bv^h)+
b^{*}({H}_{n+\theta}(\bu^h),{F}_{n+\theta}^{\eps,\nu}(\bu^h),\bv^h))
-({F}_{n+\theta}^{\eps,\nu}(p^h),\nabla \cdot \bv^h), \nonumber
\\
=\beta_{T} (\bfg{H}_{n+\theta}(T^h),\bv^h)+ \beta_{S}(\bfg {H}_{n+\theta}(S^h),\bv^h), \label{Discvel}
\\
(\nabla\cdot \bu^h,q^h)=0,  \label{Dispres}
\\
({D}_{n+\theta}(T^h),\chi^h)+\gamma({F}_{n+\theta}^{\eps_1,\gamma}(\nabla T^h),\nabla \chi^h)+c^{*}({H}_{n+\theta}(\bu^h),{F}_{n+\theta}^{\eps_1,\gamma}(T^h),\chi^h)=0,\label{Distemp}
\\
({D}_{n+\theta}(S^h),\Phi^h)+\gamma({F}_{n+\theta}^{\eps_2,D_c}(\nabla S^h),\nabla \Phi^h)+D^{*}({H}_{n+\theta}(\bu^h),{F}_{n+\theta}^{\eps_2,D_c}(S^h),\Phi^h)=0, \label{Discont}
\end{gather}
for all $(\bv^h,\chi^h,\Phi^h,q^h) \in (\bfX^h,W^h,\Psi^h, Q^h)$.
\begin{remark}
The studied method requires the specifications of the initial condition. The initial condition $\bu_0^h$ needs to be weakly divergence-free in order to achieve stability in our method.
\end{remark}

\begin{remark}
 A family of second order method (\ref{Discvel})-(\ref{Discont}) is a fully decoupled system and under the certain choices of parameters, it requires only a linear system to be solved at each time step.
\end{remark}
\section{Mathematical Preliminaries}

In this paper, we consider a convex polygonal or polyhedral domain $\Omega$ in ${\rm I\!R} ^d$, $ (d=2,3) $ with boundary $\partial \Omega$. Standard notations of \cite{Ada75} for Lebesgue and Sobolev spaces and their norms will be used throughout the paper. We denote the norm in Sobolev spaces $(H^{k}(\Omega))^{d}$ by $\norm{.}_{k}$ and the norms in Lebesgue spaces $(L^{p}(\Omega))^{d}$, $ 1\leq p\leq \infty $, $ p\neq 2$  by $\norm{.}_{L^{p}}$.

Recall that the problem (\ref{bous}) is formulated in the functional spaces $\bfX:=(\bH_{0}^1(\Omega))^d, W:=\{S\in H^1(\Omega): S=0 \ \text{on} \ \Gamma_D\}$ and $\Psi:=\{\Phi\in H^1(\Omega): \Phi=0 \ \text{on} \ \Gamma_D\}$ $Q:=L_{0}^{2}(\Omega)$  and $W:= H^1(\Omega)$. Let  $\bfV:=\{\bv\in \bfX: (\nabla\cdot\bv,q) =0,\forall q\in Q \}$ be the set of weakly divergence free functions in $\bfX$. The dual space of $\bfV$ is denoted by $\bfV^{*}$ and corresponding norm is defined by
\[\norm{ \bff }_{*}  = \sup_{0\neq\bv\in{\bfV}}
\frac{|(\bff,\bv)|}{\norm{\nabla \bv }}
\]
Error analysis will require the following upper bounds for the skew symmetric trilinear forms $(\ref{s1})$ - $(\ref{s3})$, for both velocity and temperature, respectively. We recall the following estimates from \cite{Lay08} and \cite{WT98}.
\begin{lemma}\label{sb}
For $\bu,\bv, \bw \in \bfX$, $T, \chi \in W$ and $S,\Phi \in \Psi$ the skew-symmetric trilinear forms satisfy the following bounds
\begin{eqnarray}
b^{*}(\bu,\bv,\bw)&\leq& C_1\norm{\nabla\bu}\norm{\nabla\bv}\norm{\nabla\bw},\label{nterm1}
\\
c^{*}(\bu,T,\chi)&\leq& C_2\norm{\nabla\bu}\norm{\nabla T}\norm{\nabla \chi}, \label{nterm2}
\\
d^{*}(\bu,S,\Phi)&\leq& C_3\norm{\nabla\bu}\norm{\nabla S}\norm{\nabla \Phi}, \label{nterm3}
\end{eqnarray}
where $C_1:=C_1(\Omega), C_2:=C_2(\Omega)$ and $C_3:=C_3(\Omega)$ are constants depending only on the domain $\Omega$.

Furthermore, it will be assumed that if  $\bv,\nabla\bv \in L^{\infty}(\Omega), T, \nabla T \in L^{\infty}(\Omega)$ and $ S, \nabla S \in L^{\infty}(\Omega)$, the following bounds hold
\begin{eqnarray}
b^{*}(\bu,\bv,\bw)&\leq& \dfrac{1}{2}\Big(\norm{\bu}\norm{\nabla\bv}_{\infty}\norm{\bw}+\norm{\bu}\norm{\bv}_{\infty}\norm{\nabla\bw}\Big), \label{infb} \\
c^{*}(\bu,T,\chi)&\leq& \dfrac{1}{2}\Big(\norm{\bu}\norm{\nabla T}_{\infty}\norm{ \chi}+\norm{\bu}\norm{ T}_{\infty}\norm{\nabla \chi}\Big).\label{infc}
\\
d^{*}(\bu,S,\Phi)&\leq& \dfrac{1}{2}\Big(\norm{\bu}\norm{\nabla S}_{\infty}\norm{ \Phi}+\norm{\bu}\norm{ S}_{\infty}\norm{\nabla \Phi}\Big).\label{infd}
\end{eqnarray}
	
\end{lemma}

The finite element approximation of the problem uses the conforming finite element spaces and $(\bfX^h, Q^h)$ is a pair of finite element spaces satisfying the discrete inf-sup condition,
\begin{eqnarray}
\inf_{q_h\in{Q}^h}\sup_{\bv_h\in {\bfX}^h}\frac{(q_h,\,\nabla\cdot
	\bv_h)}{||\,\nabla \bv_h\,||\,||\,q_h\,||}\geq \beta > 0.
\label{infsup}
\end{eqnarray}
where $\beta$, a constant independent of the mesh size $h$. Some examples of such spaces can be seen in \cite{GR79,GR}.

We introduce the discretely divergence free subspace $\bfV^h\subset \bfX^h$ given by
$$\bfV^h :=\{\bv^h\in\bfX^h : (q^h,\nabla\cdot\bv_h)=0,\forall q^h \in Q^h\}.$$
Note that under (\ref{infsup}), $\bfV^h$ is nonempty and in general $\bfV^h \not \subset \bfV$, see e.g., \cite{GR79}.

Following \cite{BS08,Joh16}, it is assumed that the finite element spaces $(\bfX^h,W^h,\Psi^h,Q^h)$ satisfy the typical approximation properties of piecewise polynomials of degree $(k,k,k,k-1)$,
\begin{eqnarray}
 \inf_{\bv_h\in{\bfX}^h} \left( \|(\bu-\bv_h)\|+h\|\nabla(\bu-\bv_h)\| \right)&\leq&
 C h^{k+1} \| \bu \|_{k+1}\quad{\bu} \in H^{k+1}(\Omega),\label{ap1}
 \\
 \inf_{\chi_h\in{W}^h}\|T-\chi_h\|&\leq & Ch^{k+1}\|T\|_{k+1}\quad {T} \in H^{k+1}(\Omega), \label{ap2}
  \\
  \inf_{\Phi_h\in{W}^h}\|S-\Phi_h\|&\leq & Ch^{k+1}\|S\|_{k+1}\quad {S} \in H^{k+1}(\Omega), \label{ap3}
   \\
   \inf_{q_h\in{Q}^h}\|p-q_h\|&\leq & Ch^k \|p\|_{k}\quad {p} \in H^{k}(\Omega). \label{ap4}
  \end{eqnarray}
The Poincar\'e-Friedrichs inequality will be used throughout the paper and is given by
$$\norm{\bu}\leq C\norm{\nabla\bu}, \quad \forall \bu \in H_{0}^{1}(\Omega),$$
where $C=C(\Omega)$.

The analysis of the method requires the definition of $G$-norm and $F$-norm. Following notation of \cite{JMR,C16}, for the  ${\rm I\!R}^{2n\times 2n}$ symmetric matrix
$$G=\begin{pmatrix}
 	\dfrac{\theta(2\theta+3)}{4}\dfrac{\nu+\epsilon}{\nu}I-\dfrac{\theta(2\theta+1)}{4}\dfrac{\epsilon}{\nu}I & -(\dfrac{(1-\theta)(2\theta+1)}{4}\dfrac{\nu+\epsilon}{\nu}I+\dfrac{(\theta+1)(2\theta-1)}{4}\dfrac{\epsilon}{\nu}I) \\
 	-(\dfrac{(1-\theta)(2\theta+1)}{4}\dfrac{\nu+\epsilon}{\nu}I+\dfrac{(\theta+1)(2\theta-1)}{4}\dfrac{\epsilon}{\nu}I) &\dfrac{\theta(2\theta-1)}{4}\dfrac{\nu+\epsilon}{\nu}I-\dfrac{\theta(-2\theta+3)}{4}\dfrac{\epsilon}{\nu}I)
 	\end{pmatrix}
$$
We introduce $G$-norm
\begin{eqnarray}
\norm{ \begin{bmatrix}
 	\bu\\
 	\bv
 	\end{bmatrix}}_{G}^{2} =\big( \begin{bmatrix}
 \bu\\
 \bv
 \end{bmatrix},G \begin{bmatrix}
 \bu\\
 \bv
 \end{bmatrix}\big) \label{Gnorm}
 \end{eqnarray}
which can be negative. Here $I\in\mathbb{R}^{n\times n}$ is an identity matrix and  $\begin{bmatrix}
 \bu\\
 \bv
\end{bmatrix} $ is a $2n$ vector. The form of $G$-matrix is common in  BDF2 analysis, see e.g.,\cite{HW02} and references therein.

Additionally, let $F$ $\in$ $\mathbb{R}^{n\times n}$ be a symmetric positive definite matrix as follows
\begin{eqnarray}
F=\theta(2\theta-1)I+\dfrac{4\theta^2\epsilon}{\nu}I \label{fmtrx}
\end{eqnarray}
and for any $\bu\in \mathbb{R}^{n}$, define $F$ norm of the $n$ vector $\bu$ by
\begin{eqnarray}
\norm{\bu}_{F}=(\bu,F\bu). \label{Fnorm}
\end{eqnarray}

We now state the following equality which is useful in the error analysis.

\begin{lemma}\label{lem:inn}
The symmetric positive definite matrix  $F$ $\in$  ${\rm I\!R}^{n\times n}$ and the symmetric matrix $G$ $\in$  ${\rm I\!R}^{2n\times 2n}$ which are given above satisfy the following equality:

\begin{eqnarray}
\Big(\frac{(\theta+\frac{1}{2})w_{n+1}-2\theta w_n+(\theta-\frac {1}{2})w_{n-1}}{\Delta t} ,   \theta \frac{(\nu+\epsilon)}{\nu} w_{n+1}-\Big(1-\theta\frac{\nu+2\epsilon}{\nu}\Big) w_n+\theta \frac{\epsilon}{\nu} w_{n-1}\Big) \nonumber\\
 =\norm{ \begin{bmatrix}
 	w_{n+1} \\
 	w_{n}
 	\end{bmatrix}} _{G}^{2} -\norm { \begin{bmatrix}
 	w_{n} \\
 	w_{n-1}
 	\end{bmatrix}}_{G}^{2} 	+\dfrac{1}{4}\norm{w_{n+1}-2w_{n}+w_{n-1}}_{F}^{2}. \label{innpro}
\end{eqnarray}

 \end{lemma}
\begin{proof}
Extending the left hand side of $(\ref{innpro})$ gives
\begin{eqnarray}
&&\Big(\frac{(\theta+\frac{1}{2})w_{n+1}-2\theta w_n+(\theta-\frac {1}{2})w_{n-1}}{\Delta t} ,   \theta \frac{(\nu+\epsilon)}{\nu} w_{n+1}-\Big(1-\theta\frac{\nu+2\epsilon}{\nu}\Big) w_n+\theta \frac{\epsilon}{\nu} w_{n-1}\Big)\nonumber\\
&&=\frac{1}{\Delta t}\Bigg[w_{n+1}^T\big(\theta+\frac{1}{2}\big)\theta\frac{\nu+\eps}{\nu}w_{n+1}-w_{n+1}^T\big(\theta+\frac{1}{2}\big)\bigg(1-\theta\frac{\nu+2\eps}{\nu}\bigg)w_{n} + w_{n+1}^{T}\big(\theta+\frac{1}{2}\big)\theta\frac{\eps}{\nu}w_{n-1}\nonumber\\
&&-2w_{n}^T\theta^2\frac{\nu+\eps}{\nu}w_{n+1}+2w_{n}^T\theta\bigg(1-\theta\frac{\nu+2\eps}{\nu}\bigg)w_{n}-2w_{n}^T\theta^2\frac{\eps}{\nu}w_{n-1}+w_{n-1}^T\big(\theta-\frac{1}{2}\big)\theta\frac{\nu+\eps}{\nu}w_{n+1}\nonumber\\
&&-w_{n-1}^T\big(\theta+\frac{1}{2}\big)\bigg(1-\theta\frac{\nu+2\eps}{\nu}\bigg)w_{n}+w_{n-1}^{T}\big(\theta-\frac{1}{2}\big)\theta\frac{\eps}{\nu}w_{n-1}\Bigg] .\label{L}
\end{eqnarray}
Then, rewrite each term in the right hand side of (\ref{L}) (see,\cite{J12} for details). Lastly, summing the expanded right hand side terms gives the required result for (\ref{innpro}).
\end{proof}

The following properties of $G$-norm are well known and for a detailed derivation of these estimations, the reader is referred to \cite{HW02,C16}.

\begin{lemma}\label{lem:gnorm}
For any $\bu$,$\bv$ $\in$ $R^{n}$, we have
\begin{eqnarray} 	
\big( \begin{bmatrix}
\bu\\
\bv
\end{bmatrix},G \begin{bmatrix}
\bu\\
\bv
\end{bmatrix}\big)&=&\dfrac{2\theta+1}{4}\norm{\bu}^2+\dfrac{-2\theta+1}{4}\norm{\bv}^2+\dfrac{(\theta+1)(2\theta-1)}{4}\norm{\bu-\bv}^2+\dfrac{\theta}{2}\dfrac{\epsilon}{\nu}\norm{\bu-\bv}^2\nonumber\\
&\geq&\dfrac{2\theta+1}{4}\norm{\bu}^2-\dfrac{2\theta-1}{4}\norm{\bv}^2 ,\label{first}
\end{eqnarray}
\begin{eqnarray} 	
\big( \begin{bmatrix}
\bu\\
\bv
\end{bmatrix},G \begin{bmatrix}
\bu\\
\bv
\end{bmatrix}\big)&\leq&\dfrac{2\theta+1}{4}\norm{\bu}^2+\dfrac{(\theta+1)(2\theta-1)}{4}\norm{\bu-\bv}^2+\dfrac{\theta}{2}\dfrac{\epsilon}{\nu}\norm{\bu-\bv}^2\nonumber\\
&\leq&\big(\dfrac{2\theta+1}{4}+\dfrac{(\theta+1)(2\theta-1)}{4}+\dfrac{\theta}{2}\dfrac{\epsilon}{\nu}\big)\norm{\bu}^2\nonumber\\
&&+\big(\dfrac{(\theta+1)(2\theta-1)}{4}+\dfrac{\theta}{2}\dfrac{\epsilon}{\nu}\big)\norm{\bv}^2 .\label{second}
\end{eqnarray}
\end{lemma}

To obtain error bounds for velocity and temperature, we will use the following lemma given in \cite{GR79}.

\begin{lemma}($Discrete$ $Gronwall's$ $Lemma$)\label{gr}
Let $\Delta t, B$ and $a_{n}, b_{n}, c_{n}, d_{n}$ be finite non-negative numbers such that
$$ a_{N}+\Delta t\sum_{n=0}^{N}b_{n}\leq \Delta t\sum_{n=0}^{N-1}d_{n}a_{n} + \Delta t\sum_{n=0}^{N}c_{n} + B \quad \mbox{for} \quad N\geq 1.$$	
Then for all $\Delta t\geq0$,
$$ a_{N}+\Delta t\sum_{n=0}^{N}b_{n}\leq \exp\Big(\Delta t\sum_{n=0}^{N-1}d_{n}\Big)\Big(\Delta t\sum_{n=0}^{N}c_{n} + B\Big)$$
\end{lemma}

\section{Numerical Analysis}
In this section, the numerical analysis of a fully discrete method for solving Darcy-Brinkman system (\ref{bous}) is studied based on  (\ref{Discvel})-(\ref{Discont}). We first provide the stability analysis of the method. Stability bounds are derived by using standard energy arguments. It turns out that the method (\ref{Discvel})-(\ref{Discont}) doesn't depend on any time step sizes.
\begin{lemma}($Unconditional$ $Stability$)
The solutions of (\ref{Discvel})-(\ref{Discont}) satisfy at $t_n=n\Delta t$
\begin{eqnarray}
\norm {T_{N}^{h}}^{2} + \dfrac{1}{2\theta+1}\sum_{n=1}^{N-1} \norm { T_{n+1}^{h}-2 T_{n}^{h}+T_{n-1}^{h}}_{F}^{2}+\dfrac{4\Delta t \gamma}{2\theta+1}\sum_{n=1}^{N-1}\norm {{ F}_{n+\theta}^{\eps_1,\gamma}(\nabla T^{h})}^{2} \nonumber\\
\leq
\biggl(\dfrac{2\theta -1}{2\theta+1}\biggr) ^{N}\norm{ T_{0}^{h}}^{2} + \dfrac{4N}{2\theta+1}\norm{\begin{bmatrix}
	 T_{1}^{h}\\
	 T_{0}^{h}
	\end{bmatrix}}_{G}^{2}.\label{stabt}
\end{eqnarray}
\begin{eqnarray}
\norm {S_{N}^{h}}^{2} + \dfrac{1}{2\theta+1}\sum_{n=1}^{N-1} \norm { S_{n+1}^{h}-2 S_{n}^{h}+S_{n-1}^{h}}_{F}^{2}+\dfrac{4\Delta t D_c}{2\theta+1}\sum_{n=1}^{N-1}\norm {{ F}_{n+\theta}^{\eps_2,D_c}(\nabla S^{h})}^{2} \nonumber\\
\leq
\biggl(\dfrac{2\theta -1}{2\theta+1}\biggr) ^{N}\norm{ S_{0}^{h}}^{2} + \dfrac{4N}{2\theta+1}\norm{\begin{bmatrix}
	S_{1}^{h}\\
	S_{0}^{h}
	\end{bmatrix}}_{G}^{2}.\label{stabc}
\end{eqnarray}
\begin{eqnarray}
\lefteqn{\norm {\bu_{N}^{h}}^{2} + \dfrac{1}{2\theta+1}\sum_{n=1}^{N-1} \norm {\bu_{n+1}^{h}-2\mathbf \bu_{n}^{h}+\mathbf \bu_{n-1}^{h}}_{F}^{2}+\dfrac{4\Delta t Da^{-1}}{2\theta+1}\sum_{n=1}^{N-1}\norm {{ F}_{n+\theta}^{\eps,\nu}( \bu^{h})}^{2}}\nonumber\\
&&+\dfrac{2\Delta t \nu}{2\theta+1}\sum_{n=1}^{N-1}\norm {{ F}_{n+\theta}^{\eps,\nu}(\nabla  \bu^{h})}^{2} \nonumber
\\
&\leq& \dfrac{C\Delta t(2\theta +1)}{\nu}\bigg[1-\big(\dfrac{2\theta-1}{2\theta+1}\big)^{N}\bigg](\norm{T_{1}}^{2} + \norm{T_{0}}^{2})  + \dfrac{C\Delta t}{2\theta+1}
\norm{\begin{bmatrix}
	T_{1}^{h} \\
	 T_{0}^{h}
	\end{bmatrix}}_{G}^{2} \nonumber\\
&&+\dfrac{C\Delta t(2\theta +1)}{\nu}\bigg[1-\big(\dfrac{2\theta-1}{2\theta+1}\big)^{N}\bigg](\norm{S_{1}}^{2} + \norm{S_{0}}^{2})  + \dfrac{C\Delta t}{2\theta+1}
\norm{\begin{bmatrix}
	S_{1}^{h} \\
	 S_{0}^{h}
	\end{bmatrix}}_{G}^{2}
 \nonumber
\\
&&+\dfrac{4N}{2\theta+1}
\norm{\begin{bmatrix}
	 \bu_{1}^{h} \\
	\bu_{0}^{h}
	\end{bmatrix}}_{G}^{2} +\bigg(\dfrac{2\theta-1}{2\theta+1}\bigg)^{N}\norm{\mathbf \bu_{0}^h}^{2}.
	  \label{sta}
\end{eqnarray}
\end{lemma}

\begin{proof}
For stability, one needs to obtain estimation for the temperature and  the concentration, then use them to estimate the velocity. So, first set  $\chi^h={F}_{n+\theta}^{\eps_1,\gamma}(T^{h})$ in (\ref{Distemp}), then use the definition of the skew symmetric form (\ref{bcdfn}) and Lemma \ref{lem:inn} :

\begin{eqnarray}
\dfrac{1}{\Delta t}\norm{ \begin{bmatrix}
	T_{n+1}^{h} \\
	T_{n}^{h}
	\end{bmatrix}} _{G}^{2} -\dfrac{1}{\Delta t}\norm { \begin{bmatrix}
	T_{n}^{h} \\
	T_{n-1}^{h}
	\end{bmatrix}}_{G}^{2} 	+\dfrac{1}{4\Delta t}\norm{T_{n+1}^{h}-2T_{n}^{h}+T_{n-1}^{h}}_{F}^{2}
+\gamma \norm{{F}_{n+\theta}^{\eps_1,\gamma}(\nabla T^{h})}^{2}
= 0.  \label{eq1}
\end{eqnarray}
Next, multiplying both sides of (\ref{eq1}) by $\Delta t$ and taking sum from $n=1$ to $n=N-1$ leads to
\begin{eqnarray}
\norm{ \begin{bmatrix}
	T_{N}^{h} \\
	T_{N-1}^{h}
	\end{bmatrix}} _{G}^{2}  	+\dfrac{1}{4}\sum_{n=1}^{N-1} \norm{T_{n+1}^{h}-2T_{n}^{h}+T_{n-1}^{h}}_{F}^{2} +  \Delta t \gamma \sum_{n=1}^{N-1}\norm{{F}_{n+\theta}^{\eps_1,\gamma}(\nabla T^{h})}^{2}=\norm{ \begin{bmatrix}
	T_{1}^{h} \\
	T_{0}^{h}
	\end{bmatrix}} _{G}^{2}. \label{eq2}
\end{eqnarray}
Using Lemma \ref{lem:gnorm} yields
\begin{eqnarray}
\norm{T_{N}^{h}}^{2} +\dfrac{1}{2\theta +1}\sum_{n=1}^{N-1} \norm{T_{n+1}^{h}-2T_{n}^{h}+T_{n-1}^{h}}_{F}^{2} + \dfrac{4\Delta t\gamma}{2\theta+1}\sum_{n=1}^{N-1}\norm{{F}_{n+\theta}^{\eps_1,\gamma}(\nabla T^{h})}^{2}\nonumber\\
\leq \dfrac{2\theta-1}{2\theta+1}\norm{T_{N-1}^{h} }^{2} +\dfrac{4}{2\theta+1}\norm{\begin{bmatrix}
	T_{1}^{h} \\
	T_{0}^{h}
	\end{bmatrix}}_{G}^{2}.
\end{eqnarray}
The stability estimation (\ref{stabt}) is now obtained by induction for $\norm{T_{N-1}^{h}}^{2}. $\\
For the concentration, we can obtain the stability bound by repeating the estimations of the temperature by setting $\Phi^h={F}_{n+\theta}^{\eps_2,D_c}(S^{h})$ in (\ref{Discont}).\\
For the stability of the velocity, choose $\bv^{h}={F}_{n+\theta}^{\eps,\nu}(\bu^{h})$ in $(\ref{Discvel})$ and $q_{h}={F}_{n+\theta}^{\eps,\nu}(p^{h})$ in $(\ref{Dispres})$. With the definition of the skew symmetry (\ref{bcdfn}), this results
\begin{eqnarray}
({D}_{n+\theta}(\bu^h),{F}_{n+\theta}^{\eps,\nu}(\bu^{h}))+\nu \norm{(\nabla {F}_{n+\theta}^{\eps,\nu}(\bu^{h}))}^2+Da^{-1}\norm{ {F}_{n+\theta}^{\eps,\nu}(\bu^{h})}^2\nonumber\\
=\beta_{T}(\bfg{H}_{n+\theta}(T^h),{F}_{n+\theta}^{\eps,\nu}(\bu^{h}))+ \beta_{S}(\bfg{H}_{n+\theta}(S^h),{F}_{n+\theta}^{\eps,\nu}(\bu^{h})). \label{eq6}
\end{eqnarray}
Apply Cauchy-Schwarz inequality to the right hand side of (\ref{eq6}) and estimate the terms as follows: the right hand side terms in  (\ref{eq6}) are bounded with the Poincar\'e-Friedrichs inequality and Young's inequality
\begin{eqnarray}
|\beta_{T}(\bfg{H}_{n+\theta}(T^h),{F}_{n+\theta}^{\eps,\nu}(\bu^{h}))|
&\leq& C\beta_{T}\norm{\bfg}_{\infty}^2\norm{H_{n+\theta}(T^{h})}\norm{{F}_{n+\theta}^{\eps,\nu}(\nabla \bu^{h})}\nonumber\\
&\leq& \dfrac{C(\theta+1)^{2}}{\nu}\norm{T_{n}^{h}}^{2} +\dfrac{C\theta^{2}}{\nu}\norm{T_{n-1}^{h}}^{2} +\dfrac{\nu}{4}\norm{{F}_{n+\theta}^{\eps,\nu}(\nabla \bu^{h})}^{2} \label{rhs1}
\end{eqnarray}
and
\begin{eqnarray}
|\beta_{S}(\bfg{H}_{n+\theta}(S^h),{F}_{n+\theta}^{\eps,\nu}(\bu^{h}))|
&\leq& C\beta_{S}\norm{\bfg}_{\infty}^2\norm{H_{n+\theta}(S^{h})}\norm{{F}_{n+\theta}^{\eps,\nu}(\nabla \bu^{h})}\nonumber \\
&\leq& \dfrac{C(\theta+1)^{2}}{\nu}\norm{S_{n}^{h}}^{2} +\dfrac{C\theta^{2}}{\nu}\norm{S_{n-1}^{h}}^{2} +\dfrac{\nu}{4}\norm{{F}_{n+\theta}^{\eps,\nu}(\nabla \bu^{h})}^{2} \label{rhs2}
\end{eqnarray}
Using Lemma \ref{lem:inn}, inserting the estimations (\ref{rhs1}) and (\ref{rhs2}) and multiplying by $\Delta t$ along with summation over the time steps, then the equation $(\ref{eq6})$ becomes
\begin{eqnarray}
\lefteqn{\norm{ \begin{bmatrix}
	\bu_{N}^{h} \\
	\bu_{N-1}^{h}
	\end{bmatrix}} _{G}^{2}  	+\dfrac{1}{4}\sum_{n=1}^{N-1} \norm{\bu_{n+1}^{h}-2\bu_{n}^{h}+\bu_{n-1}^{h}}_{F}^{2} + \Delta tDa^{-1}\sum_{n=1}^{N-1}\norm{{F}_{n+\theta}^{\eps,\nu}( \bu^{h})}^{2}}\nonumber\\
&&+ \dfrac{ \Delta t\nu}{2} \sum_{n=1}^{N-1}\norm{{F}_{n+\theta}^{\eps,\nu}(\nabla \bu^{h})}^{2}\nonumber
\\
&\leq&\norm{ \begin{bmatrix}
	\bu_{1}^{h} \\
	\bu_{0}^{h}
	\end{bmatrix}} _{G}^{2} +\dfrac{C\Delta t\theta^{2}}{\nu}\sum_{n=1}^{N-1}\big(\norm{T_{n}^{h}}^{2} +\norm{T_{n-1}^{h}}^{2}\big)+\dfrac{C\Delta t(2\theta+1)}{\nu} \sum_{n=1}^{N-1}\norm{T_{n}^{h}}^{2}\nonumber\\
&&+\dfrac{C\Delta t\theta^{2}}{\nu}\sum_{n=1}^{N-1}\big(\norm{S_{n}^{h}}^{2} +\norm{S_{n-1}^{h}}^{2}\big)+\dfrac{C\Delta t(2\theta+1)}{\nu} \sum_{n=1}^{N-1}\norm{S_{n}^{h}}^{2}.
\end{eqnarray}
The estimation of Lemma \ref{lem:gnorm} and the use of\quad  $\displaystyle\frac{\theta^2}{2\theta+1}\leq 1$ leads to
\begin{eqnarray}
\lefteqn{\norm{ \bu_{N}^{h}}^{2}  	+\dfrac{1}{2\theta+1}\sum_{n=1}^{N-1} \norm{\bu_{n+1}^{h}-2u_{n}^{h}+\bu_{n-1}^{h}}_{F}^{2} +\dfrac{4\Delta t Da^{-1}}{2\theta+1}\sum_{n=1}^{N-1}\norm {{ F}_{n+\theta}^{\eps,\nu}( \bu^{h})}^{2}} \nonumber\\
&&+ \dfrac{2 \Delta t\nu}{2\theta+1} \sum_{n=1}^{N-1}\norm{{F}_{n+\theta}^{\eps,\nu}(\nabla \bu^{h})}^{2}\nonumber\\
&\leq&\dfrac{2\theta-1}{2\theta+1}\norm{\bu_{N-1}^{h}}^{2}+\dfrac{4}{2\theta+1}\norm{ \begin{bmatrix}
	\bu_{1}^{h} \\
	\bu_{0}^{h}
	\end{bmatrix}} _{G}^{2}+\dfrac{C\Delta t}{\nu}\sum_{n=1}^{N-1}\big(\norm{2T_{n}^{h}}^{2} +\norm{T_{n-1}^{h}}^{2}\big)\nonumber\\
&&+\dfrac{C\Delta t}{\nu}\sum_{n=1}^{N-1}\big(\norm{2S_{n}^{h}}^{2} +\norm{S_{n-1}^{h}}^{2}\big). \label{eq10}
\end{eqnarray}
Lastly, the result follows from by using induction on $N$ with the stability bounds (\ref{stabt}) and (\ref{stabc}).
\end{proof}

We now give an error estimation for the second order time stepping method of proposed algorithm which converges in space and in time if sufficiently smoothing of the solution is satisfied. The error analysis requires the true solution of the velocity, temperature and concentration at time level $n+\theta$ i.e. $\bu_{n+\theta}=\bu(t_{n+\theta})$, $T_{n+\theta}=T(t_{n+\theta})$ and $S_{n+\theta}=S(t_{n+\theta})$.
First note that the weak formulation of (\ref{bous}) at time level $(n+\theta)$ reads as follows : find $(\bu,T,S,p) \in (\bfX,W,\Psi,Q) $ such that
\begin{gather}
(\bu_{t}(t_{n+\theta}),\bv^h)+\nu (\nabla \bu_{n+\theta},\nabla \bv^h)+
b^{*}(\bu_{n+\theta},\bu_{n+\theta},\bv^h)+Da^{-1}(\bu_{n+\theta},\bv^h)
-(p_{n+\theta},\nabla \cdot \bv^h)\nonumber\\
=\beta_{T}(\bfg T_{n+\theta},\bv^h)+ \beta_{S}(\bfg S_{n+\theta},\bv^h),\label{vel}
\\
(\nabla\cdot \bu_{n+\theta},q^h)=0,  \label{pres},
\\
(T_{t}(t_{n+\theta}),\chi^h)+\gamma(\nabla T_{n+\theta},\nabla \chi^h)+c^{*}(\bu_{n+\theta},T,\chi^h)=0\label{temp}
\\
(S_{t}(t_{n+\theta}),\Phi^h)+D_c(\nabla S_{n+\theta},\nabla \Phi^h)+d^{*}(\bu_{n+\theta},S,\Phi^h)=0\label{cont}
\end{gather}
for all $(\bv^h ,\chi^h,\Phi^h,q^h) \in (\bfX^h, W^h,\Psi^h,Q^h) $.

We use the following notations for the discrete norms. For $\bv^n \in H^p(\Omega)$, we define
\begin{equation*}
\norm{|\bv|}_{\infty,p}:=\max_{0\leq n\leq N}\|\bv^n\|_p ,  \quad \norm{|\bv|}_{m,p}:=\bigg(\Delta t \sum_{n=0}^{N} \|\bv^n\|_p^m  \bigg)^{\dfrac{1}{m}}.
\end{equation*}
To obtain the optimal convergence, we assume that the following regularity assumptions hold for the true solutions:
\begin{equation}\label{ra}
\begin{array}{rcll}
\bu, T, S &\in& L^{\infty}(0,T;H^1(\Omega))\cap H^1(0,T;H^{k+1}(\Omega))\cap H^3(0,T;L^2(\Omega))\cap H^2(0,T;H^1(\Omega)),\\
p &\in& L^2(0,T;H^{s+1}(\Omega))\cap H^2(0,T;L^2(\Omega)).\\
\end{array}
\end{equation}

\begin{theorem}\label{erreqn}
Let $(\bu, p, T, S )$ be the solution of the problems $(\ref{bous})$ such that the regularity assumptions $(\ref{ra})$ are satisfied. Then, the following bound holds for the differences $e_{n}^{\bu}=\bu_{n}-\bu_{n}^{h}$, $ e_{n}^{T}=T_{n}-T_{n}^{h}$ and $e_{n}^{S}=S_{n}-S_{n}^{h}$:
\begin{eqnarray}
\norm{e^{\bu}_{N}}^2+\norm{e^{T}_{N}}^2+\norm{e^{S}_{N}}^2+\dfrac{1}{2\theta+1}\sum_{n=1}^{N-1} \norm{e^{\bu}_{n+1}-2e^{\bu}_{n}+e^{\bu}_{n-1}}_{F}^{2}+ \dfrac{ 2\Delta t\nu}{2\theta+1} \sum_{n=1}^{N-1}\norm{{F}_{n+\theta}^{\eps,\nu}(\nabla e^{\bu})}^{2}\nonumber
\end{eqnarray}
\begin{eqnarray}
\lefteqn{+\dfrac{ 4\Delta tDa^{-1}}{2\theta+1} \sum_{n=1}^{N-1}\norm{{F}_{n+\theta}^{\eps,\nu}( e^{\bu})}^{2}+\dfrac{1}{2\theta+1}\sum_{n=1}^{N-1} \norm{e^{T}_{n+1}-2e^{T}_{n}+e^{T}_{n-1}}_{F}^{2} + \dfrac{ 2\Delta t\gamma}{2\theta+1} \sum_{n=1}^{N-1}\norm{{F}_{n+\theta}^{\eps_1,\gamma}(\nabla e^{T})}^{2}}\nonumber\\
&&+\dfrac{1}{2\theta+1}\sum_{n=1}^{N-1} \norm{e^{S}_{n+1}-2e^{S}_{n}+e^{S}_{n-1}}_{F}^{2} + \dfrac{ 2\Delta tD_c}{2\theta+1} \sum_{n=1}^{N-1}\norm{{F}_{n+\theta}^{\eps_2, D_c}(\nabla e^{S})}^{2}\nonumber\\
&\leq&\exp(\tilde{C}T)\Bigg[\bigg(\dfrac{2\theta-1}{2\theta+1}\bigg)^{N}\Big(\norm{e^{\bu}_{0}}^2+\norm{e^{T}_{0}}+\norm{e^{S}_{0}}^2\Big)+C\bigg(1-\bigg(\dfrac{2\theta-1}{2\theta+1}\bigg)^{N}\bigg)\Big(\norm{
	e^{\bu}_1}^2+
	\norm{
	e^{\bu}_{0}}^2 \nonumber\\
	&&+\norm{
	e^{T}_1}^2+
	\norm{e^{T}_{0}}^2+\norm{
	e^{S}_1}^2+\norm{
	e^{S}_{0}}^2\Big)+C\bigg(1-\bigg(\dfrac{2\theta-1}{2\theta+1}\bigg)^{N}\bigg)\Big(\nu^{-1}\Delta t^4 \norm{|p_{tt}|}_{2,0}^{2}\nonumber\\
	&&+\nu^{-1}h^{2k+2}\norm{|p|}_{2,k+1}^{2}+\nu^{-1}\Delta t^4\norm{|\bu_{ttt}|}_{2,0}^{2}+\nu^{-1}h^{2k+2}\norm{|\bu_{t}|}_{2,k+1}^{2}\nonumber\\
&&	+(\nu+\nu^{-1}+\nu^{-1}\norm{|\nabla\bu|}_{\infty}^{2}+\gamma^{-1}\norm{\nabla T}_{\infty}^{2}+D_c^{-1}\norm{\nabla S}_{\infty}^{2})\Delta t^4\norm{|\nabla\bu_{tt}|}_{2,0}^{2}\nonumber\\
	&&+(\nu +\nu^{-1}+\nu^{-1}\norm{|\nabla \bu|}_{\infty}^{2}+\gamma^{-1}\norm{|\nabla T|}_{\infty}^{2}+D_c^{-1}\norm{|\nabla S|}_{\infty}^{2})h^{2k}\norm{|\bu|}_{2,k+1}^{2}\nonumber\\
	&&+\gamma^{-1}\Delta t^4\norm{|T_{ttt}|}_{2,0}^{2}+(\gamma+\gamma^{-1}+\gamma^{-1}\norm{|\nabla\bu |}_{\infty}^{2}+\nu^{-1}\beta_{T}^2\norm{\bfg}_{\infty}^2)\Delta t^4\norm{|\nabla T_{tt}|}_{2,0}^{2} \nonumber\\
	&&+\gamma^{-1}h^{2k+2}\norm{|T_{t}|}_{2,k+1}^{2}+(\gamma+\gamma^{-1}\norm{|\nabla \bu|}_{\infty}^{2}+\nu^{-1}\beta_{T}^2\norm{\bfg}_{\infty}^2) h^{2k}\norm{|T|}_{2,k+1}^{2}\nonumber\\
	&&+D_c^{-1}\Delta t^4\norm{|S_{ttt}|}_{2,0}^{2}+(D_c+D_c^{-1}+D_c^{-1}\norm{|\nabla\bu |}_{\infty}^{2}+\nu^{-1}\beta_{S}^2\norm{\bfg}_{\infty}^2)\Delta t^4\norm{|\nabla S_{tt}|}_{2,0}^{2}\nonumber\\
	&&+D_c^{-1}h^{2k+2}\norm{|S_{t}|}_{2,k+1}^{2}+(D_c+D_c^{-1}\norm{|\nabla \bu|}_{\infty}^{2}+\nu^{-1}\beta_{S}^2\norm{\bfg}_{\infty}^2) h^{2k}\norm{|S|}_{2,k+1}^{2}\Big)\Bigg].\label{erreq}
\end{eqnarray}
\end{theorem}
\begin{remark}\label{re}
	Note that if one formulates Theorem \ref{erreqn} for the most common choice of inf-sup stable finite element spaces, like Taylor Hood element, for the velocity and pressure and piecewise quadratics polynomials for the temperature and the concentration, then the optimal errors for the velocity, temperature and concentration are obtained. Similarly, second order accuracy in time is achieved with these finite element choices.
\end{remark}
\begin{corollary} \label{corr}
	Under the assumptions of Theorem \ref{erreqn}, let $(\bfX^h, W^h,\Psi^h, Q^h)=(P_2, P_2, P_2, P_1)$ be the finite element spaces given by $Remark$ $\ref{re}$. Then the asymptotic error estimation satisfies

\begin{eqnarray}
\norm{e^{\bu}_{N}}^2+\norm{e^{T}_{N}}^2+\norm{e^{S}_{N}}^2+\dfrac{1}{2\theta+1}\sum_{n=1}^{N-1} \norm{e^{\bu}_{n+1}-2e^{\bu}_{n}+e^{\bu}_{n-1}}_{F}^{2}+ \dfrac{ 2\Delta t\nu}{2\theta+1} \sum_{n=1}^{N-1}\norm{{F}_{n+\theta}^{\eps,\nu}(\nabla e^{\bu})}^{2}\nonumber\\
+\dfrac{ 4\Delta tDa^{-1}}{2\theta+1} \sum_{n=1}^{N-1}\norm{{F}_{n+\theta}^{\eps,\nu}( e^{\bu})}^{2}+\dfrac{1}{2\theta+1}\sum_{n=1}^{N-1} \norm{e^{T}_{n+1}-2e^{T}_{n}+e^{T}_{n-1}}_{F}^{2} + \dfrac{ 2\Delta t\gamma}{2\theta+1} \sum_{n=1}^{N-1}\norm{{F}_{n+\theta}^{\eps,\gamma}(\nabla e^{T})}^{2}\nonumber\\
+\dfrac{1}{2\theta+1}\sum_{n=1}^{N-1} \norm{e^{S}_{n+1}-2e^{S}_{n}+e^{S}_{n-1}}_{F}^{2} + \dfrac{ 2\Delta tD_S}{2\theta+1} \sum_{n=1}^{N-1}\norm{{F}_{n+\theta}^{\eps, D_c}(\nabla e^{S})}^{2}\nonumber\\
  \leq C((\Delta t)^4+h^4+\norm{e_{0}^{\bu}}^2+\norm{e_{1}^{\bu}}^2+\norm{e_{0}^{T}}^2+\norm{e_{1}^{T}}^2+\norm{e_{0}^{S}}^2+\norm{e_{1}^{S}}^2))\nonumber.
		\end{eqnarray}	
	\end{corollary}
\begin{proof}
Application of the approximation properties (\ref{ap1})- $(\ref{ap4})$ in the right hand side of (\ref{erreq}) and the regularity assumption (\ref{ra}) gives the required result.
\end{proof}
We now give the proof of our main theorem.
\begin{proof}
First step is to obtain the error equations. By using the operators
(\ref{op1})-(\ref{op3}), adding and subtracting terms in  $(\ref{vel})$-$(\ref{temp})$, true solutions, $(u_{n+1}, p_{n+1}, T_{n+1}, S_{n+1})$ at time level $n+\theta $ satisfy

\begin{eqnarray}
({D}_{n+\theta}(\bu),\bv^{h})+\nu ({F}_{n+\theta}^{\eps,\nu}(\nabla \bu),\nabla \bv^{h})+
b^{*}({H}_{n+\theta}(\bu),{F}_{n+\theta}^{\eps,\nu}(\bu),\bv^{h})
+Da^{-1}({F}_{n+\theta}^{\eps,\nu}(\bu),\bv^{h})\nonumber\\
-(p_{n+\theta},\nabla \cdot \bv^{h})
=\beta_{T}(\bfg H_{n+\theta}(T),\bv^{h}) +\beta_{S}(\bfg H_{n+\theta}(S),\bv^{h})\nonumber\\ +E_1(\bu_{n+\theta},T_{n+\theta},S_{n+\theta};\bv^{h}), \label{trvel}
\end{eqnarray}
\begin{eqnarray}
({D}_{n+\theta}(T),\chi^{h})+\gamma ({F}_{n+\theta}^{\eps_1,\gamma}(\nabla T),\nabla \chi^{h})+
c^{*}({H}_{n+\theta}(\bu),{F}_{n+\theta}^{\eps_1,\gamma}(T),\chi^{h}) =E_2(\bu_{n+\theta},T_{n+\theta};\chi^{h}) \label{trtem}
\end{eqnarray}
and
\begin{eqnarray}
({D}_{n+\theta}(S),\Phi^{h})+D_c ({F}_{n+\theta}^{\eps_2,D_c}(\nabla S),\nabla \Phi^{h})+
d^{*}({H}_{n+\theta}(\bu),{F}_{n+\theta}^{\eps_2,D_c}(S),\Phi^{h}) =E_3(\bu_{n+\theta},S_{n+\theta};\Phi^{h}) \label{trcont}
\end{eqnarray}
for all $(\bv^{h},\chi^{h},\Phi^{h}) \in (\bfX^{h},W^{h},\Psi^{h})$, where
\begin{eqnarray}
E_1(\bu_{n+\theta},T_{n+\theta},S_{n+\theta};\bv^{h})&=&(D_{n+\theta}(\bu)-\bu_{t}(t_{n+\theta}),\bv^{h})+\nu(\nabla (F_{n+\theta}^{\eps,\nu}(\bu)-\bu_{n+\theta}),\nabla\bv^{h})\nonumber\\
&&+b^{*}(H_{n+\theta}(\bu)-\bu_{n+\theta},F_{n+\theta}^{\eps,\nu}(\bu),\bv^{h})
+b^{*}(\bu_{n+\theta},F_{n+\theta}^{\eps,\nu}(\bu)-\bu_{n+\theta},\bv^{h})\nonumber\\
&&+Da^{-1}({F}_{n+\theta}^{\eps,\nu}(\bu)-\bu_{n+\theta},\bv^{h})+ \beta_{T}(\bfg(H_{n+\theta}(T)-T_{n+\theta}),\bv^{h} )\nonumber\\
&&+\beta_{S}(\bfg(H_{n+\theta}(S)-S_{n+\theta}),\bv^{h}), \label{E1}
\end{eqnarray}
\begin{eqnarray}
E_2(\bu_{n+\theta},T_{n+\theta};\chi^{h})&=&(D_{n+\theta}(T)-T_{t}(t_{n+\theta}),\chi^{h})+\gamma(\nabla(F_{n+\theta}^{\eps_1,\gamma}(T)-T_{n+\theta}),\nabla \chi^{h})\nonumber\\
&&+c^{*}(H_{n+\theta}(\bu)-\bu_{n+\theta},F_{n+\theta}^{\eps_1,\gamma}(T),\chi^{h})\nonumber\\
&&+c^{*}(\bu_{n+\theta},F_{n+\theta}^{\eps_1,\gamma}(T)-T_{n+\theta},\chi^{h})\label{E2}
\end{eqnarray}
and
\begin{eqnarray}
E_3(\bu_{n+\theta},S_{n+\theta};\Phi^{h})&=&(D_{n+\theta}(S)-S_{t}(t_{n+\theta}),\Phi^{h})+D_c(\nabla(F_{n+\theta}^{\eps_2,D_c}(S)-S_{n+\theta}),\nabla \Phi^{h})\nonumber\\
&&+d^{*}(H_{n+\theta}(\bu)-\bu_{n+\theta},F_{n+\theta}^{\eps_2,D_c}(S),\Phi^{h})\nonumber\\
&&+d^{*}(\bu_{n+\theta},F_{n+\theta}^{\eps_2,D_c}(S)-S_{n+\theta},\Phi^{h}).\label{E3}
\end{eqnarray}
Let us decompose the velocity and temperature error in the following way;
\begin{eqnarray}
e^{\bu}_{n}=\bu_{n}-\bu_{n}^{h}
=(\bu_{n}-I^{h}(\bu_{n}))+(I^{h}(\bu_{n})-\bu_{n}^{h})=\bfeta^{\bu}_{n}+\bphi_{n}^h,\nonumber
\end{eqnarray}
\begin{eqnarray}
e^{T}_{n}=T_{n}-T_{n}^{h}
=(T_{n}-I^{h}(T_{n}))+(I^{h}(T_{n})-T_{n}^{h})=\eta
^{T}_{n}+\xi_{n}^h, \nonumber
\end{eqnarray}
\begin{eqnarray}
e^{S}_{n}=S_{n}-S_{n}^{h}
=(S_{n}-I^{h}(S_{n}))+(I^{h}(S_{n})-S_{n}^{h})=\eta
^{S}_{n}+\zeta_{n}^h, \nonumber
\end{eqnarray}
where $I^{h}(\bu_{n}) \in \bfV^{h}$ is the interpolant of $\bu_{n} $ in $\bfV^{h}$, $I^{h}(T_{n}) \in W^{h}$ is the interpolant of $T_{n} $ in $W^{h}$ and $I^{h}(S_{n}) \in \Psi^{h}$ is the interpolant of $S_{n} $ in $\Psi^{h}$ .\\
The error equations for the velocity, temperature and concentration are obtained by subtracting $(\ref{Discvel})$, $(\ref{Distemp})$,$(\ref{Discont})$  from $(\ref{trvel})$, $(\ref{trtem})$ and $(\ref{trcont})$, respectively:
\begin{eqnarray}
(D_{n+\theta}(e^{\bu}),\bv^{h})+\nu(F_{n+\theta}^{\eps,\nu}(\nabla e^{\bu}),\nabla\bv^{h})+b^{*}(H_{n+\theta}(\bu),F_{n+\theta}^{\eps,\nu}(e^{\bu}),\bv^{h})+Da^{-1}(F_{n+\theta}^{\eps,\nu}( e^{\bu}), \bv^h)\nonumber\\
=E_1(\bu_{n+\theta},T_{n+\theta};\bv^{h})-(F_{n+\theta}^{\eps,\nu}(p)-p_{n+\theta},\nabla\cdot\bv^{h})
+(F_{n+\theta}^{\eps,\nu}(p)-q^{h},\nabla\cdot\bv^{h})\nonumber\\
-b^{*}(H_{n+\theta}(e^{\bu}),F_{n+\theta}^{\eps,\nu}(\bu^{h}),\bv^{h})
+\beta_{T}(\bfg H_{n+\theta}(e^{T}),\bv^{h})
+\beta_{S}(\bfg H_{n+\theta}(e^{S}),\bv^{h}), \label{errvel}
\end{eqnarray}
\begin{eqnarray}
(D_{n+\theta}(e^{T}),\chi^{h})+\gamma(F_{n+\theta}^{\eps_1,\gamma}(\nabla e^{T}),\nabla \chi^{h})+c^{*}(H_{n+\theta}(\bu),F_{n+\theta}^{\eps_1,\gamma}(e^{T}),\chi^{h})\nonumber\\
=E_2(\bu_{n+\theta},T_{n+\theta};\chi^{h})
-c^{*}(H_{n+\theta}(e^{\bu}),F_{n+\theta}^{\eps_1,\gamma}(T^{h}),\chi^{h}) \label{errtem}
\end{eqnarray}
\begin{eqnarray}
(D_{n+\theta}(e^{S}),\chi^{h})+ D_c(F_{n+\theta}^{\eps_2,D_c}(\nabla e^{S}),\nabla \Phi^{h})+d^{*}(H_{n+\theta}(\bu),F_{n+\theta}^{\eps_2,D_c}(e^{S}),\Phi^{h})\nonumber\\
=E_3(\bu_{n+\theta},S_{n+\theta};\Phi^{h})
-d^{*}(H_{n+\theta}(e^{\bu}),F_{n+\theta}^{\eps_2,D_c}(S^{h}),\Phi^{h}). \label{errcont}
\end{eqnarray}
Taking $\bv^{h}=F_{n+\theta}^{\eps,\nu}(\bphi^{h}_{n})$ in (\ref{errvel}), $\chi^{h}=F_{n+\theta}^{\eps_1,\gamma}(\xi^{h}_{n})$ in (\ref{errtem}) and $\Phi^{h}=F_{n+\theta}^{\eps_2,D_c}(\zeta^{h}_{n})$ in (\ref{errcont}), using the error decompositions and using the  skew symmetry of the trilinear form, it follows that
\begin{eqnarray}
\lefteqn{\dfrac{1}{\Delta t}\norm{ \begin{bmatrix}
		\bphi^{h}_{n+1} \\
		\bphi^{h}_{n}
		\end{bmatrix}} _{G}^{2} -\dfrac{1}{\Delta t}\norm { \begin{bmatrix}
		\bphi^{h}_{n} \\
		\bphi^{h}_{n-1}
		\end{bmatrix}}_{G}^{2} 	+\dfrac{1}{4\Delta t}\norm{\bphi^{h}_{n+1}-2\bphi^{h}_{n}+\bphi^{h}_{n-1}}_{F}^{2}}\nonumber\\
&&+\nu\norm{F_{n+\theta}^{\eps,\nu}(\nabla \bphi^{h}_{n})}^{2}
+ Da^{-1}\norm{F_{n+\theta}^{\eps,\nu}(\bphi^{h}_{n})}^{2} \nonumber\\
&=&E_1(\bu_{n+\theta},T_{n+\theta},S_{n+\theta};F_{n+\theta}^{\eps,\nu}(\bphi^{h})) -(F_{n+\theta}^{\eps,\nu}(p)-p_{n+\theta},\nabla\cdot F_{n+\theta}^{\eps,\nu}(\bphi^{h}))\nonumber\\
&&+(F_{n+\theta}^{\eps,\nu}(p)-q^{h},\nabla\cdot F_{n+\theta}^{\eps,\nu}(\bphi^{h}))-(D_{n+\theta}(\bfeta^{\bu}),F_{n+\theta}^{\eps,\nu}(\bphi^ {h}))\nonumber\\
&&-\nu(F_{n+\theta}^{\eps,\nu}(\nabla\bfeta^{\bu}),F_{n+\theta}^{\eps,\nu}(\nabla(\bphi^{h}))-Da^{-1}(F_{n+\theta}^{\eps,\nu}(\bfeta^{\bu}),F_{n+\theta}^{\eps,\nu}(\bphi^{h}))\nonumber\\
&&-b^{*}(H_{n+\theta}(\bu),F_{n+\theta}^{\eps,\nu}(\bfeta^{\bu}),F_{n+\theta}^{\eps,\nu}(\bphi^{h}))
-b^{*}(H_{n+\theta}(\bphi^{h}),F_{n+\theta}^{\eps,\nu}(\bu^{h}),F_{n+\theta}^{\eps,\nu}(\bphi^{h}))\nonumber\\
&&-b^{*}(H_{n+\theta}(\bfeta^{\bu}),F_{n+\theta}^{\eps,\nu}(\bu^{h}),F_{n+\theta}^{\eps,\nu}(\bphi^{h}))+\beta_{T}(\bfg H_{n+\theta}(\xi^{h}),F_{n+\theta}^{\eps,\nu}(\bphi^{h}))\nonumber\\
&&+\beta_{T}(\bfg H_{n+\theta}(\eta^{T}),F_{n+\theta}^{\eps,\nu}(\bphi^{h}))
 +\beta_{S}(\bfg H_{n+\theta}(\xi^{h}),F_{n+\theta}^{\eps,\nu}(\bphi^{h}))\nonumber\\
 &&+\beta_{S}(\bfg H_{n+\theta}(\eta^{S}),F_{n+\theta}^{\eps,\nu}(\bphi^{h})), \label{eqnvel}
\end{eqnarray}
\begin{eqnarray}
\lefteqn{\dfrac{1}{\Delta t}\norm{ \begin{bmatrix}
		\xi_{n+1}^{h} \\
		\xi_{n}^{h}
		\end{bmatrix}} _{G}^{2} -\dfrac{1}{\Delta t}\norm { \begin{bmatrix}
		\xi_{n}^{h} \\
		\xi_{n-1}^{h}
		\end{bmatrix}}_{G}^{2} 	+\dfrac{1}{4\Delta t}\norm{\xi_{n+1}^{h}-2\xi_{n}^{h}+\xi_{n-1}^{h}}_{F}^{2}+\gamma\norm{{F}_{n+\theta}^{\eps_1,\gamma}(\nabla \xi^{h})}^{2}} \nonumber\\
&=&E_2(\bu_{n+\theta},T_{n+\theta};F_{n+\theta}^{\eps_1,\gamma}(\xi^{h}))-(D_{n+\theta}(\eta^{T}),F_{n+\theta}^{\eps,\gamma}(\xi^{h}))-\gamma(F_{n+\theta}^{\eps_1,\gamma}(\nabla\eta^{T}),F_{n+\theta}^{\eps_1,\gamma}(\nabla\xi^{h})\nonumber\\
&&-c^{*}(H_{n+\theta}(\bu),F_{n+\theta}^{\eps_1,\gamma}(\eta^{T}),F_{n+\theta}^{\eps_1,\gamma}(\xi^{h}))-c^{*}(H_{n+\theta}(\phi^{h}),F_{n+\theta}^{\eps_1,\gamma}(T^{h}),F_{n+\theta}^{\eps_1,\gamma}(\xi^{h}))\nonumber\\
&&-c^{*}(H_{n+\theta}(\bfeta^{\bu}),F_{n+\theta}^{\eps_1,\gamma}(T^{h}),F_{n+\theta}^{\eps_1,\gamma}(\xi^{h}))\label{eqntem}
\end{eqnarray}
and
\begin{eqnarray}
\lefteqn{\dfrac{1}{\Delta t}\norm{ \begin{bmatrix}
		\zeta_{n+1}^{h} \\
		\zeta_{n}^{h}
		\end{bmatrix}} _{G}^{2} -\dfrac{1}{\Delta t}\norm { \begin{bmatrix}
		\zeta_{n}^{h} \\
		\zeta_{n-1}^{h}
		\end{bmatrix}}_{G}^{2} 	+\dfrac{1}{4\Delta t}\norm{\zeta_{n+1}^{h}-2\zeta_{n}^{h}+\zeta_{n-1}^{h}}_{F}^{2}+D_c\norm{{F}_{n+\theta}^{\eps_2,D_c}(\nabla \zeta^{h})}^{2}} \nonumber\\
&=&E_3(\bu_{n+\theta},S_{n+\theta};F_{n+\theta}^{\eps_2,D_c}(\zeta^{h}))-(D_{n+\theta}(\eta^{S}),F_{n+\theta}^{\eps_2,D_c}(\zeta^{h}))-D_c(F_{n+\theta}^{\eps_2,D_c}(\nabla\eta^{S}),F_{n+\theta}^{\eps_2,D_c}(\nabla\zeta^{h})\nonumber\\
&&-d^{*}(H_{n+\theta}(\bu),F_{n+\theta}^{\eps_2,D_c}(\eta^{S}),F_{n+\theta}^{\eps_2,D_c}(\zeta^{h}))-d^{*}(H_{n+\theta}(\phi^{h}),F_{n+\theta}^{\eps_2,D_c}(S^{h}),F_{n+\theta}^{\eps_2,D_c}(\zeta^{h}))\nonumber\\
&&-d^{*}(H_{n+\theta}(\bfeta^{\bu}),F_{n+\theta}^{\eps_2,D_c}(T^{h}),F_{n+\theta}^{\eps_2,D_c}(\zeta^{h})). \label{eqncont}
\end{eqnarray}
To bound the first term in the right hand side of (\ref{eqnvel}), we consider each term in $(\ref{E1})$. Using Cauchy-Schwarz, Young's, Poincar\'e-Friedrichs inequalities and Taylor's theorem, the first term in $(\ref{E1})$ is bounded by
 \begin{eqnarray}
(D_{n+\theta}(\bu)-\bu_{t}(t_{n+\theta}),F_{n+\theta}^{\eps,\nu}(\bphi^{h}))
&\leq&\dfrac{\nu}{64}\norm{F_{n+\theta}^{\eps,\nu}(\nabla\bphi^{h})}^{2}+C\nu^{-1}\norm{D_{n+\theta}(\bu)-\bu_{t}(t_{n+\theta})}^{2}\nonumber\\
&\leq&\dfrac{\nu}{64}\norm{F_{n+\theta}^{\eps,\nu}(\nabla\bphi^{h})}^{2}+C\nu^{-1}\theta^{6}\Delta t^{3}\int_{t_{n-1}}^{t_{n+1}}\norm{\bu_{ttt}}^{2}dt. \nonumber
\end{eqnarray}
Similarly, we have
\begin{eqnarray}
\nu(\nabla(F_{n+\theta}^{\eps,\nu}(\bu)-\bu_{n+\theta}),\nabla F_{n+\theta}^{\eps,\nu}(\bphi^{h}))
&\leq& \dfrac{\nu}{64}\norm{F_{n+\theta}^{\eps,\nu}(\nabla\bphi^{h})}^{2}+C\nu\norm{\nabla(\theta\bu_{n+1}+(1-\theta)\bu_{n}-\bu_{n+\theta})}^{2}\nonumber\\
&& +C\nu^{-1}\epsilon^{2}\theta^{2}\norm{\nabla(\bu_{n+1}-2\bu_{n}+\bu_{n-1})}^{2} \nonumber\\
&\leq&\dfrac{\nu}{64}\norm{F_{n+\theta}^{\eps,\nu}(\nabla\bphi^{h})}^{2}+C\nu\theta^{2}(1-\theta)^{2}\Delta t^{3}\int_{t_{n}}^{t_{n+1}}\norm{\nabla\bu_{tt}}^{2}dt\nonumber\\
&&+C\nu^{-1}\epsilon^{2}\theta^{2}\Delta t^{3}\int_{t_{n-1}}^{t_{n+1}}\norm{\nabla\bu_{tt}}^{2}dt.\nonumber
\end{eqnarray}
We use Cauchy-Schwarz, Young's inequalities and Taylor's theorem to bound the nonlinear terms
\begin{eqnarray}
\lefteqn{b^{*}(H_{n+\theta}(\bu)-\bu_{n+\theta},F_{n+\theta}^{\eps,\nu}(\bu),F_{n+\theta}^{\eps,\nu}(\bphi^{h}))}\nonumber\\
&\leq&\dfrac{\nu}{64}\norm{F_{n+\theta}^{\eps,\nu}(\nabla\bphi^{h})}^{2} +C\nu^{-1}\norm{\nabla(H_{n+\theta}(\bu)-\bu_{n+\theta})}^{2}\norm{F_{n+\theta}^{\eps,\nu}(\nabla\bu)}^{2}\nonumber\\
&\leq&\dfrac{\nu}{64}\norm{F_{n+\theta}^{\eps,\nu}(\nabla\bphi^{h})}^{2}+C\nu^{-1}\theta^{2}(1+\theta^{2})\Delta t^{3}\norm{F_{n+\theta}^{\eps,\nu}(\nabla\bu)}^{2}\int_{t_{n-1}}^{t_{n+1}}\norm{\nabla \bu_{tt}}^{2}dt.\nonumber
\end{eqnarray}
and
\begin{eqnarray}
\lefteqn{b^{*}(\bu_{n+\theta},F_{n+\theta}^{\eps,\nu}(\bu)-\bu_{n+\theta},F_{n+\theta}^{\eps,\nu}(\bphi^{h}))}\nonumber\\
&\leq&\dfrac{\nu}{64}\norm{F_{n+\theta}^{\eps,\nu}(\nabla\bphi^{h})}^{2}+C\nu^{-1}\norm{\nabla\bu_{n+\theta}}^{2}\norm{\nabla(F_{n+\theta}^{\eps,\nu}(\bu)-\bu_{n+\theta})}^{2}\nonumber\\
&\leq&\dfrac{\nu}{64}\norm{F_{n+\theta}^{\eps,\nu}(\nabla\bphi^{h})}^{2}+C\nu^{-1}\theta^{2}(1-\theta^{2})\Delta t^{3}\norm{\nabla\bu_{n+\theta}}^{2}\int_{t_{n}}^{t_{n+1}}\norm{\nabla\bu_{tt}}^{2}dt\nonumber\\
&&+C\nu^{-3}\epsilon^{2}\theta^2\Delta t^{3}\norm{\nabla\bu_{n+\theta}}^{2}\int_{t_{n-1}}^{t_{n+1}}\norm{\nabla\bu_{tt}}^{2}dt,\nonumber
\end{eqnarray}
Similarly, we obtain
\begin{eqnarray}
\lefteqn{Da^{-1}(F_{n+\theta}^{\eps,\nu}(\bu)-\bu_{n+\theta},F_{n+\theta}^{\eps,\nu}(\bphi^{h}))}\nonumber\\
&\leq&\dfrac{\nu}{64}\norm{F_{n+\theta}^{\eps,\nu}(\nabla\bphi^{h})}^{2}+C\nu^{-1}\norm{\nabla(F_{n+\theta}^{\eps,\nu}(\bu)-\bu_{n+\theta})}^{2}\nonumber\\
&\leq&\dfrac{\nu}{64}\norm{F_{n+\theta}^{\eps,\nu}(\nabla\bphi^{h})}^{2}+C\nu^{-1}\theta^{2}(1-\theta^{2})\Delta t^{3}\int_{t_{n}}^{t_{n+1}}\norm{\nabla\bu_{tt}}^{2}dt\nonumber\\
&&+C\nu^{-3}\epsilon^{2}\theta^2\Delta t^{3}\int_{t_{n-1}}^{t_{n+1}}\norm{\nabla\bu_{tt}}^{2}dt,\nonumber
\end{eqnarray}
We proceed to bound the last two terms in $(\ref{E1})$ in a similar manner

\begin{eqnarray}
\lefteqn{\beta_{T}((\bfg H_{n+\theta}(T)-T_{n+\theta}),F_{n+\theta}^{\eps,\nu}(\bphi^{h}))}\nonumber\\
&\leq&\dfrac{\nu}{64}\norm{F_{n+\theta}^{\eps,\nu}(\nabla\bphi^ {h})}^{2}+C\beta_{T}^2\norm{\bfg}_{\infty}^2\nu^{-1}\theta^{2}(1+\theta)^{2}\Delta t^{3}\int_{t_{n-1}}^{t_{n+1}}\norm{\nabla T_{tt}}^{2}dt. \nonumber
\end{eqnarray}
and
\begin{eqnarray}
\lefteqn{\beta_{S}((\bfg H_{n+\theta}(S)-S_{n+\theta}),F_{n+\theta}^{\eps,\nu}(\bphi^{h}))}\nonumber\\
&\leq&\dfrac{\nu}{64}\norm{F_{n+\theta}^{\eps,\nu}(\nabla\bphi^ {h})}^{2}+C\beta_{S}^2\norm{\bfg}_{\infty}^2\nu^{-1}\theta^{2}(1+\theta)^{2}\Delta t^{3}\int_{t_{n-1}}^{t_{n+1}}\norm{\nabla S_{tt}}^{2}dt. \nonumber
\end{eqnarray}
We have completed to bound the terms in $(\ref{E1})$. To bound the remaining terms in right hand side of $(\ref{eqnvel})$, we use Cauchy-Schwarz and Young's inequalities along with Taylor's theorem
\begin{eqnarray}
(F_{n+\theta}^{\eps,\nu}(p)-p_{n+\theta},\nabla\cdot F_{n+\theta}^{\eps,\nu}(\bphi^{h}))
&\leq&\dfrac{\nu}{64}\norm{F_{n+\theta}^{\eps,\nu}(\nabla\bphi^{h})}^{2} +C\nu^{-1}\norm{\theta p_{n+1}+(1-\theta)p_{n}-p_{n+\theta}}^{2}\nonumber\\
&&+C\nu^{-3}\epsilon^{2}\theta^{2}\norm{p_{n+1}-2p_{n}+p_{n-1}}^{2}\nonumber\\
&\leq&\dfrac{\nu}{64}\norm{F_{n+\theta}^{\eps,\nu}(\nabla\bphi^{h})}^{2} +C\nu^{-1}\theta^{2}(1-\theta)^{2}\Delta t^{3}\int_{t_{n}}^{t_{n+1}}\norm{p_{tt}}^{2}dt\nonumber\\
&&+C\nu^{-3}\epsilon^{2}\theta^{2}\Delta t^{3}\int_{t_{n-1}}^{t_{n+1}}\norm{p_{tt}}^{2}dt,\nonumber
\end{eqnarray}

\begin{eqnarray}
(F_{n+\theta}^{\eps,\nu}(p)-q^{h},\nabla\cdot F_{n+\theta}^{\eps,\nu}(\nabla\bphi^{h}))
&\leq&\dfrac{\nu}{64}\norm{F_{n+\theta}^{\eps,\nu}(\nabla\bphi^{h})}^{2} +C\nu^{-1}\bigg(\norm{F_{n+\theta}^{\eps,\nu}(p)-p_{n+\theta}}^{2}+\norm{p_{n+\theta}-q^{h}}^{2}\bigg) \nonumber\\
&\leq&\dfrac{\nu}{64}\norm{F_{n+\theta}^{\eps,\nu}(\nabla\bphi^{h})}^{2} +C\nu^{-1}\theta^{2}(1-\theta)^{2}\Delta t^{3}\int_{t_{n}}^{t_{n+1}}\norm{p_{tt}}^{2}dt\nonumber\\
&& +C\nu^{-3}\epsilon^{2}\theta^{2}\Delta t^{3}\int_{t_{n-1}}^{t_{n+1}}\norm{p_{tt}}^{2}dt +C\nu^{-1}\norm{p_{n+\theta}-q^{h}}^{2},\nonumber
\end{eqnarray}
and
\begin{eqnarray}
\lefteqn{(D_{n+\theta}(\bfeta^{\bu}),F_{n+\theta}^{\eps,\nu}(\bphi^{h}))}\nonumber\\
&\leq& C\norm{\dfrac{\dfrac{1}{2}(\bfeta^{\bu}_{n+1}-\bfeta^{\bu}_{n-1})+\theta(\eta^{\bu}_{n+1}-\bfeta^{\bu}_{n})-\theta(\bfeta^{\bu}_{n}-\bfeta^{\bu}_{n-1})}{\Delta t}}\norm{F_{n+\theta}^{\eps,\nu}(\nabla\bphi^{h})}\nonumber\\
&\leq& C \bigg(\norm{\dfrac{1}{2\Delta t}\int_{t_{n-1}}^{t_{n+1}}\bfeta^{\bu}_{t}dt} +\norm{\dfrac{\theta}{\Delta t}\int_{t_{n}}^{t_{n+1}}\bfeta^{\bu}_{t}dt}+\norm{\dfrac{\theta}{\Delta t}\int_{t_{n-1}}^{t_{n}}\bfeta^{\bu}_{t}dt}\bigg)\norm{F_{n+\theta}^{\eps,\nu}(\nabla\bphi^{h})}\nonumber\\
&\leq& \dfrac{C\nu^{-1}(\theta^{2}+4)}{\Delta t }\int_{t_{n+1}}^{t_{n-1}}\norm{\bfeta^{\bu}_{t}}^{2}dt+\dfrac{\nu}{64}\norm{F_{n+\theta}^{\eps,\nu}(\nabla\bphi^{h})}^{2}.\nonumber
\end{eqnarray}
The next term in (\ref{eqnvel}) is bounded by using Cauchy-Schwarz and Young's inequalities, taking into account the expansion of $F_{n+\theta}^{\eps,\nu}$. It follows that
\begin{eqnarray}
\nu(F_{n+\theta}^{\eps,\nu}(\nabla\bfeta^{\bu}),F_{n+\theta}^{\eps,\nu}(\nabla\bphi^{h}))
&\leq& \nu\norm{F_{n+\theta}^{\eps,\nu}(\nabla\bfeta^{\bu})}\norm{F_{n+\theta}^{\eps,\nu}(\nabla\bphi^{h})}\nonumber\\
&\leq& C\nu \bigg(\big(\theta+\eps\theta\nu^{-1}\big)^{2}\norm{\nabla\bfeta^{\bu}_{n+1}}^{2}+\big(1-\theta-2\epsilon\theta\nu^{-1}\big)^{2}\norm{\nabla\bfeta^{\bu}_{n}}^{2}\nonumber\\
&&+\epsilon^{2}\theta^{2}\nu^{-2}\norm{\nabla\bfeta^{\bu}_{n-1}}^{2}\bigg) +\dfrac{\nu}{64}\norm{F_{n+\theta}^{\eps,\nu}(\nabla\bphi^{h})}^{2}.\nonumber
\end{eqnarray}
Similarly, we have
\begin{eqnarray}
Da^{-1}(F_{n+\theta}^{\eps,\nu}(\bfeta^{\bu}),F_{n+\theta}^{\eps,\nu}(\bphi^{h}))
&\leq& CDa^{-1}\norm{F_{n+\theta}^{\eps,\nu}(\nabla\bfeta^{\bu})}\norm{F_{n+\theta}^{\eps,\nu}(\nabla\bphi^{h})}\nonumber\\
&\leq& C\nu^{-1} \bigg(\big(\theta+\eps\theta\nu^{-1}\big)^{2}\norm{\nabla\bfeta^{\bu}_{n+1}}^{2}+\big(1-\theta-2\epsilon\theta\nu^{-1}\big)^{2}\norm{\nabla\bfeta^{\bu}_{n}}^{2}\nonumber\\
&&+\epsilon^{2}\theta^{2}\nu^{-2}\norm{\nabla\bfeta^{\bu}_{n-1}}^{2}\bigg) +\dfrac{\nu}{64}\norm{F_{n+\theta}^{\eps,\nu}(\nabla\bphi^{h})}^{2}.\nonumber
\end{eqnarray}
Applying estimation (\ref{nterm1}) for the nonlinear terms and expansion of the operators along with Cauchy-Schwarz and Young's inequalities leads to
\begin{eqnarray}
\lefteqn{b^{*}(H_{n+\theta}(\bu),F_{n+\theta}^{\eps,\nu}(\bfeta^{\bu}),F_{n+\theta}^{\eps,\nu}(\bphi^{h}))}\nonumber\\
&\leq& C\norm{H_{n+\theta}(\nabla\bu)}\norm{F_{n+\theta}^{\eps,\nu}(\nabla\bfeta^{\bu})}\norm{F_{n+\theta}^{\eps,\nu}(\nabla\bphi^{h})}\nonumber\\
&\leq& C\nu^{-1}\bigg((\theta+1)^{2}\norm{\nabla\bu_{n}}^{2}+\theta^{2}\norm{\nabla\bu_{n-1}}^{2}\bigg)\bigg(\big(\theta+\epsilon\theta\nu^{-1}\big)^{2}\norm{\nabla\bfeta^{\bu}_{n+1}}^{2}\nonumber\\
&&+\big(1-\theta-2\epsilon\theta\nu^{-1}\big)^{2}\norm{\nabla\bfeta^{\bu}_{n}}^{2}+\epsilon^{2}\theta^{2}\nu^{-2}\norm{\nabla\bu_{n-1}}^{2}\bigg)+\dfrac{\nu}{64}\norm{F_{n+\theta}^{\eps,\nu}(\nabla\bphi^{h})}^{2} \nonumber
\end{eqnarray}
and
\begin{eqnarray}
\lefteqn{b^{*}(H_{n+\theta}(\bfeta^{\bu}),F_{n+\theta}^{\eps,\nu }(\bu^{h}),F_{n+\theta}^{\eps,\nu}(\bphi^{h}))}\nonumber\\
&\leq& C\nu^{-1}\norm{F_{n+\theta}^{\eps,\nu}(\nabla\bu^{h})}^{2}\norm{H_{n+\theta}(\nabla\bfeta^{\bu})}^{2}+\dfrac{\nu}{64}\norm{F_{n+\theta}^{\eps,\nu}(\nabla\bphi^{h})}^{2} \nonumber\\
&\leq& C\nu^{-1} \bigg(\big(\theta+\epsilon\theta\nu^{-1}\big)^{2}\norm{\nabla{\bu}_{n+1}^{h}}^{2}+\big(1-\theta-2\epsilon\theta\nu^{-1}\big)^{2}\norm{\nabla\bu_{n}^{h}}^{2}+\epsilon^{2}\theta^{2}\nu^{-2}\norm{\nabla\bu_{n-1}^{h}}^{2}\bigg)\nonumber\\
&&\times\big((\theta+1)^2\norm{\nabla\bfeta^{\bu}_{n}}^{2}+\theta^2\norm{\nabla\bfeta^{\bu}_{n-1}}^2\big)
+\dfrac{\nu}{64}\norm{F_{n+\theta}^{\eps,\nu}(\nabla\bphi^{h})}^{2}. \nonumber
\end{eqnarray}
Similarly, to bound the next nonlinear term, with the help of (\ref{infb}), one gets
\begin{eqnarray}
\lefteqn{b^{*}(H_{n+\theta}(\bphi^{h}),F_{n+\theta}^{\eps,\nu}(\bu^{h}),F_{n+\theta}^{\eps,\nu}(\bphi^{h}))}\nonumber\\
&\leq& C\norm{F_{n+\theta}^{\eps,\nu}(\nabla \bu^{h})}_{\infty}\norm{H_{n+\theta}( \bphi^{h})}\norm{F_{n+\theta}^{\eps,\nu}(\nabla\bphi^{h})} +C\norm{F_{n+\theta}^{\eps,\nu}(\bu^{h})}_{\infty}\norm{H_{n+\theta}( \bphi^{h})}\norm{F_{n+\theta}^{\eps,\nu}(\nabla\bphi^{h})}\nonumber\\
&\leq& C\nu^{-1}\bigg(\norm{F_{n+\theta}^{\eps,\nu}(\nabla \bu^{h})}_{\infty}^2+\norm{F_{n+\theta}^{\eps,\nu}(\bu^{h})}_{\infty}^2\bigg) \bigg((\theta+1)^2\norm{\bphi^{h}_{n}}^2+\theta^2\norm{\bphi^{h}_{n-1}}^2\bigg)+\dfrac{\nu}{64}\norm{F_{n+\theta}^{\eps,\nu}(\nabla\bphi^{h})}^{2}. \nonumber
\end{eqnarray}
In a similar manner, the last four terms in (\ref{eqnvel}) are bounded by
\begin{eqnarray}
\beta_{T}(\bfg H_{n+\theta}(\xi^{h}),F_{n+\theta}^{\eps,\nu}(\bphi^{h}))
&\leq&\dfrac{\nu}{64}\norm{F_{n+\theta}^{\eps,\nu}(\nabla\bphi^{h})}^{2}+C\beta_{T}^2\norm{\bfg}_{\infty}^2\nu^{-1}\bigg((\theta+1)^2\norm{\xi^{h}_{n}}^2+\theta^2\norm{\xi^{h}_{n-1}}^2\bigg),\nonumber
\end{eqnarray}
\begin{eqnarray}
\beta_{T}(\bfg H_{n+\theta}(\eta^{T}),F_{n+\theta}^{\eps,\nu}(\bphi^{h}))
&\leq&\dfrac{\nu}{64}\norm{F_{n+\theta}^{\eps,\nu}(\nabla\bphi^{h})}^{2}+C\beta_{T}^2\norm{\bfg}_{\infty}^2\nu^{-1}\bigg((\theta+1)^2\norm{\nabla\eta^{T}_{n}}^2+\theta^2\norm{\nabla\eta^{T}_{n-1}}^2\bigg),\nonumber
\end{eqnarray}
\begin{eqnarray}
\beta_{S}(\bfg H_{n+\theta}(\zeta^{h}),F_{n+\theta}^{\eps,\nu}(\bphi^{h}))
&\leq&\dfrac{\nu}{64}\norm{F_{n+\theta}^{\eps,\nu}(\nabla\bphi^{h})}^{2}+C\beta_{S}^2\norm{\bfg}_{\infty}^2\nu^{-1}\bigg((\theta+1)^2\norm{\zeta^{h}_{n}}^2+\theta^2\norm{\zeta^{h}_{n-1}}^2\bigg),\nonumber
\end{eqnarray}
\begin{eqnarray}
\beta_{S}(\bfg H_{n+\theta}(\eta^{S}),F_{n+\theta}^{\eps,\nu}(\bphi^{h}))
&\leq&\dfrac{\nu}{64}\norm{F_{n+\theta}^{\eps,\nu}(\nabla\bphi^{h})}^{2}+C\beta_{S}^2\norm{\bfg}_{\infty}^2\nu^{-1}\bigg((\theta+1)^2\norm{\nabla\eta^{S}_{n}}^2+\theta^2\norm{\nabla\eta^{S}_{n-1}}^2\bigg).\nonumber
\end{eqnarray}
Next insert the above bounds into the $(\ref{errvel})$, use the approximation property (\ref{ap4}), multiply by $\Delta t$ and take the sum from $n=1$ to $n=N-1$ ;
\begin{eqnarray}
\lefteqn{\norm{ \begin{bmatrix}
	\bphi^{h}_{N} \\
	\bphi^{h}_{N-1}
	\end{bmatrix}} _{G}^{2}  	+\dfrac{1}{4}\sum_{n=1}^{N-1} \norm{\bphi^{h}_{n+1}-2\bphi^{h}_{n}+\bphi^{h}_{n-1}}_{F}^{2} + \dfrac{ \Delta t\nu}{2} \sum_{n=1}^{N-1}\norm{{F}_{n+\theta}^{\eps,\nu}(\nabla \bphi^{h})}^{2}+ \Delta tDa^{-1} \sum_{n=1}^{N-1}\norm{{F}_{n+\theta}^{\eps,\nu}( \bphi^{h})}^{2}} \nonumber\\
&\leq& \norm{ \begin{bmatrix}
	\bphi^{h}_1 \\
	\bphi^{h}_{0}
	\end{bmatrix}} _{G}^{2} +C\Delta t\sum_{n=1}^{N-1}\bigg[\nu^{-1}\theta^2(1-\theta)^2\Delta t^3\int_{t_{n}}^{t_{n+1}}\norm{p_{tt}}^2dt +\nu^{-3}\epsilon^2\theta^2\Delta t^3\int_{t_{n-1}}^{t_{n+1}}\norm{p_{tt}}^2dt\nonumber\\
&&+\nu^{-1}h^{2k+2}\norm{p_{n+\theta}}_{k+1}^2
+\nu^{-1}\theta^6\Delta t^3\int_{t_{n-1}}^{t_{n+1}}\norm{\bu_{ttt}}^2dt
 +C\nu\theta^2(1-\theta)^2\Delta t^3\int_{t_{n}}^{t_{n+1}}\norm{\nabla\bu_{tt}}^2dt\nonumber\\
 &&+C\nu^{-1}\epsilon^2\theta^2\Delta t^3\int_{t_{n-1}}^{t_{n+1}}\norm{\nabla\bu_{tt}}^2dt
 +C\nu^{-1}\theta^3(1+\theta)^2\Delta t^3\norm{F_{n+\theta}^{\eps,\nu}(\nabla\bu)}^2\int_{t_{n-1}}^{t_{n+1}}\norm{\nabla\bu_{tt}}^2dt\nonumber\\
 &&+C\nu^{-1}\theta^2(1-\theta)^2\Delta t^3\norm{\nabla\bu_{n+\theta}}^2\int_{t_{n}}^{t_{n+1}}\norm{\nabla\bu_{tt}}^2dt
+C\nu^{-3}\epsilon^2\theta^2\Delta t^3\norm{\nabla\bu_{n+\theta}}^2\int_{t_{n-1}}^{t_{n+1}}\norm{\nabla\bu_{tt}}^2dt\nonumber\\
&&+C\nu^{-1}\theta^{2}(1-\theta^{2})\Delta t^{3}\int_{t_{n}}^{t_{n+1}}\norm{\nabla\bu_{tt}}^{2}dt
+C\nu^{-3}\epsilon^{2}\theta^2\Delta t^{3}\int_{t_{n-1}}^{t_{n+1}}\norm{\nabla\bu_{tt}}^{2}dt\nonumber\\
&&+C\beta_{T}^2\norm{\bfg}_{\infty}^2\nu^{-1}\theta^2(1+\theta)^2\Delta t^3\int_{t_{n-1}}^{t_{n+1}}\norm{\nabla T_{tt}}^2dt+C\beta_{S}^2\norm{\bfg}_{\infty}^2\nu^{-1}\theta^2(1+\theta)^2\Delta t^3\int_{t_{n-1}}^{t_{n+1}}\norm{\nabla S_{tt}}^2dt\nonumber\\
  &&+\dfrac{(\theta^2+4)\nu^{-1}}{\Delta t}\int_{t_{n-1}}^{t_{n+1}}\norm{\bfeta^{\bu}_{t}}^2dt+C\nu \bigg(\big(\theta+\eps\theta\nu^{-1}\big)^{2}\norm{\nabla\bfeta^{\bu}_{n+1}}^{2}+\big(1-\theta-2\epsilon\theta\nu^{-1}\big)^{2}\norm{\nabla\bfeta^{\bu}_{n}}^{2}\nonumber\\
&&+\epsilon^{2}\theta^{2}\nu^{-2}\norm{\nabla\bfeta^{\bu}_{n-1}}^{2}\bigg)+\nu^{-1}\bigg((\theta+\epsilon \theta \nu^{-1})^2\norm{\nabla\bfeta^{\bu}_{n+1}}^2+(1-\theta-2\epsilon\theta\nu^{-1})^2\norm{\nabla\bfeta^{\bu}_{n}}^2\nonumber\\
&&+\epsilon^2\theta^2\nu^{-2}\norm{\nabla\bfeta^{\bu}_{n-1}}^2\bigg)+\nu^{-1}\bigg((\theta+1)^2\norm{\nabla\bu_{n}}^2+\theta^2\norm{\nabla\bu_{n-1}}^2\bigg)\bigg((\theta+\epsilon\theta\nu^{-1})^2\norm{\nabla\bfeta^{\bu}_{n+1}}^2\nonumber\\
&&+(1-\theta-2\epsilon\theta\nu^{-1})^2\norm{\nabla\bfeta^{\bu}_{n}}^2+\epsilon^2\theta^2\nu^{-2}\norm{\nabla\bfeta^{\bu}_{n-1}}^2\bigg)+\nu^{-1}\bigg((\theta+\epsilon\theta\nu^{-1})^2\norm{\nabla\bu_{n+1}^{h}}^2\nonumber
\end{eqnarray}
\begin{eqnarray}
&&+(1-\theta-2\epsilon\theta\nu^{-1})^2\norm{\nabla\bu_{n}^{h}}^2+\epsilon^2\theta^2\nu^{-2}\norm{\nabla\bu_{n-1}^{h}}^2\bigg)\big((\theta+1)^2\norm{\nabla\bfeta^{\bu}_{n}}^2+\theta^2\norm{\nabla\bfeta^{\bu}_{n-1}}^2\big)\nonumber\\
&&+C\beta_{T}^2\norm{\bfg}_{\infty}^2\nu^{-1}\bigg((\theta+1)^2\norm{\nabla\eta^{T}_{n}}^2+\theta^2\norm{\nabla\eta^{T}_{n-1}}^2\bigg)+C\beta_{S}^2\norm{\bfg}_{\infty}^2\nu^{-1}\bigg((\theta+1)^2\norm{\nabla\eta^{S}_{n}}^2\nonumber\\
&&+\theta^2\norm{\nabla\eta^{S}_{n-1}}^2\bigg)+\nu^{-1}\bigg(\norm{F_{n+\theta}^{\eps,\nu}(\nabla\bu^{h})}_{\infty}^{2}+\norm{F_{n+\theta}^{\eps,\nu}(\bu^{h})}_{\infty}^{2}\bigg)\bigg((\theta+1)^2\norm{\bphi^{h}_{n}}^2+\theta^2\norm{\bphi^{h}_{n-1}}^2\bigg)\nonumber\\
&&+C\beta_{T}^2\norm{\bfg}_{\infty}^2\nu^{-1}\bigg((\theta+1)^2\norm{\xi^{h}_{n}}^2+\theta^2\norm{\xi^{h}_{n-1}}^2\bigg)\nonumber\\
&&+C\beta_{S}^2\norm{\bfg}_{\infty}^2\nu^{-1}\bigg((\theta+1)^2\norm{\zeta^{h}_{n}}^2+\theta^2\norm{\zeta^{h}_{n-1}}^2\bigg)\bigg].\nonumber
\end{eqnarray}
Next we observe that due to Lemma $\ref{lem:gnorm}$ and approximation results (\ref{ap1}), (\ref{ap2}) and (\ref{ap3}) we have;
\begin{eqnarray}
\lefteqn{\norm{\bphi^{h}_{N}}^2+\dfrac{1}{2\theta+1}\sum_{n=1}^{N-1} \norm{\bphi^{h}_{n+1}-2\bphi^{h}_{n}+\bphi^{h}_{n-1}}_{F}^{2}+ \dfrac{ 2\Delta t\nu}{2\theta+1} \sum_{n=1}^{N-1}\norm{{F}_{n+\theta}^{\eps,\nu}(\nabla \bphi^{h})}^{2}+\dfrac{4Da^{-1}\Delta t}{2\theta+1}\sum_{n=1}^{N-1}\norm{{F}_{n+\theta}^{\eps,\nu}( \bphi^{h})}^{2}}\nonumber\\
&\leq&\big(\dfrac{2\theta-1}{2\theta+1}\big)^{N}\norm{\bphi^{h}_{0}}^2+2\bigg(1-\big(\dfrac{2\theta-1}{2\theta+1}\big)^{N}\bigg)\Bigg[\norm{ \begin{bmatrix}
	\bphi ^{h}_1 \\
	\bphi^{h}_{0}
	\end{bmatrix}} _{G}^{2} +C\Big(\nu^{-1}\Delta t^4 \norm{|p_{tt}|}_{2,0}^{2}+\nu^{-1}h^{2k+2}\norm{|p|}_{2,k+1}^{2}\nonumber\\
	&&+\nu^{-1}\Delta t^4\norm{|\bu_{ttt}|}_{2,0}^{2}
	+\nu\Delta t^4\norm{|\nabla\bu_{tt}|}_{2,0}^{2}+\nu^{-1}\Delta t^4\norm{|\nabla\bu_{tt}|}_{2,0}^{2}+\nu^{-1}\Delta t^4\norm{|\nabla\bu|}_{\infty,0}^{2}\norm{|\nabla\bu_{tt}|}_{2,0}^{2}\nonumber\\
&&+\nu^{-1}\beta_{T}^2\norm{\bfg}_{\infty}^2\Delta t^4\norm{|\nabla T_{tt}|}_{2,0}^{2}+\nu^{-1}\beta_{S}^2\norm{\bfg}_{\infty}^2\Delta t^4\norm{|\nabla S_{tt}|}_{2,0}^{2}+\nu^{-1}h^{2k+2}\norm{|\bu_{t}|}_{2,k+1}^{2}\nonumber\\
&&+\nu h^{2k}\norm{|\bu|}_{2,k+1}^{2}+\nu^{-1} h^{2k}\norm{|\bu|}_{2,k+1}^{2}+\nu^{-1}h^{2k}\norm{|\nabla \bu|}_{\infty}^{2}\norm{|\bu|}_{2,k+1}^{2}+\nu^{-1}h^{2k}\beta_{T}^2\norm{\bfg}_{\infty}^2\norm{|T|}_{2,k+1}^{2}\nonumber\\
&&+\nu^{-1}h^{2k}\beta_{S}^2\norm{\bfg}_{\infty}^2\norm{|S|}_{2,k+1}^{2}\Big)+C\nu^{-1}\Delta t\sum_{n=0}^{N-1}\norm{\bphi^{h}_{n}}^2+C\nu^{-1}\beta_{T}^2\norm{\bfg}_{\infty}^2\Delta t \sum_{n=0}^{N-1}\norm{\xi^{h}_{n}}^2\nonumber\\
&&+C\nu^{-1}\beta_{S}^2\norm{\bfg}_{\infty}^2\Delta t \sum_{n=0}^{N-1}\norm{\zeta^{h}_{n}}^2\Bigg] .\label{velbound}
\end{eqnarray}
The proof of temperature proceeds along the lines of the velocity error estimation.
The first term  $E_2(\bu_{n+\theta},T_{n+\theta};F_{n+\theta}^{\eps_1,\gamma}(\xi^{h}))$ in (\ref{eqntem}) is bounded by  using Cauchy-Schwarz, Young's inequalities, expansion of operators and Taylor's theorem. Then, one gets
\begin{eqnarray}
(D_{n+\theta}(T)-T_{t}(t_{n+\theta}),F_{n+\theta}^{\eps_1,\gamma}(\xi^{h}))
&\leq&\dfrac{\gamma}{64}\norm{F_{n+\theta}^{\eps_1,\gamma}(\nabla\xi^{h})}^{2}+C\gamma^{-1}\theta^{6}\Delta t^{3}\int_{t_{n-1}}^{t_{n+1}}\norm{T_{ttt}}^{2}dt \nonumber
\end{eqnarray}
and
\begin{eqnarray}
\gamma(\nabla(F_{n+\theta}^{\eps_1,\gamma}(T)-T_{n+\theta}),\nabla F_{n+\theta}^{\eps_1,\gamma}(\xi^{h}))
&\leq&\dfrac{\gamma}{64}\norm{F_{n+\theta}^{\eps_1,\gamma}(\nabla\xi^{h})}^{2}+C\gamma\theta^{2}(1-\theta)^{2}\Delta t^{3}\int_{t_{n}}^{t_{n+1}}\norm{\nabla T_{tt}}^{2}dt\nonumber\\
&&+C\gamma^{-1}\epsilon_1^{2}\theta^{2}\Delta t^{3}\int_{t_{n-1}}^{t_{n+1}}\norm{\nabla T_{tt}}^{2}dt.\nonumber
\end{eqnarray}
The trilinear terms are bounded similar as in the velocity case
\begin{eqnarray}
\lefteqn{c^{*}(H_{n+\theta}(\bu)-\bu_{n+\theta},F_{n+\theta}^{\eps_1,\gamma}(T),F_{n+\theta}^{\eps_1,\gamma}(\xi^{h}))}\nonumber\\
&\leq&\dfrac{\gamma}{64}\norm{F_{n+\theta}^{\eps_1,\gamma}(\nabla\xi^{h})}^{2}+C\gamma^{-1}\theta^{2}(1+\theta)^2\Delta t^{3}\norm{F_{n+\theta}^{\eps_1,\gamma}(\nabla T)}^{2}\int_{t_{n-1}}^{t_{n+1}}\norm{\nabla \bu_{tt}}^{2}dt \nonumber
\end{eqnarray}
and
\begin{eqnarray}
\lefteqn{c^{*}(\bu_{n+\theta},F_{n+\theta}^{\eps_1,\gamma}(T)-T_{n+\theta},F_{n+\theta}^{\eps_1,\gamma}(\xi^{h}))}\nonumber\\
&\leq&\dfrac{\gamma}{64}\norm{F_{n+\theta}^{\eps_1,\gamma}(\nabla\xi^{h})}^{2}+C\gamma^{-1}\theta^{2}(1-\theta)^{2}\Delta t^3\norm{\nabla \bu_{n+\theta}}^{2}\int_{t_{n}}^{t_{n+1}}\norm{\nabla T_{tt}}^{2}dt\nonumber\\
&&+C\gamma^{-3}\epsilon_1^{2}\Delta t^{3}\norm{\nabla \bu_{n+\theta}}^{2}\int_{t_{n-1}}^{t_{n+1}}\norm{\nabla T_{tt}}^{2}dt.\nonumber
\end{eqnarray}
In a similar manner, the remainder terms in (\ref{eqntem}) follows analogously the proof of the velocity . One gets the bound
\begin{eqnarray}
(D_{n+\theta}(\eta^{T}),F_{n+\theta}^{\eps_1,\gamma}(\xi^{h}))
&\leq&\dfrac{C\gamma^{-1}(\theta^{2}+4)}{\Delta t }\int_{t_{n-1}}^{t_{n+1}}\norm{\eta^{T}_{t}}^{2}dt+\dfrac{\gamma}{64}\norm{F_{n+\theta}^{\eps_1,\gamma}(\nabla\xi^{h})}^{2}\nonumber
\end{eqnarray}
and the following bound for the viscous term
\begin{eqnarray}
\lefteqn{\gamma(F_{n+\theta}^{\eps_1,\gamma}(\nabla\eta^{T}),F_{n+\theta}^{\eps_1,\gamma}(\nabla\xi^{h}))}\nonumber\\
&\leq& C\gamma \bigg(\big(\theta+\dfrac{\eps_1\theta}{\gamma}\big)^{2}\norm{\nabla\eta^{T}_{n+1}}^{2}+\big(1-\theta-\dfrac{2\epsilon_1\theta}{\gamma}\big)^{2}\norm{\nabla\eta^{T}_{n}}^{2}+\dfrac{\epsilon_1^{2}\theta^{2}}{\gamma^{2}}\norm{\nabla\eta^{T}_{n-1}}^{2}\bigg) +\dfrac{\gamma}{64}\norm{F_{n+\theta}^{\eps_1,\gamma}(\nabla\xi^{h})}^{2}.\nonumber
\end{eqnarray}
Applying Cauchy-Schwarz, expansion of the operators along with the Young's inequality yields
\begin{eqnarray}
\lefteqn{c^{*}(H_{n+\theta}(\bu),F_{n+\theta}^{\eps_1,\gamma}(\eta^{T}),F_{n+\theta}^{\eps_1,\gamma}(\xi^{h}))}\nonumber\\
&\leq& C\gamma^{-1}\bigg((\theta+1)^{2}\norm{\nabla \bu_{n}}^{2}+\theta^{2}\norm{\nabla \bu_{n-1}}^{2}\bigg)\bigg(\big(\theta+\dfrac{\epsilon_1\theta}{\gamma}\big)^{2}\norm{\nabla\eta^{T}_{n+1}}^{2}+\big(1-\theta-\dfrac{2\epsilon_1\theta}{\gamma}\big)^{2}\norm{\nabla\eta^{T}_{n}}^{2}\nonumber\\
&&+\dfrac{\epsilon_1^{2}\theta^{2}}{\gamma^{2}}\norm{\nabla\eta^ T_ {n-1}}^{2}\bigg)+\dfrac{\gamma}{64}\norm{F_{n+\theta}^{\eps_1,\gamma}(\nabla\xi^{h})}^{2} \nonumber
\end{eqnarray}
and
\begin{eqnarray}
\lefteqn{c^{*}(H_{n+\theta}(\bfeta^{\bu}),F_{n+\theta}^{\eps_1,\gamma }(T^{h}),F_{n+\theta}^{\eps_1,\gamma}(\phi_{T}))}\nonumber\\
&\leq& C\gamma^{-1} \bigg(\big(\theta+\epsilon_1\theta\kappa^{-1}\big)^{2}\norm{\nabla{T}_{n+1}^{h}}^{2}+\big(1-\theta-2\epsilon_1\theta\gamma^{-1}\big)^{2}\norm{\nabla T_{n}^{h}}^{2}+\epsilon_1^{2}\theta^{2}\gamma^{-2}\norm{\nabla T_{n-1}^{h}}^{2}\bigg)\nonumber\\
&&\times\big((\theta+1)^2\norm{\nabla\bfeta^{\bu}_{n}}^{2}+\theta^2\norm{\nabla\bfeta^{\bu}_{n-1}}^2\big)+\dfrac{\gamma}{64}\norm{F_{n+\theta}^{\eps_1,\gamma}(\nabla\phi_{T})}^{2}. \nonumber
\end{eqnarray}
Finally, the last trilinear term is bounded by Lemma \ref{sb}:
\begin{eqnarray}
\lefteqn{c^{*}(H_{n+\theta}(\phi^{h}),F_{n+\theta}^{\eps_1,\gamma}(T^{h}),F_{n+\theta}^{\eps_1,\gamma}(\xi^{h}))}\nonumber\\
&\leq& C\gamma^{-1}\bigg(\norm{F_{n+\theta}^{\eps_1,\gamma}(\nabla T^{h})}_{\infty}^2+\norm{F_{n+\theta}^{\eps_1,\gamma}(T^{h})}_{\infty}^2\bigg) \big((\theta+1)^2\norm{\bphi^{h}_{n}}^2+\theta^2\norm{\bphi^{h}_{n-1}}^2\big)+\dfrac{\gamma}{64}\norm{F_{n+\theta}^{\eps_1,\gamma}(\nabla\bphi^{h})}^{2}. \nonumber
\end{eqnarray}
Next insert the above bounds into the $(\ref{errtem})$  and take the sum from $n=1$ to $n=N-1$ :
\begin{eqnarray}
\lefteqn{\norm{ \begin{bmatrix}
	\xi^{h}_{N} \\
	\xi^{h}_{N-1}
	\end{bmatrix}} _{G}^{2}  	+\dfrac{1}{4}\sum_{n=1}^{N-1} \norm{\xi^{h}_{n+1}-2\xi_{n}^{h}+\xi^{h}_{n-1}}_{F}^{2} + \dfrac{ \Delta t\gamma}{2} \sum_{n=1}^{N-1}\norm{{F}_{n+\theta}^{\eps_1,\gamma}(\nabla \xi^{h})}^{2}} \nonumber
\\
&\leq&\norm{ \begin{bmatrix}
	\xi^{h}_1 \\
	\xi^{h}_{0}
	\end{bmatrix}} _{G}^{2} +C\Delta t\sum_{n=1}^{N-1}\bigg[
 \gamma^{-1}\theta^6\Delta t^3\int_{t_{n-1}}^{t_{n+1}}\norm{T_{ttt}}^2dt+\gamma\theta^2(1-\theta)^2\Delta t^3\int_{t_{n}}^{t_{n+1}}\norm{\nabla T_{tt}}^2dt\nonumber\\
&&+\gamma^{-1}\epsilon_1^2\theta^2\Delta t^3\int_{t_{n-1}}^{t_{n+1}}\norm{\nabla T_{tt}}^2dt+C\gamma^{-1}\theta^2(1+\theta)^2\Delta t^3\norm{F_{n+\theta}^{\eps_1,\gamma}(\nabla T)}^2\int_{t_{n-1}}^{t_{n+1}}\norm{\nabla \bu_{tt}}^2dt\nonumber\\
&&+C\gamma^{-1}\theta^2(1-\theta)^2\Delta t ^3\norm{\nabla \bu_{n+\theta}}^2\int_{t_{n}}^{t_{n+1}}\norm{\nabla T_{tt}}^2dt+C\gamma^{-3}\epsilon_1^2\theta^2\Delta t^2\norm{\nabla \bu_{n+\theta}}^2\int_{t_{n-1}}^{t_{n+1}}\norm{\nabla T_{tt}}^2dt\nonumber\\
&& +\dfrac{(\theta^2+4)\gamma^{-1}}{\Delta t }\int_{t_{n-1}}^{t_{n+1}}\norm{\eta^{T}_{t}}^2dt+\gamma\bigg((\theta+\epsilon_1\theta\gamma^{-1})^2\norm{\nabla\eta^{T}_{n+1}}^2+(1-\theta-2\epsilon_1\theta\gamma^{-1})^2\norm{\nabla\eta^{T}_{n}}^2\nonumber\\
&&+\epsilon_1^2\theta^2\gamma^{-2}\norm{\nabla\eta^{T}_{n-1}}^2\bigg)+\gamma^{-1}\bigg((\theta+1)^2\norm{\nabla \bu_{n}}^2+\theta^2\norm{\nabla \bu_{n-1}}^2\bigg)\bigg((\theta+\epsilon_1\theta\gamma^{-1})^2\norm{\nabla\eta^{T}_{n+1}}^2\nonumber\\
&&+(1-\theta-2\epsilon_1\theta\gamma^{-1})^2\norm{\nabla\eta^{T}_{n}}^2+\epsilon_1^2\theta^2\gamma^{-2}\norm{\nabla\eta^{T}_{n-1}}^2\bigg)+\gamma^{-1}\bigg((\theta+\epsilon_1\theta\gamma^{-1})^2\norm{\nabla T_{n+1}^{h}}^2\nonumber\\
&&+(1-\theta-2\epsilon_1\theta\gamma^{-1})^2\norm{\nabla T_{n}^{h}}^2+\epsilon_1^2\theta^2\gamma^{-2}\norm{\nabla T_{n-1}^{h}}^2\bigg)\bigg((\theta+1)^2\norm{\nabla\bfeta^{\bu}_{n}}^2+\theta^2\norm{\nabla\bfeta^{\bu}_{n-1}}^2\bigg)\nonumber\\
&&+\gamma^{-1}\bigg(\norm{F_{n+\theta}^{\eps_1,\gamma}(\nabla T^{h})}_{\infty}^{2}+\norm{F_{n+\theta}^{\eps_1,\gamma}(T^{h})}_{\infty}^{2}\bigg)\bigg((\theta+1)^2\norm{\bphi^{\bu}_{n}}^2+\theta^2\norm{\bphi^{\bu}_{n-1}}^2\bigg)\bigg].\nonumber
\end{eqnarray}
Applying the stability bound  Lemma $\ref{lem:gnorm}$ in the last estimation yields;
\begin{eqnarray}
\lefteqn{\norm{\xi_{N}^{h}}^2+\dfrac{1}{2\theta+1}\sum_{n=1}^{N-1} \norm{\xi^{h}_{n+1}-2\xi^{h}_{n}+\xi^{h}_{n-1}}_{F}^{2}+\dfrac{ 2\Delta t\gamma}{2\theta+1} \sum_{n=1}^{N-1}\norm{{F}_{n+\theta}^{\eps_1,\gamma}(\nabla \xi^{h})}^{2}}\nonumber\\
&\leq&\bigg(\dfrac{2\theta-1}{2\theta+1}\bigg)^{N}\norm{\xi^{h}_{0}}^2+2\bigg(1-\bigg(\dfrac{2\theta-1}{2\theta+1}\bigg)^{N}\bigg)\Bigg[\norm{ \begin{bmatrix}
	\xi _{1}^h \\
	\xi_{0}^{h}
	\end{bmatrix}} _{G}^{2} \nonumber\\
&&+C\Big(
\gamma^{-1}\Delta t^4\norm{|T_{ttt}|}_{2,0}^{2}+\gamma\Delta t^4\norm{|\nabla T_{tt}|}_{2,0}^{2}+\gamma^{-1}\Delta t^4\norm{|\nabla T_{tt}|}_{2,0}^{2}+\gamma^{-1}\Delta t^4\norm{\nabla T}_{\infty}^{2}\norm{|\nabla  \bu_{tt}|}_{2,0}^{2}\nonumber\\
&&+\gamma^{-1}\Delta t^4\norm{|\nabla\bu |}_{\infty}^{2}\norm{|\nabla T_{tt}|}_{2,0}^{2}+\gamma^{-1}h^{2k+2}\norm{|T_{t}|}_{2,k+1}^{2}+\gamma h^{2k}\norm{|T|}_{2,k+1}^{2}\nonumber\\ &&+\gamma^{-1}h^{2k}\norm{|\nabla \bu|}_{\infty}^{2}\norm{|T|}_{2,k+1}^{2}+\gamma^{-1}h^{2k}\norm{|\nabla T|}_{\infty}^{2}\norm{|\bu|}_{2,k+1}^{2}\Big)+C\gamma^{-1}\Delta t\sum_{n=0}^{N-1}\norm{\bphi^{h}_{n}}^2\Bigg]. \label{tembound}
\end{eqnarray}
Repeating the similar arguments of the temperature error, the concentration error equation (\ref{eqncont}) is estimated by
\begin{eqnarray}
\lefteqn{\norm{\zeta_{N}^{h}}^2+\dfrac{1}{2\theta+1}\sum_{n=1}^{N-1} \norm{\zeta^{h}_{n+1}-2\zeta^{h}_{n}+\zeta^{h}_{n-1}}_{F}^{2}+\dfrac{ 2\Delta t D_c}{2\theta+1} \sum_{n=1}^{N-1}\norm{{F}_{n+\theta}^{\eps_2,D_c}(\nabla \zeta^{h})}^{2}}\nonumber\\
&\leq&\bigg(\dfrac{2\theta-1}{2\theta+1}\bigg)^{N}\norm{\zeta^{h}_{0}}^2+2\bigg(1-\big(\dfrac{2\theta-1}{2\theta+1}\big)^{N}\bigg)\Bigg[\norm{ \begin{bmatrix}
	\zeta _{1}^h \\
	\zeta_{0}^{h}
	\end{bmatrix}} _{G}^{2} \nonumber\\
&&+C\Big(
D_c^{-1}\Delta t^4\norm{|S_{ttt}|}_{2,0}^{2}+D_c\Delta t^4\norm{|\nabla S_{tt}|}_{2,0}^{2}+D_c^{-1}\Delta t^4\norm{|\nabla S_{tt}|}_{2,0}^{2}+D_c^{-1}\Delta t^4\norm{\nabla S}_{\infty}^{2}\norm{|\nabla  \bu_{tt}|}_{2,0}^{2}\nonumber
\end{eqnarray}
\begin{eqnarray}
&&+D_c^{-1}\Delta t^4\norm{|\nabla\bu |}_{\infty}^{2}\norm{|\nabla S_{tt}|}_{2,0}^{2}+D_c^{-1}h^{2k+2}\norm{|S_{t}|}_{2,k+1}^{2}+D_c h^{2k}\norm{|S|}_{2,k+1}^{2}\nonumber\\ &&+D_c^{-1}h^{2k}\norm{|\nabla \bu|}_{\infty}^{2}\norm{|S|}_{2,k+1}^{2}+D_c^{-1}h^{2k}\norm{|\nabla S|}_{\infty}^{2}\norm{|\bu|}_{2,k+1}^{2}\Big)+C D_c^{-1}\Delta t\sum_{n=0}^{N-1}\norm{\bphi^{h}_{n}}^2\Bigg]. \label{contbound}
\end{eqnarray}
Now, summing $(\ref{velbound})$, (\ref{tembound}) and (\ref{contbound}) we obtain
\begin{eqnarray}
\lefteqn{\norm{\bphi^{h}_{N}}^2+\norm{\xi^{h}_{N}}^2+\norm{\zeta^{h}_{N}}^2+\dfrac{1}{2\theta+1}\sum_{n=1}^{N-1} \norm{\bphi^{h}_{n+1}-2\bphi^{h}_{n}+\bphi^{h}_{n-1}}_{F}^{2}+ \dfrac{ 2\Delta t\nu}{2\theta+1} \sum_{n=1}^{N-1}\norm{{F}_{n+\theta}^{\eps,\nu}(\nabla \bphi^{h})}^{2}}\nonumber\\
\lefteqn{+\dfrac{ 4\Delta tDa^{-1}}{2\theta+1} \sum_{n=1}^{N-1}\norm{{F}_{n+\theta}^{\eps,\nu}( \bphi^{h})}^{2}+\dfrac{1}{2\theta+1}\sum_{n=1}^{N-1} \norm{\xi^{h}_{n+1}-2\xi^{h}_{n}+\xi^{h}_{n-1}}_{F}^{2} + \dfrac{ 2\Delta t\gamma}{2\theta+1} \sum_{n=1}^{N-1}\norm{{F}_{n+\theta}^{\eps,\gamma}(\nabla \xi^{h})}^{2}}\nonumber\\
&&+\dfrac{1}{2\theta+1}\sum_{n=1}^{N-1} \norm{\zeta^{h}_{n+1}-2\zeta^{h}_{n}+\zeta^{h}_{n-1}}_{F}^{2} + \dfrac{ 2\Delta tD_c}{2\theta+1} \sum_{n=1}^{N-1}\norm{{F}_{n+\theta}^{\eps, D_c}(\nabla \zeta^{h})}^{2}\nonumber\\
&\leq&\bigg(\dfrac{2\theta-1}{2\theta+1}\bigg)^{N}\Big(\norm{\bphi^{h}_{0}}^2+\norm{\xi^{h}_{0}+\norm{\zeta^{h}_{0}}^2}^2\Big)+2\bigg(1-\bigg(\dfrac{2\theta-1}{2\theta+1}\bigg)^{N}\bigg)\Bigg[\norm{ \begin{bmatrix}
	\bphi^{h}_1 \\
	\bphi^{h}_{0}
	\end{bmatrix}} _{G}^{2}+\norm{ \begin{bmatrix}
	\xi^{h}_1 \\
	\xi^{h}_{0}
	\end{bmatrix}} _{G}^{2}+\norm{ \begin{bmatrix}
	\zeta^{h}_1 \\
	\zeta^{h}_{0}
	\end{bmatrix}} _{G}^{2}\nonumber\\
&&+C\Big(\nu^{-1}\Delta t^4 \norm{|p_{tt}|}_{2,0}^{2}+\nu^{-1}h^{2k+2}\norm{|p|}_{2,k+1}^{2}+\nu^{-1}\Delta t^4\norm{|\bu_{ttt}|}_{2,0}^{2}
	+\nu\Delta t^4\norm{|\nabla\bu_{tt}|}_{2,0}^{2}\nonumber\\
	&&+\nu^{-1}\Delta t^4\norm{|\nabla\bu_{tt}|}_{2,0}^{2}+\nu^{-1}\Delta t^4\norm{|\nabla\bu|}_{\infty,0}^{2}\norm{|\nabla\bu_{tt}|}_{2,0}^{2}+\nu^{-1}\beta_{T}^2\norm{\bfg}_{\infty}^2\Delta t^4\norm{|\nabla T_{tt}|}_{2,0}^{2}\nonumber\\
	&&+\nu^{-1}\beta_{S}^2\norm{\bfg}_{\infty}^2\Delta t^4\norm{|\nabla S_{tt}|}_{2,0}^{2}+\nu^{-1}h^{2k+2}\norm{|\bu_{t}|}_{2,k+1}^{2}+\nu h^{2k}\norm{|\bu|}_{2,k+1}^{2}+\nu^{-1} h^{2k}\norm{|\bu|}_{2,k+1}^{2}\nonumber\\
	&&+\nu^{-1}h^{2k}\norm{|\nabla \bu|}_{\infty}^{2}\norm{|\bu|}_{2,k+1}^{2}+\nu^{-1}h^{2k}\beta_{T}^2\norm{\bfg}_{\infty}^2\norm{|T|}_{2,k+1}^{2}+\nu^{-1}h^{2k}\beta_{S}^2\norm{\bfg}_{\infty}^2\norm{|S|}_{2,k+1}^{2}\nonumber\\
&&+\gamma^{-1}\Delta t^4\norm{|T_{ttt}|}_{2,0}^{2}+\gamma\Delta t^4\norm{|\nabla T_{tt}|}_{2,0}^{2}+\gamma^{-1}\Delta t^4\norm{|\nabla T_{tt}|}_{2,0}^{2}+\gamma^{-1}\Delta t^4\norm{\nabla T}_{\infty}^{2}\norm{|\nabla  \bu_{tt}|}_{2,0}^{2}\nonumber\\
&&+\gamma^{-1}\Delta t^4\norm{|\nabla\bu |}_{\infty}^{2}\norm{|\nabla T_{tt}|}_{2,0}^{2}+\gamma^{-1}h^{2k+2}\norm{|T_{t}|}_{2,k+1}^{2}+\gamma h^{2k}\norm{|T|}_{2,k+1}^{2}\nonumber\\ &&+\gamma^{-1}h^{2k}\norm{|\nabla \bu|}_{\infty}^{2}\norm{|T|}_{2,k+1}^{2}+\gamma^{-1}h^{2k}\norm{|\nabla T|}_{\infty}^{2}\norm{|\bu|}_{2,k+1}^{2}+D_c^{-1}\Delta t^4\norm{|S_{ttt}|}_{2,0}^{2}\nonumber\\
&&+D_c\Delta t^4\norm{|\nabla S_{tt}|}_{2,0}^{2}+D_c^{-1}\Delta t^4\norm{|\nabla S_{tt}|}_{2,0}^{2}+D_c^{-1}\Delta t^4\norm{\nabla S}_{\infty}^{2}\norm{|\nabla  \bu_{tt}|}_{2,0}^{2}\nonumber\\
&&+D_c^{-1}\Delta t^4\norm{|\nabla\bu |}_{\infty}^{2}\norm{|\nabla S_{tt}|}_{2,0}^{2}+D_c^{-1}h^{2k+2}\norm{|S_{t}|}_{2,k+1}^{2}+D_c h^{2k}\norm{|S|}_{2,k+1}^{2}\nonumber\\ &&+D_c^{-1}h^{2k}\norm{|\nabla \bu|}_{\infty}^{2}\norm{|S|}_{2,k+1}^{2}+D_c^{-1}h^{2k}\norm{|\nabla S|}_{\infty}^{2}\norm{|\bu|}_{2,k+1}^{2}\Big)\nonumber\\
&&+\tilde{C}\Delta t\sum_{n=0}^{N-1}(\norm{\bphi^{h}_{n}}^2+\norm{\xi^{h}_{n}}^2+\norm{\zeta^{h}_{n}}^2)\Bigg]\label{e}
\end{eqnarray}
where $\tilde{C}:=\tilde{C}(\nu^{-1},\gamma^{-1}, D_S^{-1}, \beta_{T}^2, \beta_{S}^2, \norm{\bfg}_{\infty}^2)$.\\
We next apply the Lemma \ref{gr} and use the following inequality in (\ref{e}),

$$0\leq\big(\dfrac{2\theta-1}{2\theta+1}\big)^{N}\leq 1\quad \mbox{for any}\quad N \ge 0.$$
The final result follows from the triangle inequality and Lemma $\ref{lem:gnorm}$.
\end{proof}

\section{Numerical Experiments}
\noindent In this section, we perform two numerical tests in order to show the efficiency of proposed method and validate the theoretical findings. The first example is verification of the numerical convergence rates for an analytic test problem with known solution. The second example is of more practical interest; it is a buoyancy driven cavity flow example in a tall rectangular cavity.

The simulations are performed with the finite element software package FreeFem++ \cite{hec}. In all computations, the Taylor-Hood finite element for velocity and pressure, and  piecewise quadratics for temperature and concentration are used on triangular grids. The Darcy flow regime ($Da=\infty$) is assumed for all tests. In order to see the effect of stabilization parameters, the results are also compared with the usual BDF2LE method, which is obtained through picking $\epsilon = \epsilon_1 =\epsilon_2 =0$ and $\theta=1$ (unstabilized case) in (\ref{forvel})-(\ref{forcont}), which gives
\begin{gather}
\frac{(\theta+\frac{1}{2})\bu_{n+1}-2\theta \bu_n+(\theta-\frac {1}{2})\bu_{n-1}}{\Delta t} -\theta \nu\Delta \bu_{n+1}-(\nu-\theta \nu)\Delta \bu_n \nonumber
\\
+((\theta+1)\bu_n-\theta \bu_{n-1})\cdot \nabla (\theta \bu_{n+1}+ (1-\theta \bu_{n}))  \label{bdforvel}
\\
+\theta \nabla p_{n+1}+(1-\theta)\nabla p_n=\Big(\beta_T((\theta+1)T_n-\theta T_{n-1})+ \beta_S((\theta+1)S_n-\theta S_{n-1}))\Big)\bfg +\bff_{n+\theta}\nonumber
\\
\nabla \cdot \bu_{n+1}=0
\\
\frac{(\theta+\frac{1}{2})T_{n+1}-2\theta T_n+(\theta-\frac {1}{2})T_{n-1}}{\Delta t} -\theta \gamma\Delta T_{n+1}-(\gamma-\theta \gamma)\Delta T_n \nonumber
\\
+((\theta+1)\bu_n-\theta \bu_{n-1})\cdot \nabla (\theta T_{n+1}+(1-\theta T_{n}))
=\varphi_{n+\theta} \label{bdfortemp}
\\
\frac{(\theta+\frac{1}{2})S_{n+1}-2\theta S_n+(\theta-\frac {1}{2})S_{n-1}}{\Delta t} -\theta D_c\Delta S_{n+1}-(D_c-\theta D_c)\Delta S_n \nonumber
\\
+((\theta+1)\bu_n-\theta \bu_{n-1})\cdot \nabla (\theta S_{n+1}+(1-\theta S_{n}))
=\psi_{n+\theta}. \label{bdforcont}
\end{gather}
Here, the forcing functions $\bff_{n+\theta}, \varphi_{n+\theta}$ and $\psi_{n+\theta}$ are included in (\ref{bdforvel})-(\ref{bdforcont}). We also note that the similar results are also obtained with the CNLE with the choices of parameters $\epsilon = \epsilon_1 =\epsilon_2 =0$ and $\theta = 1/2$.

\subsection{Numerical convergence study}

In this subsection, we show that the theoretical orders of the errors are also obtained through a numerical simulation. In order to do so, we pick the known-solution
\begin{align}\label{truesol}
\bu=\left(%
\begin{array}{c}
\cos(y) \\
\sin(x)%
\end{array}%
\right)e^t,\quad
p=(x-y)(1+t),\quad
T=\sin(x+y)e^{1-t}, \quad
S=\cos(x+y)e^{1-t}.
\end{align}
with the parameters $Pr=D_c= \gamma=\beta_T=\beta_S=1$ and the right hand side functions $\bff,\varphi$ and $\psi$ are chosen such that (\ref{truesol}) satisfies (\ref{bous}).

We will present computational results with $\epsilon = \epsilon_1 =\epsilon_2 =0$, $\theta=1$ and $\epsilon = \epsilon_1= \epsilon_2 =1$ (with stabilization) in a unit square. The final time and the time step size are chosen as $t=10^{-1}$ and  $\Delta t = t/16$.
To test the spatial convergence, we fix the time step size and calculate the errors for varying $h$ and consider the velocity errors in the discrete norm  $L^2(0,T;{\bf H}^1(\Omega))$
 $$\|\textbf{u}-\textbf{u}^h\|_{2,1}=\left\{\Delta t \sum_{n=1}^{N}\|\textbf{u}(t^n)-\textbf{u}_{n}^{h}\|^{2}\right\}^{1/2}.$$ The results of different $\epsilon$, $\epsilon_1$ and $\epsilon_2$ values for the spatial errors and error rates are given in Table \ref{table:tab1} and Table \ref{table:tab2}. One can see that the orders of convergence of $\|\textbf{u}-\textbf{u}^h\|_{2,1}$, $\|T-T^h\|_{2,1}$ are quadratic, which is an optimal order for both BDF2LE and for the proposed method. We note that because of the parameter choices of this numerical test, the errors for $\|S-S^h\|_{2,1}$ are similar.
\begin{table}[h!]
\begin{center}

\begin{tabular}{|c|c|c|c|c|}
  \hline
  $h$ & $\|\textbf{u}-\textbf{u}^h\|_{2,1}$ & Rate &$\|T-T^h\|_{2,1}$ & Rate\\
  \hline
   1/4&1.606e-3 &--& 3.99e-3 &--   \\
   \hline
  1/8&4.357e-4 &1.88  &1.00e-3  &1.99   \\
  \hline
  1/16&1.124e-4 &1.95  &2.527e-4  &1.98   \\
  \hline
  1/32&2.848e-5 &1.98 & 6.318e-5 & 2.00   \\
  \hline
  1/64&7.171e-6 &1.98  &1.592e-5  &1.98    \\
  \hline
\end{tabular}
\caption{Spatial errors and rates of convergence for $\epsilon= \epsilon_1 =\epsilon_2 =0$.}
\label{table:tab1}
\end{center}
\end{table}

\begin{table}[h!]
\begin{center}
\begin{tabular}{|c|c|c|c|c|}
 \hline
  $h$ & $\|\textbf{u}-\textbf{u}^h\|_{2,1}$ & Rate &$\|T-T^h\|_{2,1}$ & Rate\\
  \hline
   1/4&1.621e-3 &--& 4.003e-3 &--   \\
   \hline
  1/8&4.403e-4 &1.88  &1.01e-3  &1.98   \\
  \hline
  1/16&1.136e-4 &1.95  &2.531e-4  &2.00   \\
  \hline
  1/32&2.879e-5 &1.98 & 6.365e-5 & 1.99   \\
  \hline
  1/64&7.244e-6 &1.98  &1.740e-5  &1.88    \\
  \hline
\end{tabular}
\caption{Spatial errors and rates of convergence for $\epsilon= \epsilon_1 =\epsilon_2 =1$.}
\label{table:tab2}
\end{center}
\end{table}
We also fix the mesh size to $h=1/128$ to see the temporal errors and the convergence rates by using different time steps with an end time of $t=1$. The results are given in Table \ref{table:tab3} and Table \ref{table:tab31}. As expected, we observe a second order convergence in time. However, the velocity error rates becomes better for the stabilized case as $\Delta t$ decreases. In addition, the rates for the temperature errors are far more better than unstabilized case when they are compared with the proposed method.
In summary, the observations of convergence orders of (\ref{Discvel})-(\ref{Distemp}) are in accordance with the discussion in Corollary \ref{corr}.
\begin{table}[h!]
\begin{center}

\begin{tabular}{|c|c|c|c|c|}
  \hline
  $\Delta t$ & $\|\textbf{u}-\textbf{u}^h\|_{2,1}$ & Rate &$\|T-T^h\|_{2,1}=\|S-S^h\|_{2,1}$  & Rate \\
  \hline
   1&3.093e-2 &--& 6.572e-2 &--   \\
   \hline
  1/2&6.662e-3 &2.21 &3.415e-2 &1.01   \\
  \hline
  1/4&1.568e-3&2.08 &1.220e-2 &1.49  \\
  \hline
  1/8&3.842e-4&2.02 & 3.617e-3 & 1.75   \\
  \hline
  1/16&1.007e-4&1.93 &9.841e-4 &1.88    \\
  \hline
\end{tabular}
\caption{Temporal errors and rates of convergence for $\epsilon= \epsilon_1 =\epsilon_2 =0$.}
\label{table:tab3}
\end{center}
\end{table}

\begin{table}[h!]
\begin{center}

\begin{tabular}{|c|c|c|c|c|}
  \hline
  $\Delta t$ & $\|\textbf{u}-\textbf{u}^h\|_{2,1}$ & Rate &$\|T-T^h\|_{2,1}=\|S-S^h\|_{2,1}$  & Rate \\
  \hline
   1&6.203e-3 &--& 7.005e-1 &--   \\
   \hline
  1/2&2.880e-3 &1.10&1.991e-1 &1.81   \\
  \hline
  1/4&1.293e-3&1.15 &5.233e-2 &1.92  \\
  \hline
  1/8&3.921e-4&1.72 & 1.151e-2 & 2.18   \\
  \hline
  1/16&1.058e-4&1.90 &2.610e-3 &2.14    \\
  \hline
\end{tabular}
\caption{Temporal errors and rates of convergence for $\epsilon= \epsilon_1 =\epsilon_2 =1$.}
\label{table:tab31}
\end{center}
\end{table}
\subsection{Buoyancy Driven Cavity Test}
As another numerical test, we apply the proposed method to so-called buoyancy driven cavity flow in a tall rectangular enclosure.
The purpose of this example is to capture correct flow patterns on coarse mesh and to get the correct solution where the unstabilized case fails. In this test, the effects of several dimensionless problem parameters on the solution are considered. We also calculate the Nusselt numbers and Sherwood numbers for this cavity test and compare our results with those reported previously.

The computational domain we use is a rectangular cavity of height $2$ and width $1$ with different temperature and concentration values at vertical walls, which are regarded as hot and cold walls, see Figure \ref{fig:newdomain}. The horizontal walls are insulated and assumed to allow no heat and species transfer through. The boundary conditions are no-slip boundary conditions for the velocity and Dirichlet boundary conditions for the temperature and concentration at vertical walls as well. The horizontal walls accept the boundary conditions, $\frac{\partial T}{\partial n}= \frac{\partial S}{\partial n} = 0$ .
At initial state, the fluid has no motion. According to the variation of temperature and concentration at vertical walls, the motion will be started due to the buoyancy forces as density varies. The final time is chosen to be $t=1$ and the time interval is divided in equidistant time steps of length $10^{-4}$. The stabilization parameters are taken as $\epsilon={O}(\nu),\, \epsilon_1 ={O}(\gamma),\,\epsilon_2 ={O}(D_c)$.
\begin{figure}[htb]
\centerline{\hbox{
\epsfig{figure=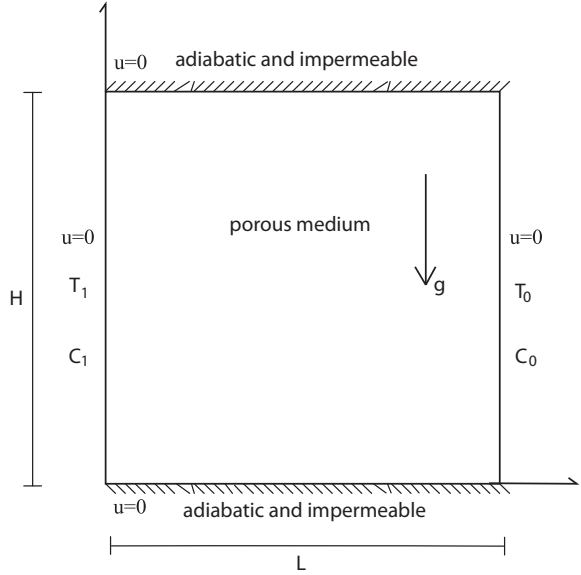,width=0.5\textwidth}
}}
\caption{\label{fig:newdomain} Buoyancy driven cavity flow domain with its boundary conditions}
\end{figure}

Before we present our results, we remark that the correct patterns are captured for all different parameter cases for a very coarse mesh consisting of only $8262$ velocity d.o.f, $4131$ temperature d.o.f and concentration d.o.f.. In general, the proposed method and the unstabilized case produce very similar results for the tests with $Ra \leq 10^5$. However, the unstabilized case gave no result and the solution diverges for $Ra =10^6$. This might be noted as the greatest superiority of our method against the unstabilized case.

\subsubsection{The effect of buoyancy ratio $N$}

\noindent In this test, the effect of buoyancy ratio $N$ is considered for $N=0.8$ and $N=1.3$, by fixing $Pr=1, Ra=10^5$ and $ Le=2$. The results are shown in Figure \ref{fig:N}. It can be observed that the variation of density in concentration is larger than variation in temperature for $N>1$. As it is expected, due to the increase in the buoyancy ratio, the concentration stratification increases. Thus, the force pushing the low concentration fluid up becomes greater, \cite{march}.

For $N<1$, this time density variations are due to the temperature gradients mostly and situation turns out for temperature. These graphics perfectly agree with the benchmark studies of \cite{chen} and \cite{march}. 
\begin{figure}[h!]
\centerline{\hbox{\epsfig{figure=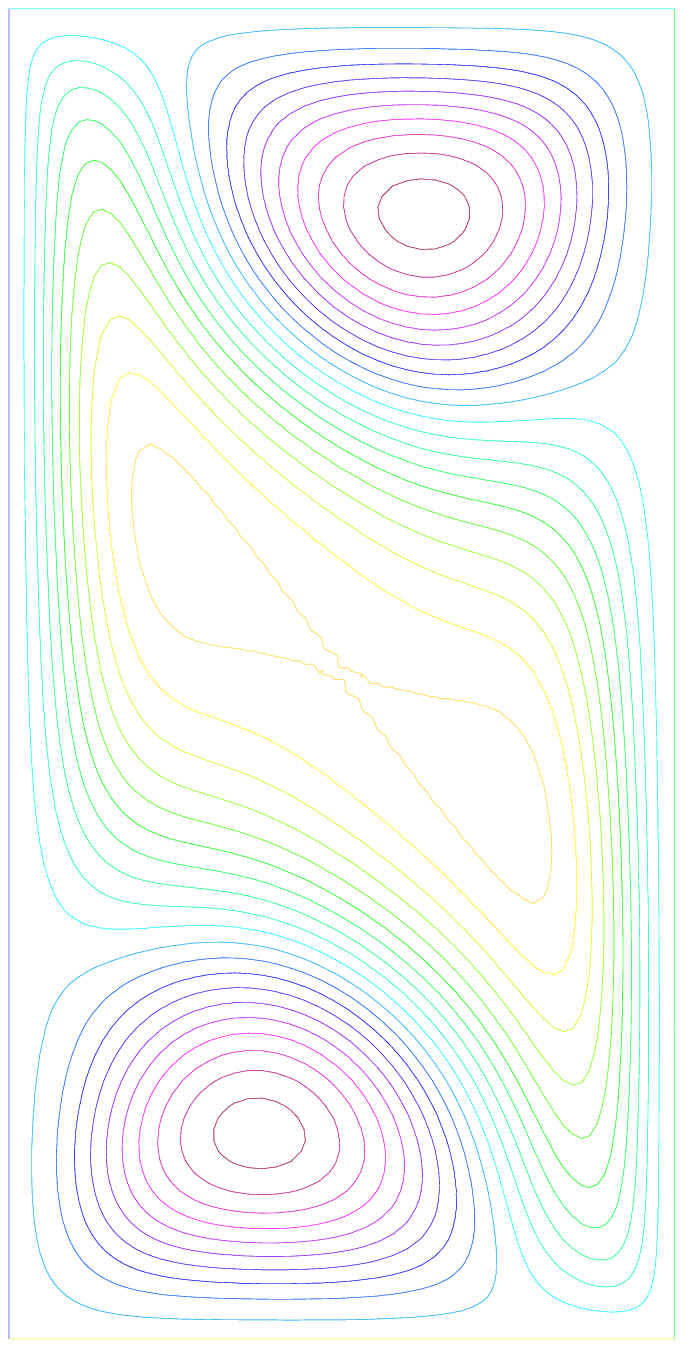,width=0.45\textwidth}
\epsfig{figure=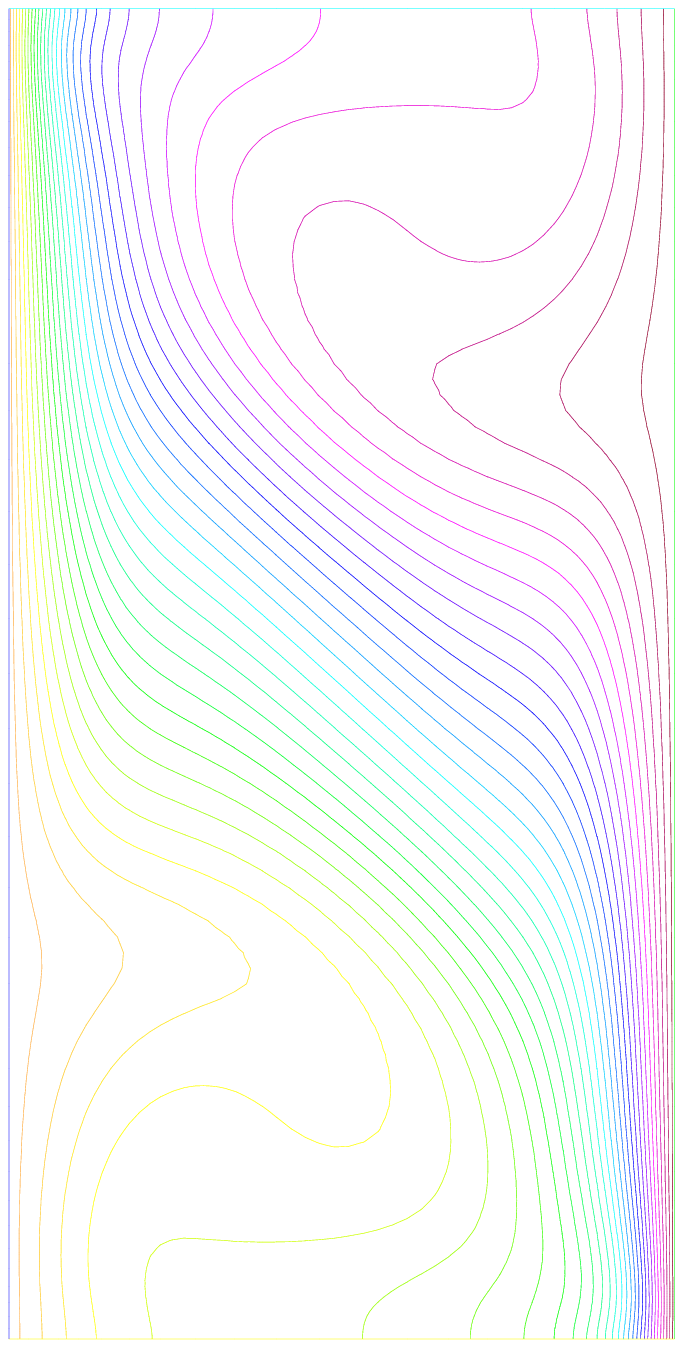,width=0.45\textwidth}
\epsfig{figure=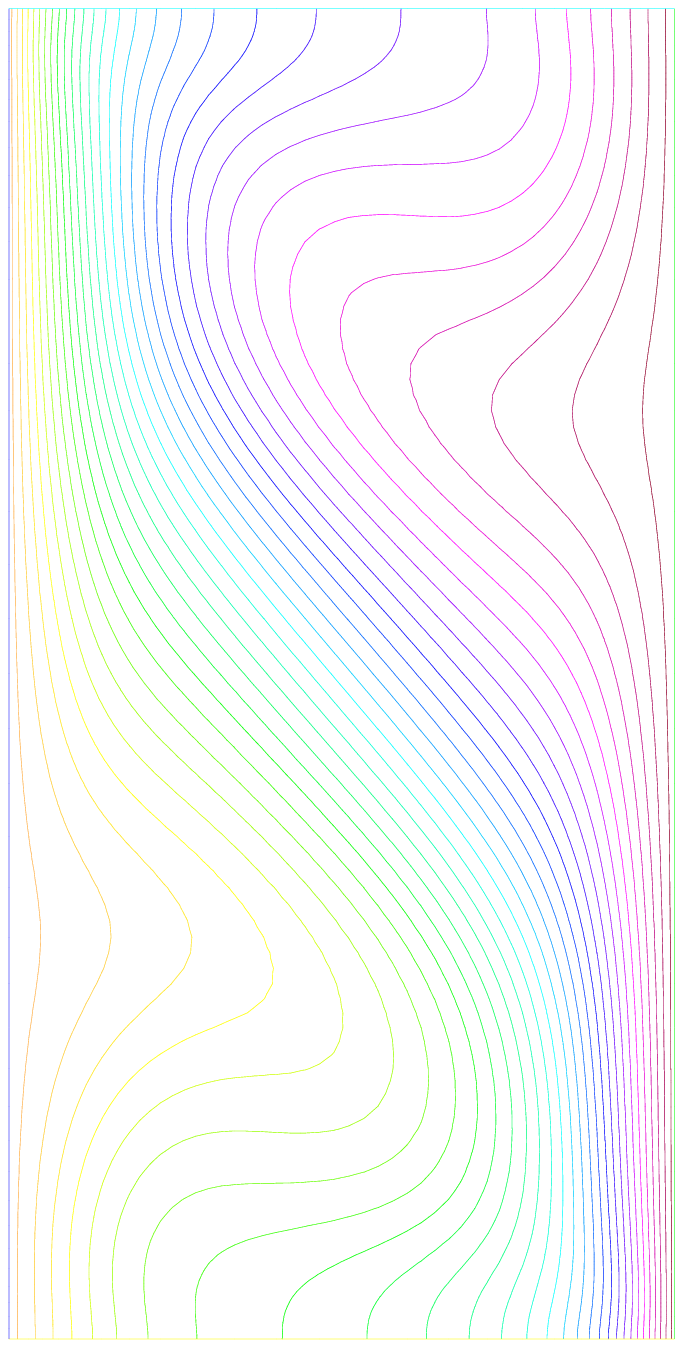,width=0.45\textwidth}
 }}
\centerline{\hbox{\epsfig{figure=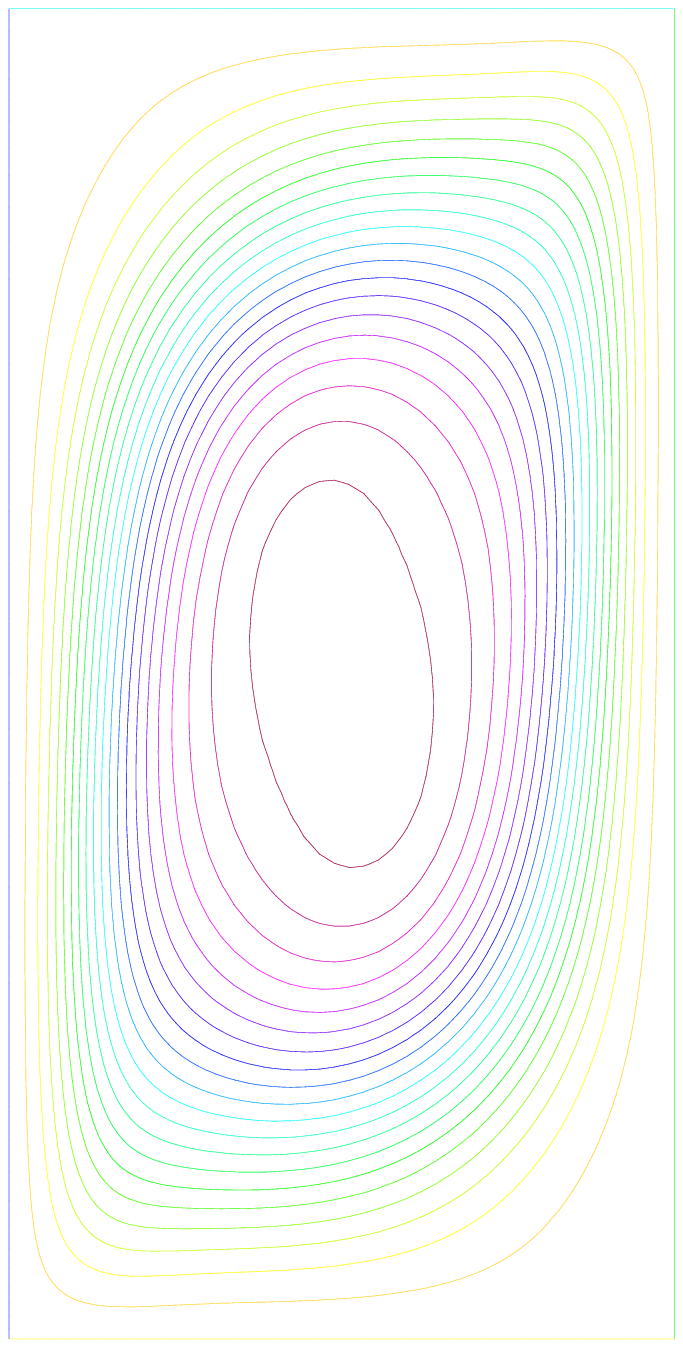,width=0.45\textwidth}
\epsfig{figure=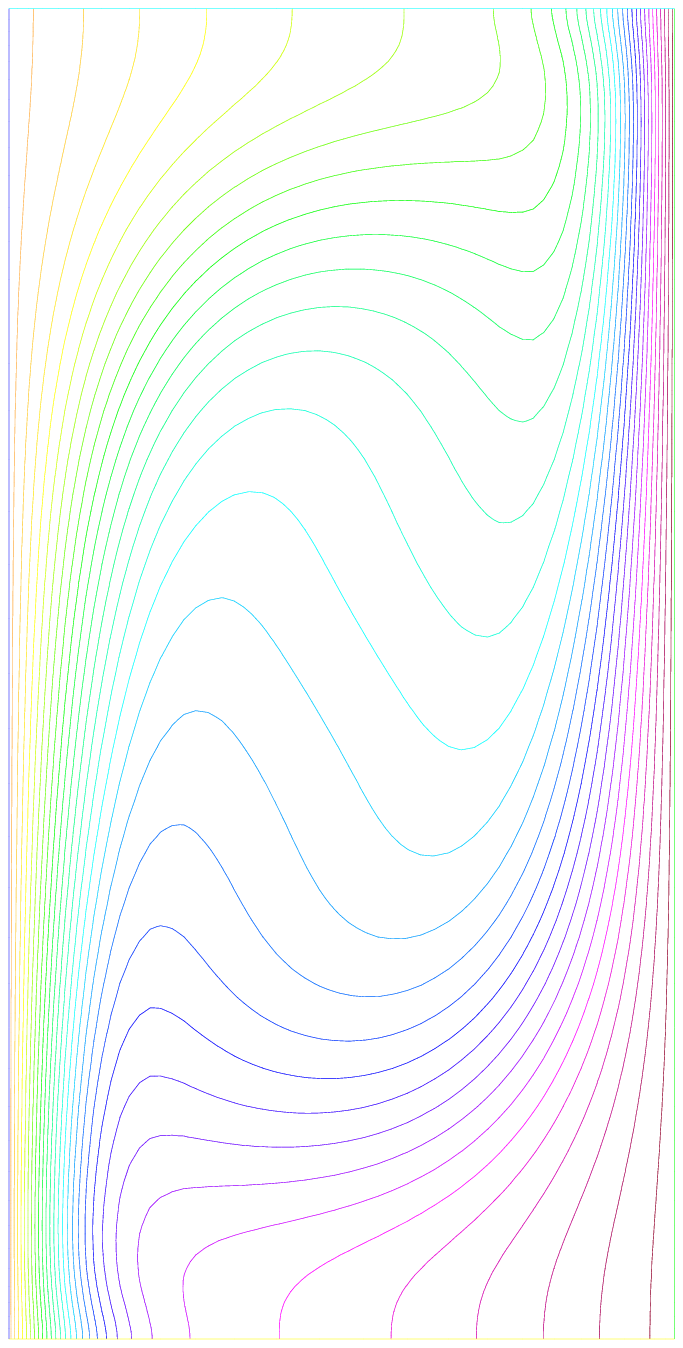,width=0.45\textwidth}
\epsfig{figure=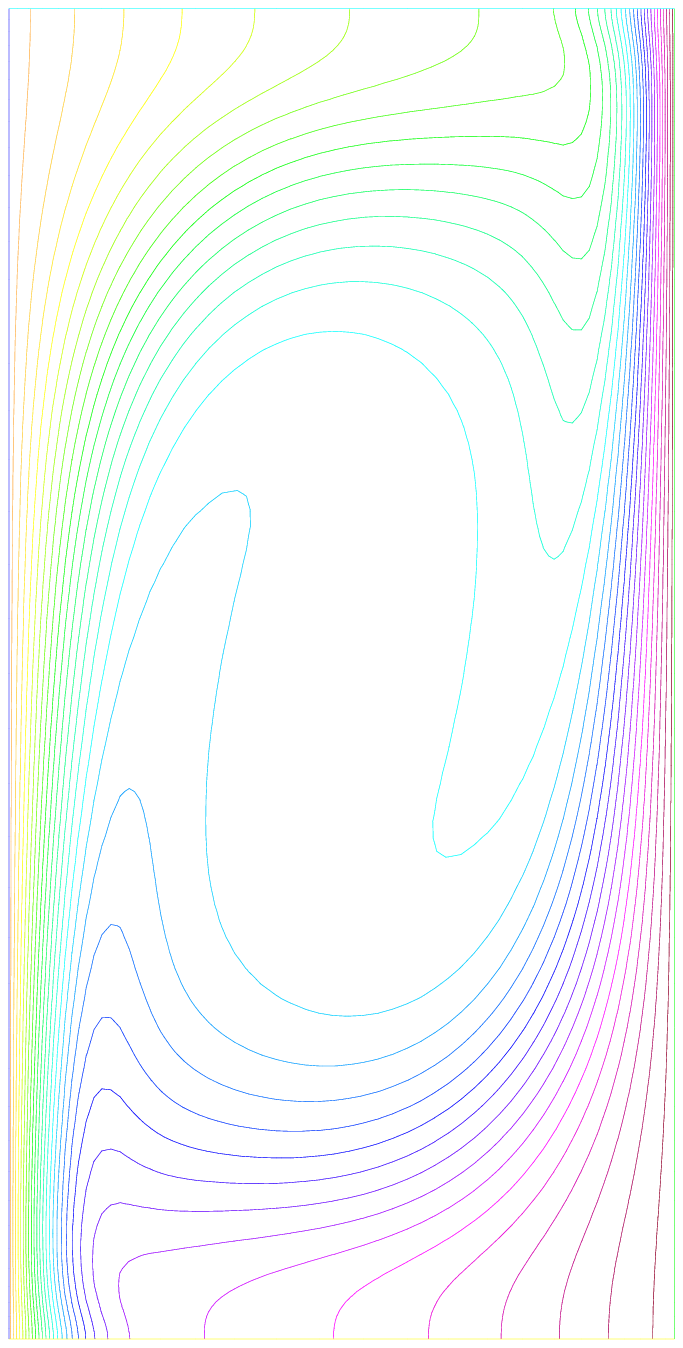,width=0.45\textwidth}
}}
\caption{\label{fig:N} Velocity streamlines, Temperature contours and Concentration contours(from left to right) for $Pr=1, Ra=10^5, Le=2$ with  $N=1.3$ (up) and $N=0.8$ (down) }
\end{figure}

\subsubsection{The effect of Lewis number $Le$}
\noindent The effect of Lewis number is considered with the choices of $Le= 0.2$ and $Le= 1.0$. The values of $Pr=1, Ra=10^5$ and $N=1$ are fixed in all computations. Due to the definition of Lewis number, $Le \leq 1.0$ means the mass diffusivity is greater than the thermal diffusivity. In this case, the concentration becomes dominant because of its better capability of spreading higher concentration values. The value of $Le= 1.0$ means equal diffusivity case. When temperature and concentration behave in the same way, the forces made by the temperature and concentration cancel each other in both walls initially. Thus, the fields diffuse exactly in the same way and these forces always are balanced equally. The final solution is just the diffusion of the fields through the domain as it is noted in \cite{march}. The mentioned situations above could be observed directly from Figure \ref{fig:Le}.
\begin{figure}[h!]
\centerline{\hbox{
\epsfig{figure=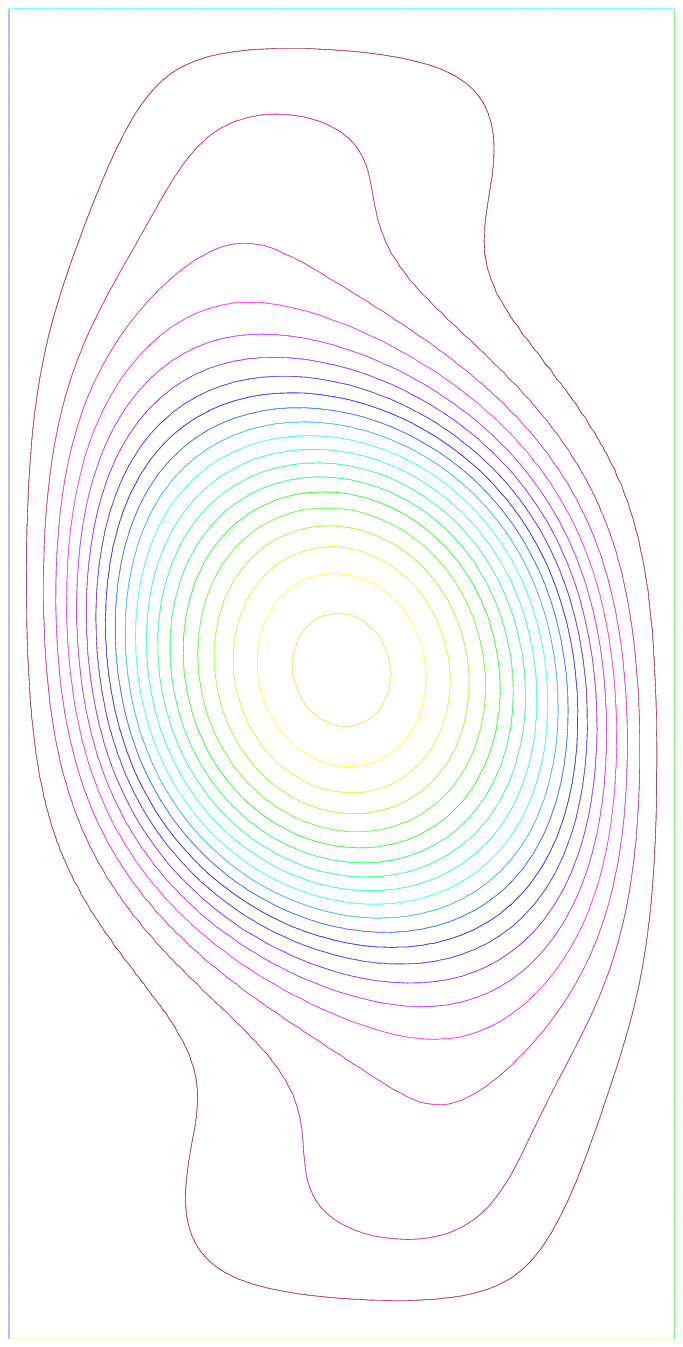,width=0.45\textwidth}
\epsfig{figure=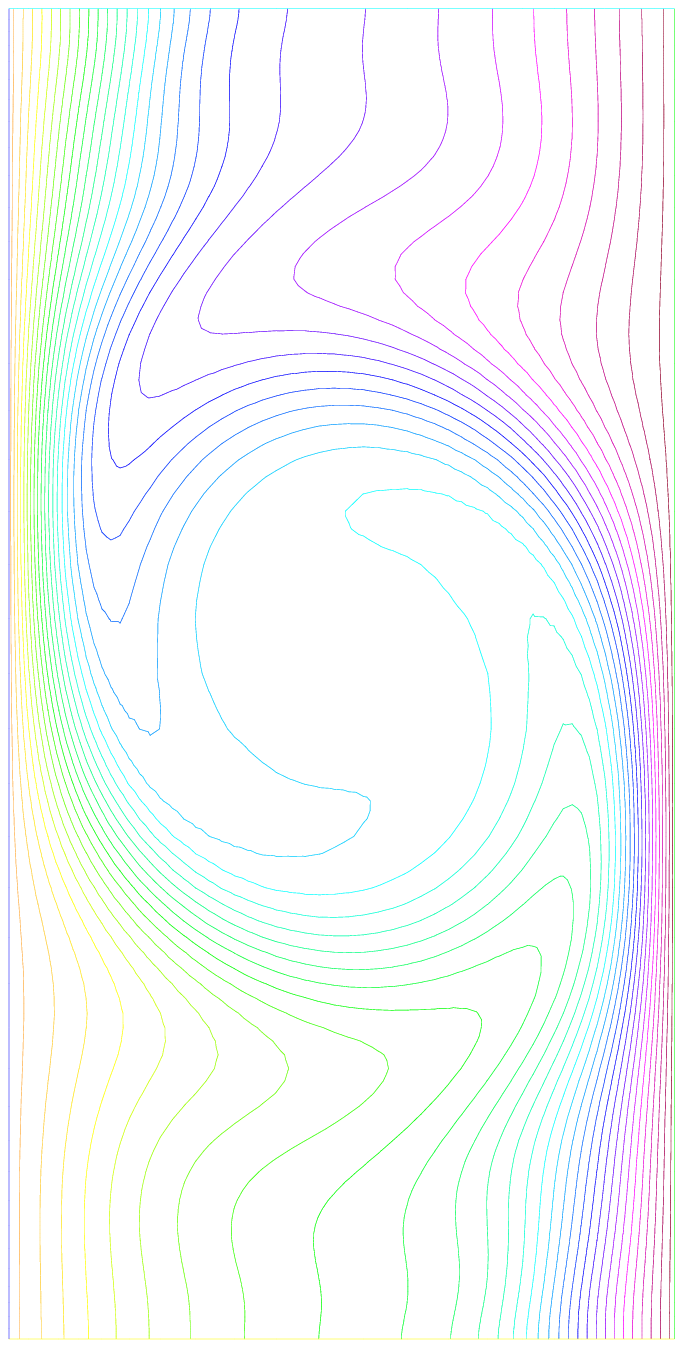,width=0.45\textwidth}
\epsfig{figure=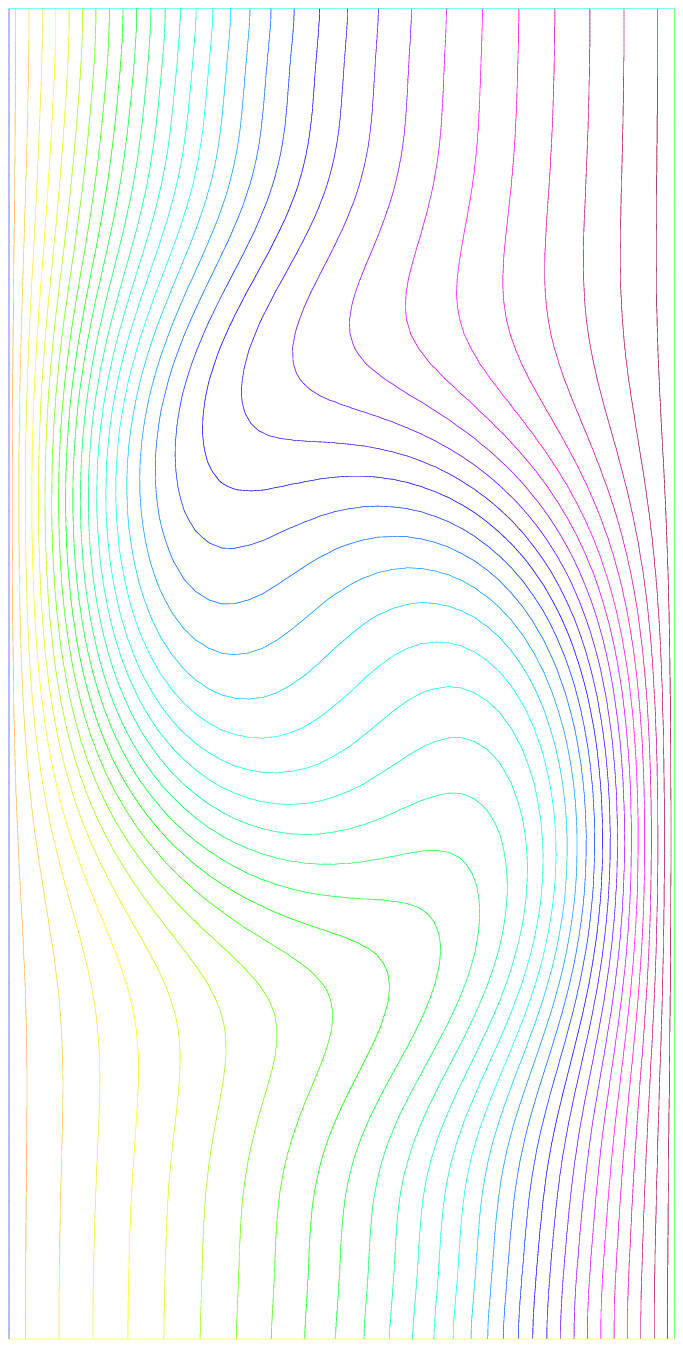,width=0.45\textwidth}
}}
\centerline{\hbox{
\epsfig{figure=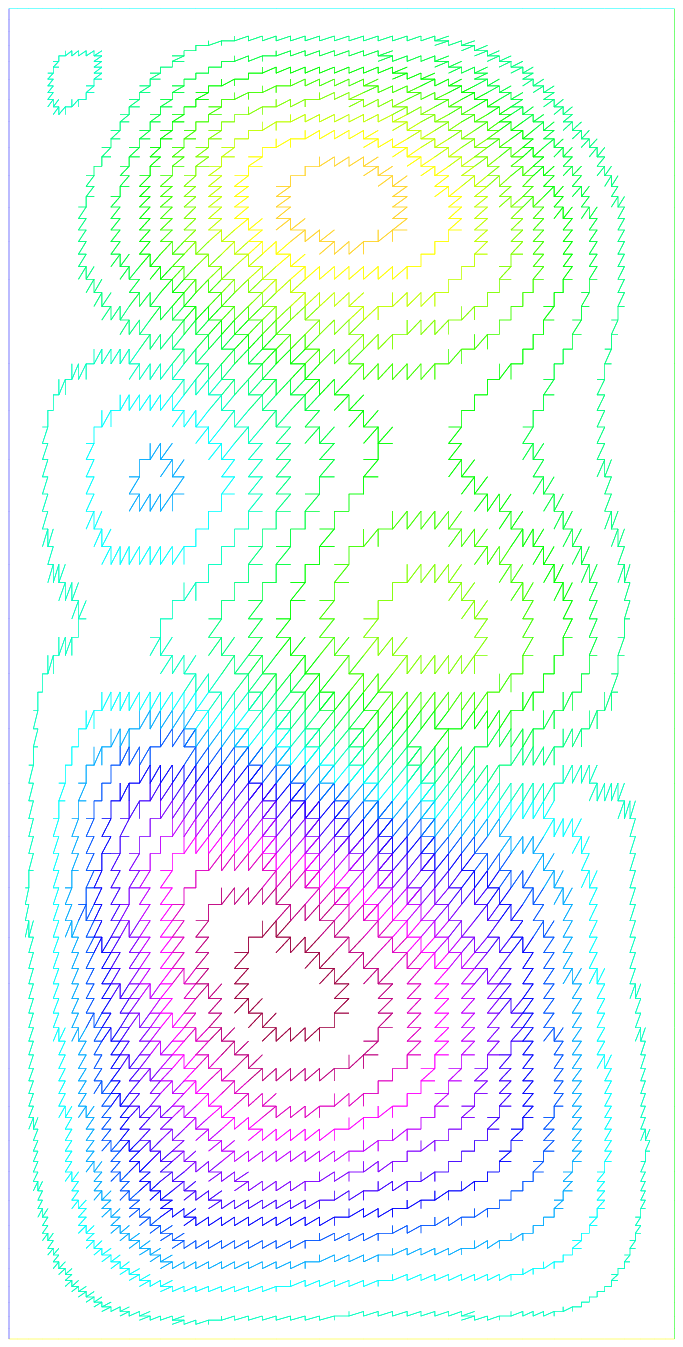,width=0.45\textwidth}
\epsfig{figure=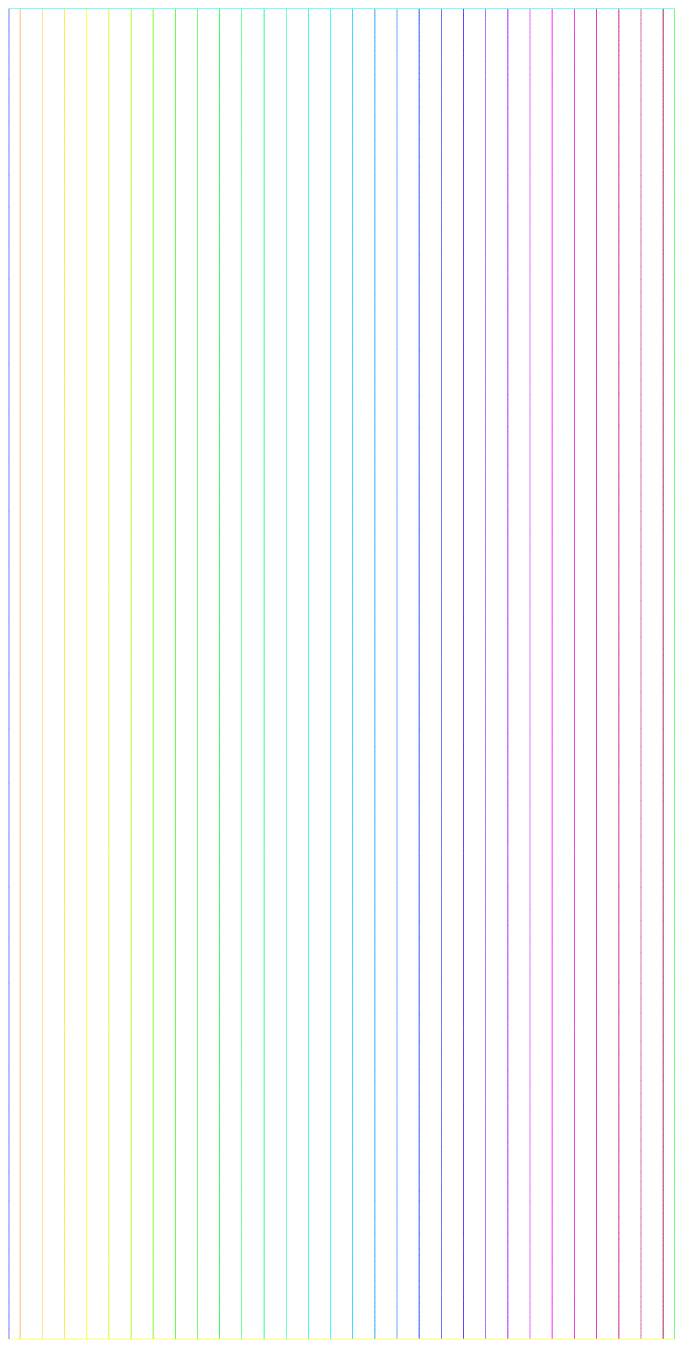,width=0.45\textwidth}
\epsfig{figure=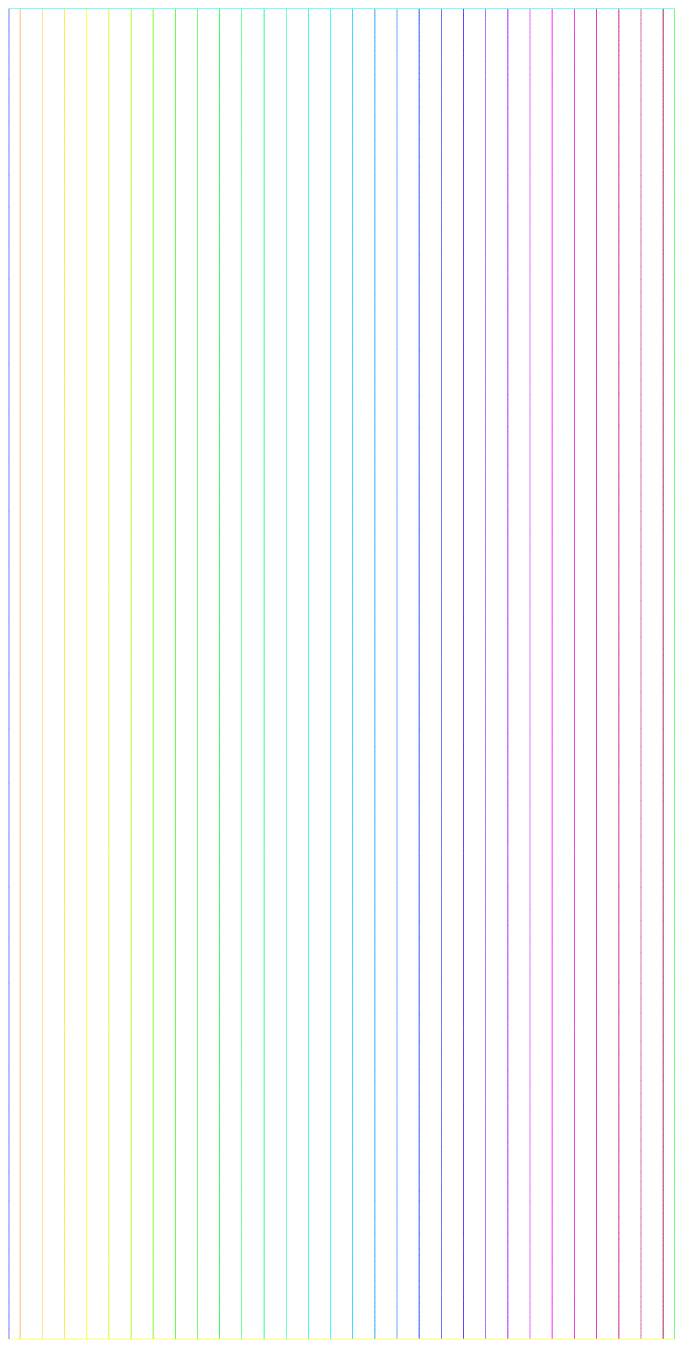,width=0.45\textwidth}
}}
\caption{\label{fig:Le} Velocity streamlines, Temperature contours and Concentration contours (from left to right) for $Pr=1, Ra=10^5, N=1$ with  $Le=0.2$ (up) and $Le=1.0$ (down) }
\end{figure}

\subsubsection{The effect of Rayleigh number $Ra$}
\noindent For natural convection type problems, increasing the Rayleigh number and keeping the thermal and mass diffusivity parameters constant will increase the characteristic velocity of the flow. This can cause the flow behave turbulent. Since the transition to turbulent case means richness of the flow scales, dealing with a very challenging numerical problem is inevitable as Rayleigh number increases.
The test is carried out for three different Rayleigh numbers, $Ra = 10^4, 10^5,10^6$ with the coarse mesh discretization. The results are presented in Figure \ref{fig:Ra} only for the case $Ra = 10^6$. For other $Ra$ values, the figures are similar and we do not depict. All the results are comparable with \cite{fat}, which uses a POD-ROM scheme and an extra VMS stabilization for  $Ra = 10^6$ for finer meshes.

\begin{figure}[h!]
\centerline{\hbox{
\epsfig{figure=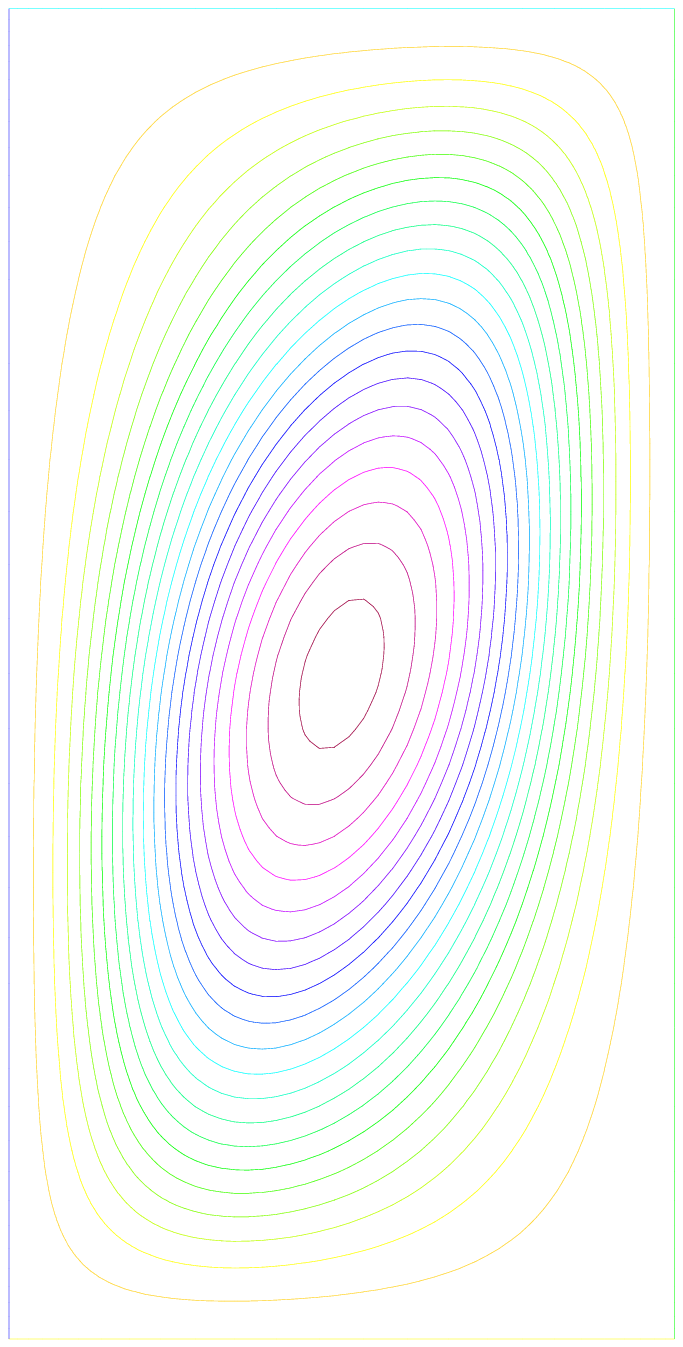,width=0.45\textwidth}
\epsfig{figure=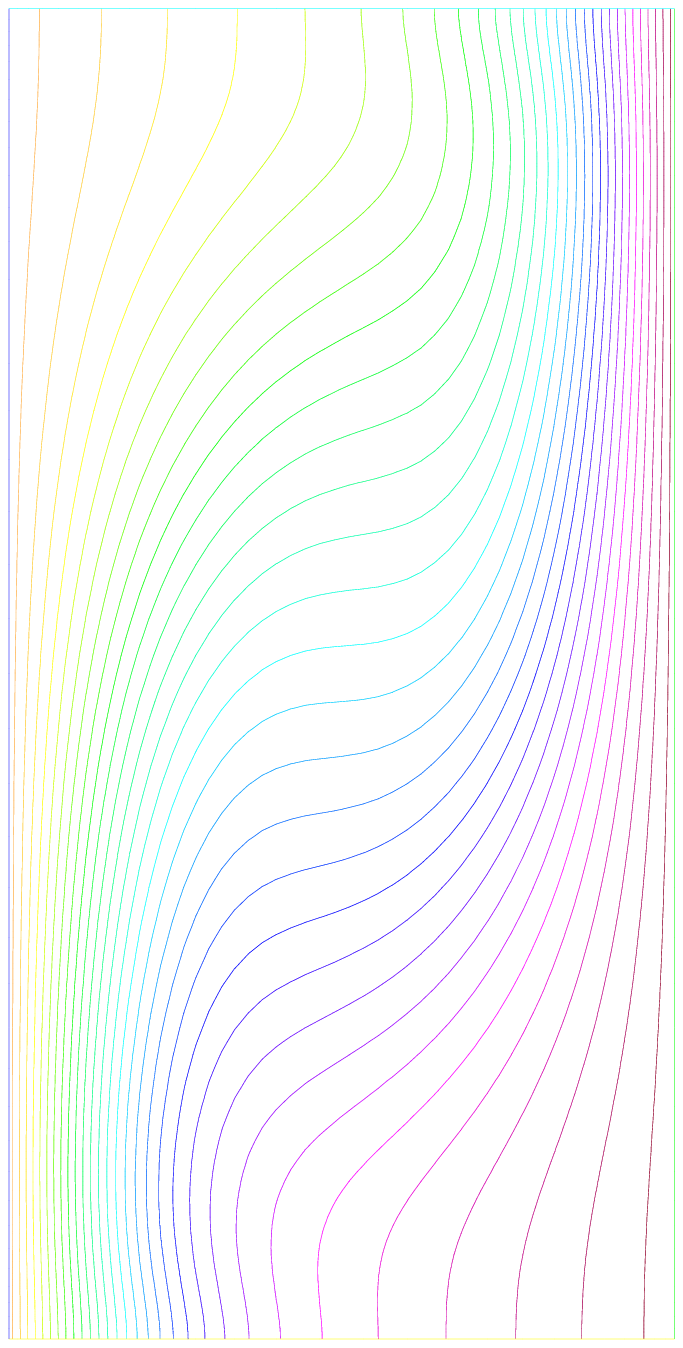,width=0.45\textwidth}
\epsfig{figure=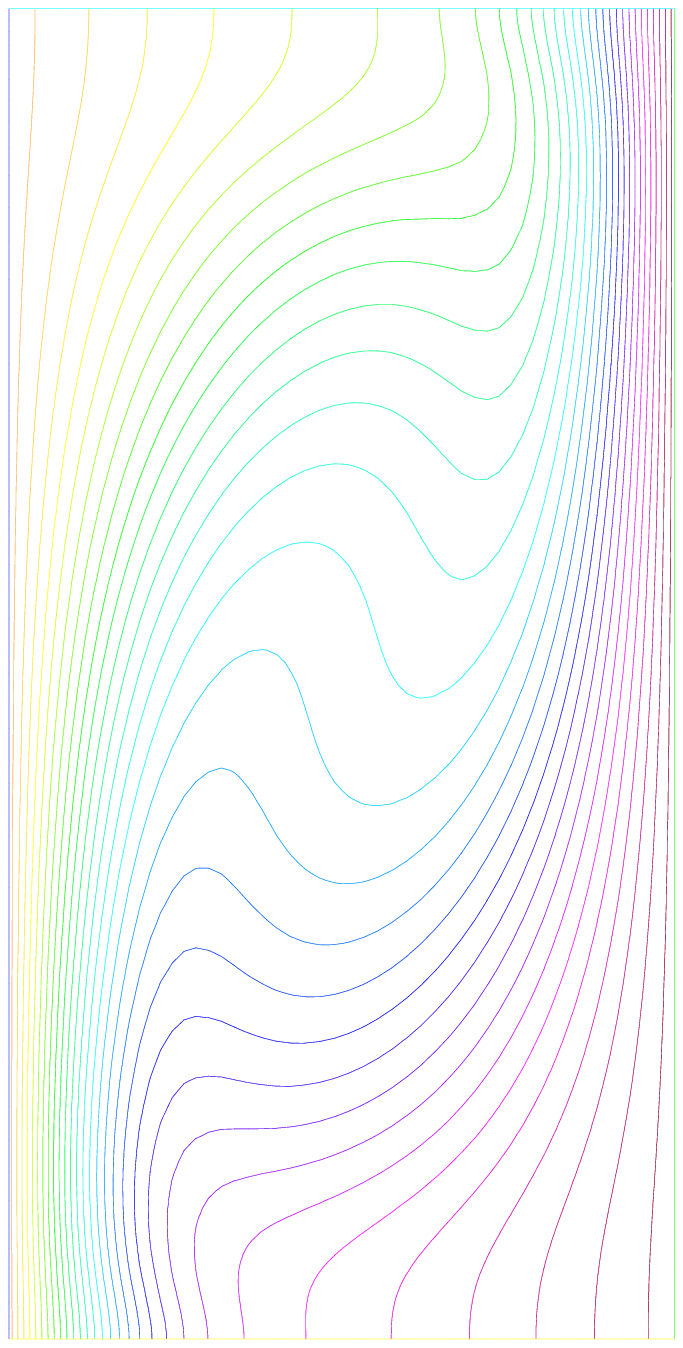,width=0.45\textwidth}
}}
\centerline{\hbox{
\epsfig{figure=str_105_Eps_1.eps,width=0.45\textwidth}
\epsfig{figure=T_105_Eps_1_empty.eps,width=0.45\textwidth}
\epsfig{figure=C_105_Eps_1_empty.eps,width=0.45\textwidth}
}}
\centerline{\hbox{
\epsfig{figure=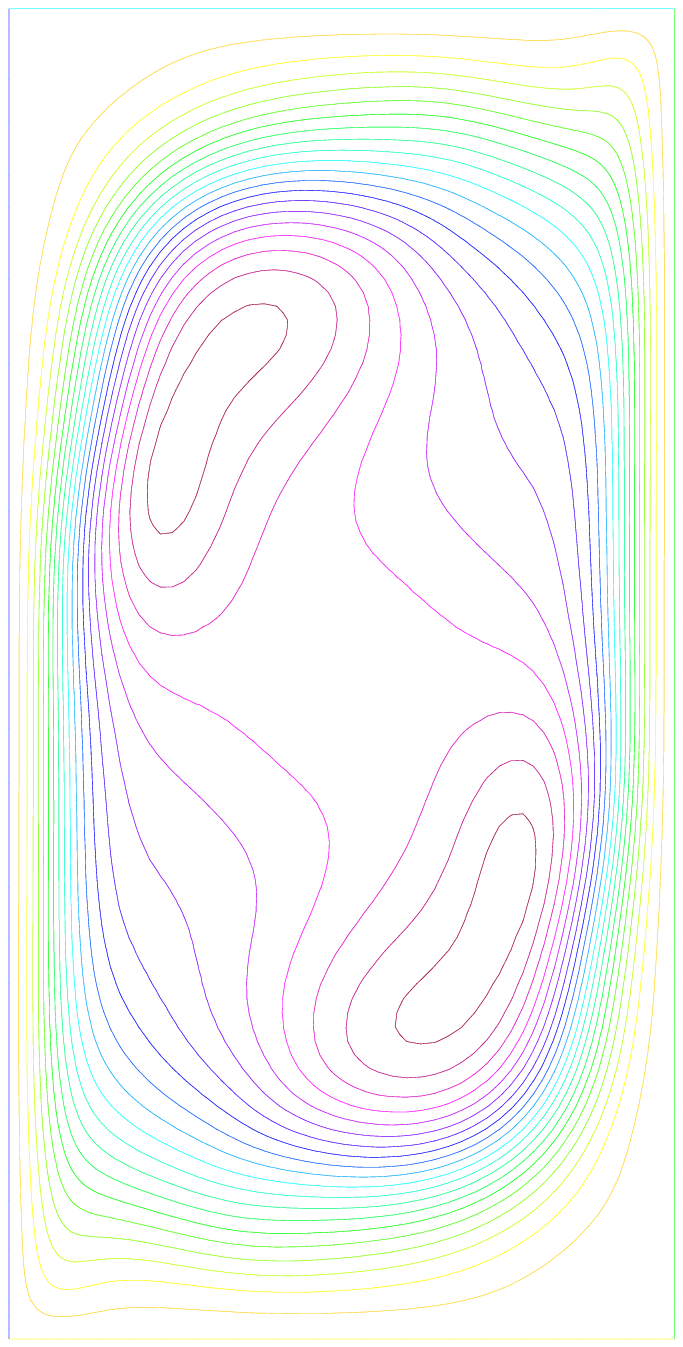,width=0.45\textwidth}
\epsfig{figure=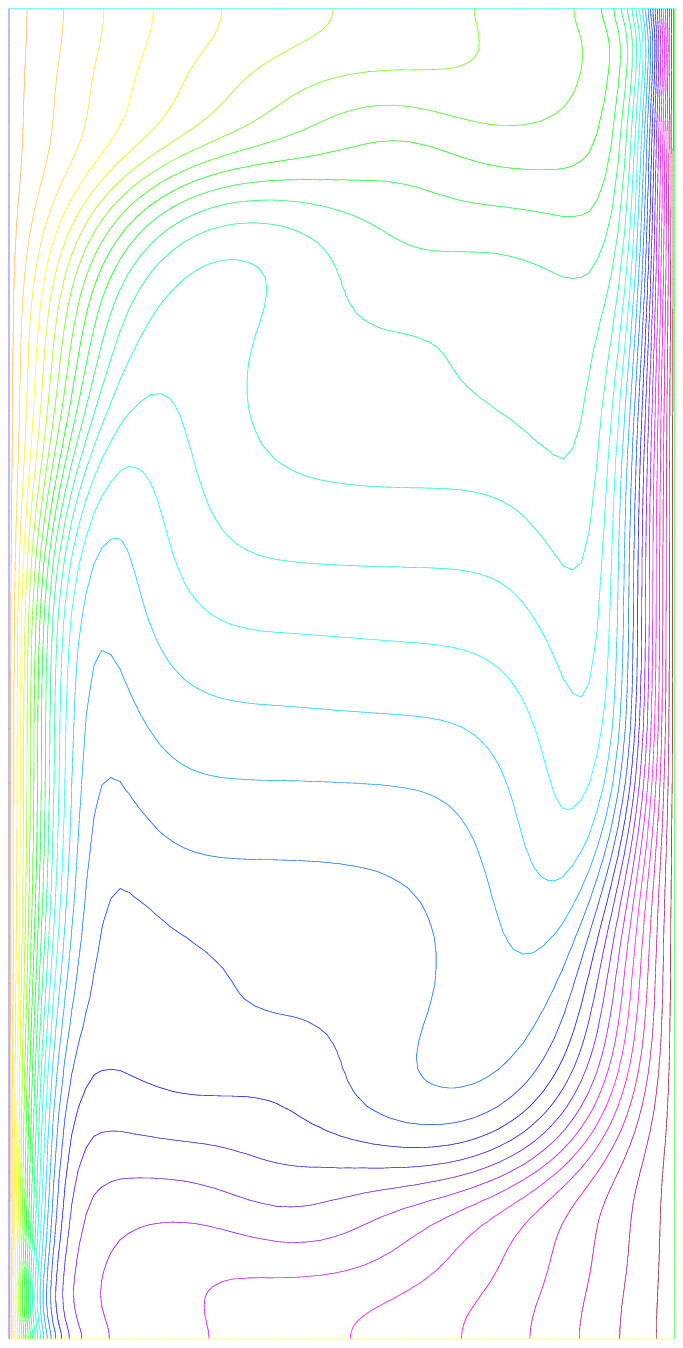,width=0.45\textwidth}
\epsfig{figure=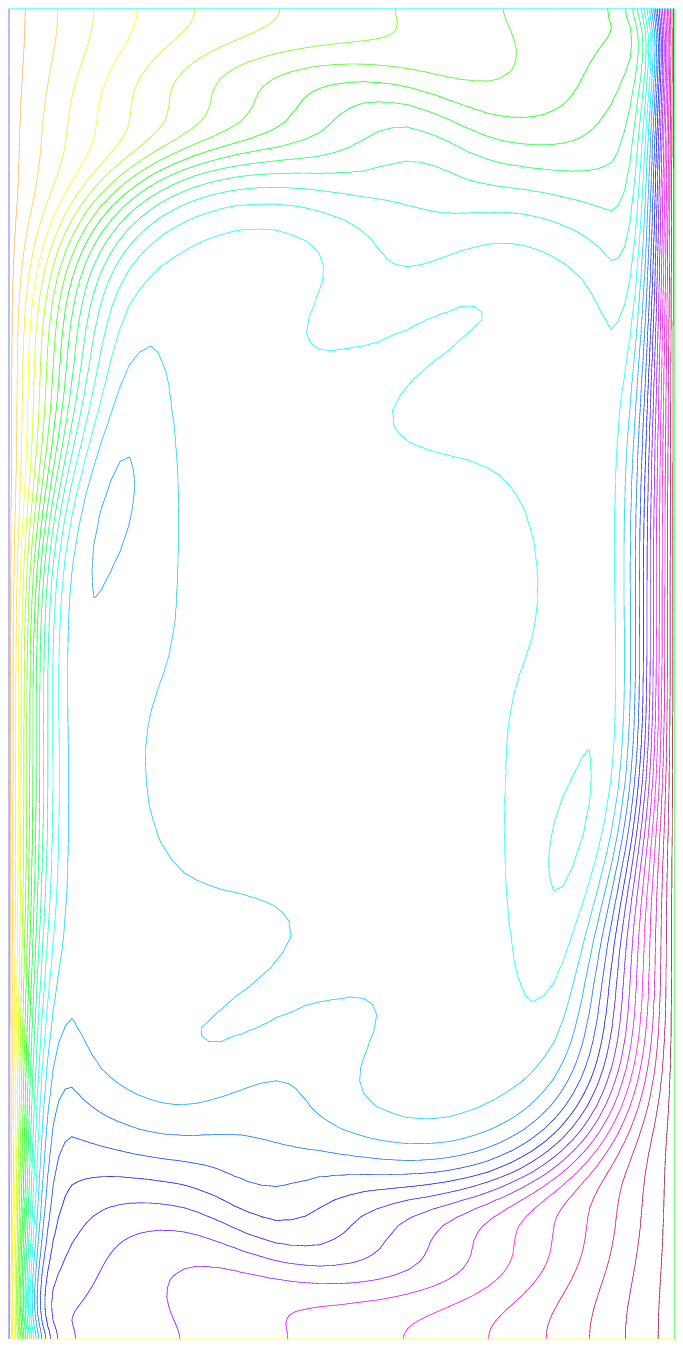,width=0.45\textwidth}
}}
\caption{\label{fig:Ra} Velocity streamlines, Temperature contours and Concentration contours (from left to right) for $Pr=1, Le=2, N=0.8$ with  $Ra=10^4$ (up), $Ra=10^5$ (middle) and $Ra=10^6$ (down) }
\end{figure}

\subsubsection{Thermal and Mass Distributions in Buoyancy Driven Cavity}
In terms of engineering, calculation of thermal and mass distributions along with different boundaries which are kept at different temperature and concentration are of vital importance for convective flows inside enclosures. There are physical parameters called the Nusselt number (Nu) and Sherwood number (Sh) for measuring these distributions. Local and average Nusselt and Sherwood numbers are given with the following formulas
\begin{align*}
\mbox{Nu}_{loc} = \pm\left\{\frac{\partial T}{\partial x}\right\}_{wall},
\mbox{Nu}_{av} =  \int_{\Omega} \mbox{Nu}_{loc} dy.\\
\mbox{Sh}_{loc} = \pm\left\{\frac{\partial S}{\partial x}\right\}_{wall},
\mbox{Sh}_{av} =  \int_{\Omega} \mbox{Sh}_{loc} dy.
\end{align*}
Calculation of $\mbox{Nu}_{av}$ and $\mbox{Sh}_{av}$ at a buoyancy driven cavity test example has been widely used in order to verify and validate proposed numerical schemes on produced codes.  The flow parameters are taken as $Pr=1, Le=2, N=0.8$ for  $Ra=10^4$ and $Ra=10^5$  in this test. Well-known numerical simulations in literature known to obtain such results for a $100 \times 200$ rectangle, which is regarded as a coarse mesh  \cite{chen}.
Table \ref{table:flux1} and Table \ref{table:flux2}  gives a comparison of the results of both presented method and results of \cite{chamka, chen}. As it is seen, acceptable results for $\mbox{Nu}$ and $\mbox{Sh}$ are obtained with the proposed algorithm.
\begin{table}[h!]
    \centering
    \begin{tabular}{|c|c|c|c|c|c|c|}
    \hline
    Ra & Proposed Method & Ref. \cite{chamka}  & Ref. \cite{chen} \\[0.5ex] \hline
    $10^4$  & 3.65(25$\times$40) & 3.67(31$\times$41) & 3.68(100$\times$200) \\[1ex] \hline
     $10^5$  & 6.78(25$\times$40) & 6.82(31$\times$41) & 6.84(100$\times$200) \\[1ex] \hline
    \end{tabular}
    \caption{Comparison of average Nusselt numbers on the vertical boundary of the cavity at $x=0$ (hot wall) for $Pr=1, Le=2, N=0.8$ with mesh size used in computation for varying Rayleigh Numbers}
     \label{table:flux1}
    \end{table}

\begin{table}[h!]
    \centering
    \begin{tabular}{|c|c|c|c|c|c|c|}
    \hline
    Ra & Proposed Method & Ref. \cite{chamka}  & Ref. \cite{chen} \\[0.5ex] \hline
    $10^4$  & 4.78(25$\times$40) & 4.89(31$\times$41) & 4.91(100$\times$200) \\[1ex] \hline
     $10^5$  & 8.75(25$\times$40) & 6.82(31$\times$41) & 8.70(100$\times$200) \\[1ex] \hline
    \end{tabular}
    \caption{Comparison of average Sherwood numbers on the vertical boundary of the cavity at $x=0$ (hot wall) for $Pr=1, Le=2, N=0.8$ with mesh size used in computation for varying Rayleigh Numbers}
     \label{table:flux2}
    \end{table}
\section{Conclusion}
In this paper, we proposed, analyzed and tested a new optimally accurate numerical regularization based on the idea of curvature stabilization for (a family of) second order time-stepping methods for the double-diffusive convection system. Unconditional stability results are derived for velocity, temperature and concentration. Since the method has the advantage of requiring the solution of only one linear system per time step, it is efficient in terms of computational effort. We also give a rigorous proof of the convergence of the method. Several numerical tests were presented to prove the efficiency of the proposed method. The idea of curvature stabilization could be cast on different types of flow problems which would be considered as future studies.
\bibliography{reference}

\providecommand{\bysame}{\leavevmode\hbox to3em{\hrulefill}\thinspace}
\providecommand{\MR}{\relax\ifhmode\unskip\space\fi MR }
\providecommand{\MRhref}[2]{%
  \href{http://www.ams.org/mathscinet-getitem?mr=#1}{#2}
}
\providecommand{\href}[2]{#2}
\begin{thebibliography}{10}

\bibitem{Ada75}
R.A. Adams, \emph{Sobolev spaces}, Academic Press, New York, 1975.

\bibitem{AKR}
M.~Akbas, S.~Kaya, and L.~G. Rebholz, \emph{On the stability at all times of
  linearly extrapolated {BDF2} time-stepping for multiphysics incompressible
  flow problems}, Numer. Methods Partial Differential Equations \textbf{33}
  (2016), 999--1017.

\bibitem{B76}
G.~Baker, \emph{{G}alerkin approximations for the{ Navier}-{Stokes} equations},
  Tech. report, Harvard University.

\bibitem{BS89}
T.~L. Bergman and R.~Srinivasan, \emph{Numerical simulation of {Soret}-induced
  double diffusion in an initially uniform concentration binary liquid}, Int.
  J. Heat Mass Transfer \textbf{32} (1989), 679--687.

\bibitem{BS08}
S.~C. Brenner and L.~R. Scott, \emph{The mathematical theory of finite element
  methods}, Texts in Applied Mathematics, vol.~15, Springer, Berlin, 2008.

\bibitem{chamka}
A.~Chamka and H.~Al-Naser, \emph{Hydromagnetic double-diffusive convection in a
  rectangular enclosure with opposing temperature and concentration
  gradients.}, Int. J. Heat Mass Transfer \textbf{45} (2002), 2465--2483.

\bibitem{chen}
S.~Chen, J.~T\"{o}lke, and M.~Krafczyk, \emph{Numerical investigation of
  double-diffusive (natural) convection in vertical annuluses with opposing
  temperature and concentration gradients.}, International Journal of Heat and
  Fluid Flow \textbf{31} (2010), 217--226.

\bibitem{DP09}
L.~Davis and F.~Pahlevani, \emph{Semi-implicit schemes for transient {Navier-
  Stokes} equations and eddy viscosity models}, Numer. Methods Partial
  Differential Equations \textbf{25} (2009), 212--231.

\bibitem{fat}
F.~G. Eroglu, S.~Kaya, and L.~Rebholz, \emph{Pod-rom for the darcy-{B}rinkman
  equations in double-diffusive convection.}, J. Numer. Math. (2018).

\bibitem{GR79}
V.~Girault and P.~A. Raviart, \emph{Finite element approximation of the
  {Navier}-{Stokes} equations}, Lecture Notes in Mathematics 749,
  Springer-Verlag, Berlin, 1979.

\bibitem{GR}
\bysame, \emph{Finite element methods for the {Navier}-{Stokes} equations
  theory and algorithms}, Springer-Verlag, 1986.

\bibitem{goy}
B.~Goyeau, J.P. Songbe, and D.~Gobin, \emph{Numerical study of double-diffusive
  natural convection in a porous cavity using the {Darcy-Brinkman}
  formulation}, Int. J. Heat Mass Transfer. \textbf{39} (1995), 1363--1378.

\bibitem{GS00}
P.~Gresho and R.~Sani, \emph{Incompressible flow and the finite element method
  vol. 1: Advection-diffusion and isothermal laminar flow}, Wiley., New York,
  2000.

\bibitem{HW02}
E.~Hairer and G.~Wanner, \emph{Solving ordinary differential equations ıı :
  stiff and differential algebraic problems}, Springer-Verlag, 2002.

\bibitem{H03}
Y.~He, \emph{Two-level method based on finite element and {Crank}-{Nicolson}
  extrapolation for the time-dependent {Navier-Stokes} equations}, SIAM J.
  Numer. Anal. \textbf{41} (2003), 1263 -- 1285.

\bibitem{hec}
F.~Hecht, \emph{New development in freefem++.}, J. Numer. Math. \textbf{20}
  (2012), 251--265.

\bibitem{I13}
R.~Ingram, \emph{A new linearly extrapolated { Crank}-{Nicolson} time-stepping
  scheme for the {Navier}-{Stokes} equations}, Math. Comp. \textbf{82} (2013),
  953--1973.

\bibitem{JMR}
N.~Jiang, M.~Mohebujjaman, L.~G. Rebholz, and C.~Trenchea, \emph{An optimally
  accurate discrete regularization for second order time stepping methods for
  {Navier-Stokes} equations}, Comput. Methods Appl. Mech. Engrg. \textbf{310}
  (2016), 388--405.

\bibitem{Joh16}
V.~John, \emph{Finite element methods for incompressible flow problems},
  Springer Ser. Comput. Math., vol.~51, Springer-Verlag, Berlin, 2016.

\bibitem{J12}
A.~D. Jorgenson, \emph{Unconditional stability of a
  {Crank}-{Nicolson}/{Adams}-{Bashforth} 2 implict/explicit method for ordinary
  differential equations}, Master's thesis, University of Pittsburgh, 2012.

\bibitem{L09}
A.~Labovsky, W.~Layton, C.~Manica, M.~Neda, and L.~Rebholz, \emph{The
  stabilized extrapolated trapezoidal finite element method for the
  {Navier-Stokes} equations}, Comput. Methods Appl. Mech. Engrg. \textbf{198}
  (2009), 958--974.

\bibitem{Lay08}
W.~Layton, \emph{Introduction to finite element methods for incompressible,
  viscous flows}, SIAM, 2008.

\bibitem{WT98}
W.~Layton and L.~Tobiska, \emph{A two-level method with backtracking for the
  {Navier}- {Stokes} equations}, SIAM J. Numer. Anal. \textbf{35} (1998),
  2035--2054.

\bibitem{HLLT}
Y.~Li, N.~Hurl, W.~Layton, and C.~Trenchea, \emph{Stability analysis of the
  {Crank-Nicolson-Leapfrog} method with the { Robert-Asselin-Williams} time
  filter}, BIT \textbf{54} (2014), 1--13.

\bibitem{LLT}
Y.~Li, W.~Layton, and C.~Trenchea, \emph{Recent developments in imex methods
  with time filters for systems of evolution equations}, J. Comput. Appl. Math.
  \textbf{299} (2016), 50--67.

\bibitem{LC14}
Y.~Li and C.~Trenchea, \emph{A higher-order {Robert-Asselin} type time filter},
  J. Comput. Phys. \textbf{259} (2014), 23--32.

\bibitem{march}
R.~March, A.~Coutinho, and R.~Elias, \emph{Stabilized finite element simulation
  of double-diffusive natural convection}, Mecanica Computacional \textbf{29}
  (2010), 7985--8000.

\bibitem{NT15}
J.~Nichele and D.~A. Teixeira, \emph{Evaluation of {D}arcy-{B}rinkman equation
  for simulations of oil flows in rocks}, J. Petrol Sci. Eng. \textbf{134}
  (2015), 76--78.

\bibitem{R12}
S.~S. Ravindran, \emph{Convergence of extrapolated {BDF2} finite element
  schemes for unsteady penetrative convection model}, Numer. Funct. Anal. and
  Optim. \textbf{33:1} (2012), 48--79.

\bibitem{R15}
\bysame, \emph{An analysis of the blended three-step backward differentiation
  formula time-stepping scheme for the{N}avier-{S}tokes-{T}ype system related
  to soret convection}, Numer. Funct. Anal. and Optim. \textbf{36} (2015),
  658--686.

\bibitem{SG14}
J.~Serrano-Arellano, M.~Gij�n-Rivera, J.M. Riesco-�vila, and
  F.~Elizalde-Blancas, \emph{Numerical study of the double diffusive convection
  phenomena in a closed cavity with internal {CO2} point sources}, Int. J. Heat
  Mass Transfer \textbf{71} (2014), 664--674.

\bibitem{C14}
C.~Trenchea, \emph{Stability of partitioned imex methods for systems of
  evolution equations with skew-symmetric coupling}, ROMAI J. \textbf{10}
  (2014), 175--189.

\bibitem{C16}
\bysame, \emph{Second order implicit for local effects and explicit for
  nonlocal effects is unconditionally stable}, ROMAI J. \textbf{1} (2016),
  163--178.

\bibitem{WL97}
E.~Weinan and J.-G. Liu, \emph{Simple finite element method in vorticity
  formulation for incompressible flows}, Math. Comp. \textbf{70} (1997),
  579--593.

\bibitem{W9}
P.~D. Williams, \emph{A proposed modification to the {Robert-Asselin} time
  filter}, Mon. Weather Rev. \textbf{137} (2009), 175--189.

\bibitem{XPT}
S.H. Xin, P.L.~Qu\' er\' e, and L.S. Tuckerman, \emph{Bifurcation analysis of
  double-diffusive convection with opposing horizontal thermal and solutal
  gradients}, Phys. Fluids \textbf{10} (1998), 85--858.

\bibitem{ZXZ}
C.~Xu, Y.~Zhang, and J.~Zhou, \emph{Convergence of a linearly extrapolated {
  BDF2 } finite element scheme for viscoelastic fluid flow}, Bound. Value
  Probl. \textbf{140} (2017).

\bibitem{YJ}
Y.~Yang and Y.L. Jiang, \emph{Numerical analysis and computation of a type of
  imex method for the time-dependent natural convection problem}, Comput. Meth.
  Appl. Math. \textbf{16} (2016), 321--344.

\end{thebibliography}
\bibliographystyle{amsplain}
\end{document}